\documentclass[11pt,xcolor=dvipsnames]{article}
\pdfoutput=1 
\usepackage[utf8x]{inputenc}
\usepackage[english]{babel}
\usepackage{graphicx}           
\usepackage{url}
\usepackage{eurosym}
\usepackage{makeidx}
\usepackage[plain]{algorithm}
\usepackage{algorithmic} 
\usepackage{color} 
\usepackage[margin=3.3cm]{geometry}
\usepackage{setspace}
\usepackage{datetime}
\usepackage{pifont}
\usepackage{tikz}
\usepackage{natbib}
\usepackage{hyperref}
\usepackage{amsmath,amsfonts,amssymb,latexsym,stmaryrd}

\newcommand{\eqdef}     {\stackrel{{\textrm{\rm\tiny def}}}{=}}

\newtheorem{theorem}      {Theorem}[section]
\newtheorem{theorem*}     {theorem}
\newtheorem{proposition}  [theorem]{Proposition}

\newtheorem{lemma}        [theorem]{Lemma}
\newtheorem{corollary}    [theorem]{Corollary}

\newtheorem{remark}       [theorem]{Remark}

\newcommand{\proof}        {\paragraph{Proof}}

\newsavebox{\fmbox}

\newcommand{\mmax}{{\textrm{max}}}

\newcommand{\M}        {\mathbb M}
\newcommand{\N}        {\mathbb N}
\newcommand{\D}        {\mathcal D}

\newcommand{\C}        {\mathbb C}

\newcommand{\E}        {\mathbb E}
\newcommand{\X}        {\mathbb X}

\newcommand{\R}      {\mathbb R}

\renewcommand{\P}      {\mathbb P}

\newcommand{\cv}        {\mathop{\;{\rightarrow}\;}}

\newcommand{\norm}   [1] {\left\Vert #1 \right\Vert}
\newcommand{\normtv}   [1] {\left\Vert #1 \right\Vert_{\textrm{\tiny TV}}}

\newcommand{\crochet}[1] {\langle #1 \rangle}

\newcommand{\ot}        {\leftarrow}
\newcommand{\carre}     {\hfill$\Box$}

\newcommand{\BB}   {{\mathcal B}}
\newcommand{\CC}   {{\mathcal C}}
\newcommand{\DD}   {{\mathcal D}}
\newcommand{\EE}   {{\mathcal E}}

\newcommand{\NN}   {{\mathcal N}}
\newcommand{\MM}   {{\mathcal M}}

\newcommand{\PP}   {{\mathcal P}}

\newcommand{\RR}   {{\mathcal R}}

\newcommand{\XX}   {{\mathcal X}}

\newcommand{\rmd}   {{{\textrm{\upshape d}}}}

\newcommand{\dontforget}[1]
{{\mbox{}\\\noindent\rule{1cm}{2mm}\hfill  #1 \hfill\rule{1cm}{2mm}}\typeout{---------- #1 ------------}}

\renewcommand{\epsilon}{\varepsilon}

\def\dobm{
    \copy1\kern-\wd1\kern0.05ex\copy1\kern-\wd1\kern0.05ex\box1}

\newcommand{\fenumi}  {\textrm{\rm({\textit{i}}\/)}}
\newcommand{\fenumii} {\textrm{\rm({\textit{ii}}\/)}}

\newcommand
      {\sysdys}
      {{\sf S\kern-.15em\raise.3ex\hbox{Y}\kern-.15em
            SD\kern-.15em\raise.3ex\hbox{Y}\kern-.15emS}}

\newcommand{\normm}[1]{{\vert\kern-0.25ex\vert\kern-0.25ex\vert #1 
    \vert\kern-0.25ex\vert\kern-0.25ex\vert}}
\renewcommand{\X}{\mathcal{X}}
\renewcommand{\L}{\mathcal{L}}

\newcommand{\F}{\mathcal{F}}
\renewcommand{\C}{\mathcal{C}}

\renewcommand{\M}{\mathcal{M}}
\renewcommand{\EE}{\mathbb{E}}
\renewcommand{\PP}{\mathbb{P}}
\renewcommand{\RR}{\mathbb{R}}
\renewcommand{\R}{\mathbb{R}}
\renewcommand{\NN}{\mathbb{N}}

\newcommand{\dif}{\mathrm{d}}

\newcommand{\ds}{d_{\textrm{\tiny\rm S}}}
\newcommand{\dtv}{d_{\textrm{\tiny\rm TV}}}
\newcommand{\dpr}{d_{\textrm{\tiny\rm PR}}}

\newcommand{\tauw}{\tau_{\textrm{\tiny\rm w}}}

\newcommand{\nub}{\bar\nu}

\newcommand{\lambdab}{\bar\lambda}
\newcommand{\Exp}{{\textrm{\rm Exp}}}
\newcommand{\tmax}{t_{\textrm{\tiny\rm max}}}
\newcommand{\rmax}{r_{\textrm{\tiny\rm max}}}

\newcommand{\Sin}{{\mathbf s}_{\textrm{\tiny\rm in}}}
\renewcommand{\mmax}{m_{\textrm{\tiny\rm max}}}
\newcommand{\mdiv}{m_{\textrm{\tiny\rm div}}}

\newcommand{\rhos}{\rho_{\textrm{\tiny\rm\!\! s}}}
\newcommand{\rhog}{\rho_{\textrm{\tiny\rm\!\! g}}}

\renewcommand{\eqdef}     {\stackrel{{\textrm{\rm\tiny def}}}{=}}

\newcommand{\norme}[1]{\left\Vert #1 \right\Vert }

\newcommand{\umoins}{{\smash{u^{\raisebox{-1pt}{\scriptsize\scalebox{0.5}{$-$}}}}}}

\usepackage{ifthen}    
\newboolean{showComments}        
\setboolean{showComments}{true}  
\newcommand{\fnote}[1]
    {{\mbox{}\\\noindent\color{red}
    \rule{1cm}{2mm}\hfill  #1 \hfill\rule{1cm}{2mm}}
    \typeout{---------- #1 ------------}}

\renewcommand{\dontforget}[1]
      {\ifthenelse {\boolean{showComments}} {{\color{red}(#1)}} {}}
\newcommand{\Dontforget}[1]
      {\ifthenelse {\boolean{showComments}} {\fnote{#1}} {}}

\begin{document}

\title{A mass-structured individual-based model of the chemostat:\\
convergence and simulation}
\author{Fabien Campillo\thanks{INRIA\,, \texttt{Fabien.Campillo@inria.fr}} \and Coralie Fritsch\thanks{Montpellier 2 University and INRA/MIA\,, \texttt{Coralie.Fritsch@supagro.inra.fr}\protect\\ Fabien Campillo and Coralie Fritsch are members of the MODEMIC joint INRA and INRIA project-team.
MODEMIC Project-Team, INRA/INRIA, UMR MISTEA, 2 place Pierre Viala,
34060 Montpellier cedex 01, France.}}

\maketitle

\begin{abstract}
We propose a model of chemostat where the bacterial population is individually-based, each bacterium is explicitly represented and has a mass evolving continuously over time. The substrate concentration is represented as a conventional ordinary differential equation. These two components are coupled with the bacterial consumption. Mechanisms acting on the bacteria are explicitly described (growth, division and up-take). Bacteria interact via consumption. We set the exact Monte Carlo simulation algorithm of this model and its mathematical representation as a stochastic process. We prove the convergence of this process to the solution of an integro-differential equation when the population size tends to infinity. Finally, we propose several numerical simulations.

\paragraph{Mathematics Subject Classification (MSC2010):}
 60J80, 60J85 (primary); 37N25, 92D25 (secondary).

\end{abstract}

\section{Introduction}
\label{sec.intro}

The chemostat is a biotechnological process of continuous culture developed in the 50s \citep{monod1950a,novick1950a} and which is at the heart of several industrial applications as well as laboratory devices \citep{smith1995a}. Bioreactors operating under this mode are maintained under perfect mixing conditions and usually at large bacterial population sizes.

These features allow such processes to be modeled by ordinary (deterministic) differential systems since, in large populations and under certain conditions, demographic randomness can be neglected. Moreover, perfect mixing conditions permit us to neglect the spatial distribution and express these models in terms of mean concentration in the chemostat. In its simplest version, the chemostat model is expressed as a system of two coupled ordinary differential equations respectively for biomass and substrate concentrations \citep{smith1995a}. This approach extends to the case of several bacterial species and several substrates. The simplicity of such models makes possible the development of efficient tools for automatic control and the improvement of the associated biotechnological processes. 
However, it is increasingly necessary to develop models beyond the standard assumption of perfect mixing with a bacterial population possessing uniform characteristics. For this purpose, several paths are available which take into account the different sources of randomness or the structuring of the bacterial population and its discrete nature. All these aspects have been somewhat neglected in previous models.

In addition, the  recent development of so-called ``omics'' approaches such as genomics and large-scale DNA sequencing technology are the basis of a renewed interest in chemostat techniques \citep{hoskisson2005a}. These bacterial cultures may also be considered as laboratory models to study selection phenomena and evolution in bacterial ecosystems.

Beyond classical models based on systems of (deterministic) ordinary differential equations (ODE) which neglect any structuring of the bacterial populations, have also appeared in the 60's and 70's bacterial growth models structured  in size or mass based on integro-differential equations (IDE) \citep{fredrickson1967a,ramkrishna1979a}, ​​see also the monograph \cite{ramkrishna2000a} on these so-called population balance equations for  growth-frag\-men\-tation models.

Various research papers have been devoted to the stochastic modeling of the chemostat.
\cite{crump1979a} propose a model of pure jump for the biomass growth coupled with a differential equation for the substrate evolution. \cite{stephanopoulos1979a} propose a model with randomly fluctuating dilution rate, hence the noise is rather of an environmental nature, whereas in the previous model it is rather demographic. \cite{grasman2005a} propose a chemostat model at three trophic levels where randomness appears only in the upper trophic level. \cite{imhof2005a} also propose a  stochastic chemostat model but, as in previous models, the noise is simply ``added'' to the classical deterministic model. In contrast, in \cite{campillo2011chemostat} article the demographic noise emerges from a description of the dynamics at the microscopic level.

In recent years, many models for the evolution in chemostats have been proposed either using integro-differential equations \citep{diekmann-odo2005a,mirrahimi2012a,mirrahimi2012b} or individual-based models (IBM) \citep{champagnat2013a}.

There are also computer models, for example \cite{lee-min-woo2009a} propose an IBM  structured in mass for a `` batch'' culture process. In this model, as in \citep{champagnat2013a} the dynamics of the substrate is described by deterministic differential equations. 
Indeed, the difference in scale between a bacterial cell and a substrate molecule guaranties that, at the scale of the bacterial population, the dynamics of the substrate can be correctly represented by the fluid limit model, while the dynamics of the bacterial population is discrete and random.

We focus here on an individual-based model (IBM) of the chemostat. In contrast to deterministic models with continuous variables, in IBMs all variables, or at least some of them, are stochastic and discrete. These models are generally cumbersome in terms of simulation and difficult to analyze mathematically, but they can be useful in accounting for phenomena inaccessible in earlier models. The majority of IBMs are described initially in natural languages with simple rules. From there they are described as  computer models in terms of algorithms, it is this approach that is often termed IBM.
Nonetheless, they can also be described mathematically using a Markov process.
The advantage of this approach is to allow the mathematical analysis of the IBM.
In particular, as we will see here,  the convergence of the IBM to an integro-differential model can be demonstrated.
 The latter approach has been developed in a series of papers: for a simple model of position \citep{fournier2004a}, for the evolution of trait structured population \citep{champagnat2006b}, which is then extended to take into account the age of individuals \citep{vietchitran2006a,vietchitran2008a}. More recently \citet{champagnat2013a} proposed a chemostat model with multiple resources where the bacterial population has a genetic trait subject to evolution.

In the context of a growth-fragmentation model, \cite{hatzis1995a}  proposed an IBM, without  substrate variables, and draw a parallel between this model and an integro-differential model.

\medskip

In Section \ref{sec.model} we introduce the IBM where each individual in the bacterial population is explicitly represented by its mass. We describe the phenomenon which the model will take into account at a microscopic scale:   
individual cell growth,  cell division, up-take 
(substrate and bacteria are constantly withdrawn from the chemostat vessel), as well as the individual consumption described as a coupling with the ordinary differential equation which models the dynamics of the substrate. Then we describe the associated exact Monte Carlo algorithm, noting that this algorithm is asynchronous in time, i.e. different events occur at random instants which are not predetermined.

In Section \ref{sec.notations} we  introduce some notation, then in Section \ref{sec.processus.microscopique} we construct the stochastic process associated with the IBM as a Markov process with values in the space of finite measures over the state-space of masses.

In Section \ref{sec.convergence} we prove the convergence, in large population limit, of the IBM towards an integro-differential equation of the population-balance equation type \citep{fredrickson1967a,ramkrishna1979a,ramkrishna2000a}  coupled with an equation for the dynamics of the substrate. Finally in Section \ref{sec.simulations}  we present several numerical simulations.

\section{The model}
\label{sec.model}

\subsection{Description of the dynamics}

We consider an \emph{individual-based model (IBM) structured in mass} where the
bacterial population is represented as individuals growing in a perfectly mixed vessel of volume $V$ (l). Each individual is solely characterized by its mass $x\in\X \eqdef [0,\mmax]$, this model does not take into account spatialization. At time $t$ the system is characterized by the pair:
\begin{align}
\label{eq.xi}
  (S_{t},\nu_{t})
\end{align}
where
\begin{enumerate}

\item $S_{t}$ is the \emph{substrate concentration} (mg/l) which is assumed to be uniform in the vessel;

\item $\nu_{t}$ is the \emph{bacterial population}, that is $N_{t}$ individuals
and the mass of the individual number $i$ will be denoted  $x^i_{t}$ (mg)  
for $i=1,\dots,N_{t}$. 
It will be convenient to represent the population $\{x^i_{t}\}_{i=1,\dots,N_{t}}$ at time $t$ as the following punctual measure:
\begin{align}
\label{eq.nu}
  \nu_t(\rmd x)=\sum_{i=1}^{N_t}\delta_{x_t^i}(\rmd x)\,.
\end{align}
\end{enumerate}

The dynamics of the chemostat combines \emph{discrete evolutions}, cell division and bacterial up-take, as well as \emph{continuous evolutions}, the growth of each individual and the dynamics of the substrate. We now describe the four components of the dynamics, first the discrete ones and then the continuous ones which occur between the discrete ones.

\begin{enumerate}
\vskip0.8em
\item{\textbf{Cell division}} --
\emph{Each individual of mass $x$ divides at rate $\lambda(s,x)$ into two
individuals of respective masses  $\alpha\,x$ and $(1-\alpha)\,x$:
\begin{center}
\includegraphics[width=5cm]{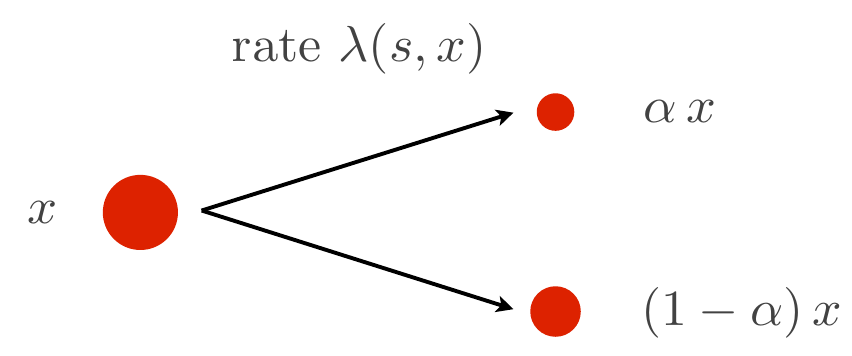}
\end{center}
where $\alpha$ is distributed according to a given probability distribution  $Q(\rmd\alpha)$ on $[0,1]$,
and  $s$ is the substrate concentration.} 

\medskip

For instance, the function $\lambda(s,x)$ does not depend on  the substrate concentration $s$ and could be of the following form which will be used in the simulation presented in Section \ref{sec.simulations}:
\begin{center}
\includegraphics[width=4cm]{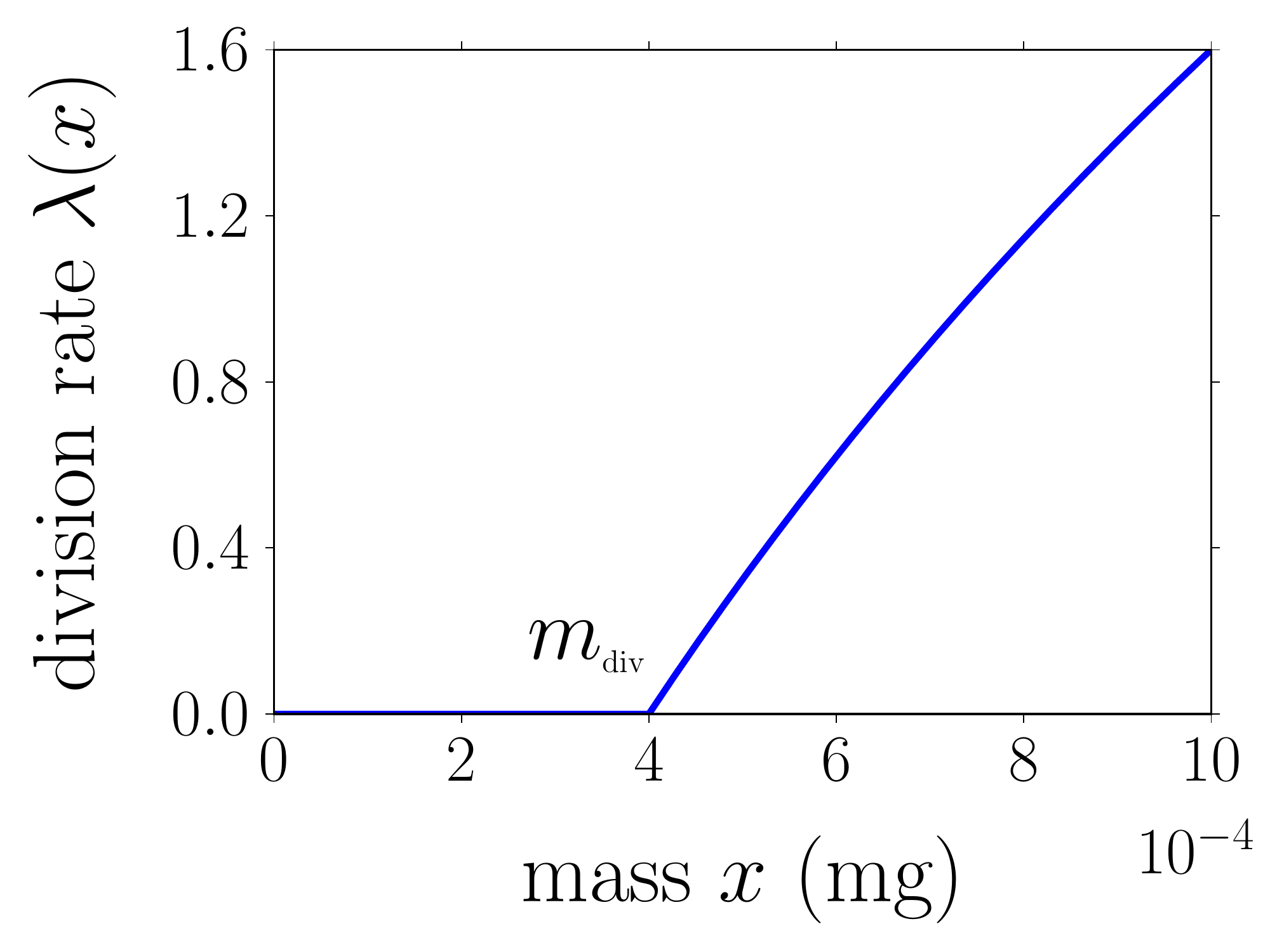}
\end{center} 
Thus, below a certain mass $\mdiv$ it is assumed that the cell cannot divide. There are models where the rate also depends on the concentration $s$, see for example  \citep{daoutidis2002a,henson2003b}.

\medskip

We suppose that the distribution $Q(\rmd\alpha)$ is symmetric with respect to $\frac{1}{2}$, i.e. $Q(\rmd\alpha)=Q(1-\rmd \alpha)$. It also may admit a density  $Q(\rmd \alpha)=q(\alpha)\,\rmd \alpha$ with the same symmetry:
\begin{center}
\includegraphics[width=4cm]{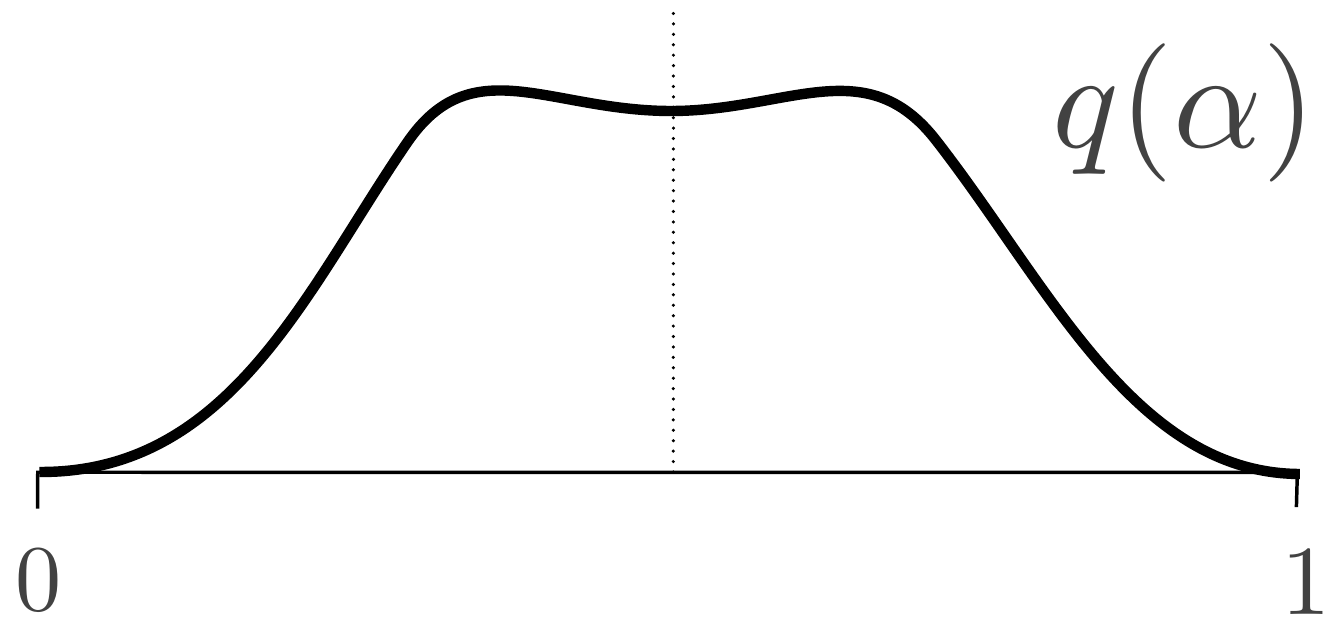}
\end{center}

Thus, the  division kernel of an individual of mass $x$ is $ K(x,\rmd y) = Q(\frac{1}{x}\,\rmd y)$ with support $[0,x]$. In the case of perfect mitosis, an individual of mass $x$ is divided into two individuals of masses $\frac x 2$ and then $Q(\rmd \alpha)=\delta_{1/2}(\rmd \alpha)$. 

\medskip

\emph{It is therefore assumed that, relative to their mass, the division kernel is the same for all individuals. This allows us to reduce the model to a single division kernel. More complex scenarios can also be investigated.}

\vskip0.8em
\item{\textbf{Up-take}} --
\emph{Each individual is withdrawn from the chemostat at rate~$D$.}
One places oneself in the framework of a \emph{perfect mixing} hypothesis, where individuals are uniformly distributed in the volume $V$ independently from their mass. During a time step $\delta$, a total volume of $D\,V\,\delta$ is withdrawn from the chemostat:
\begin{center}
\includegraphics[width=10cm]{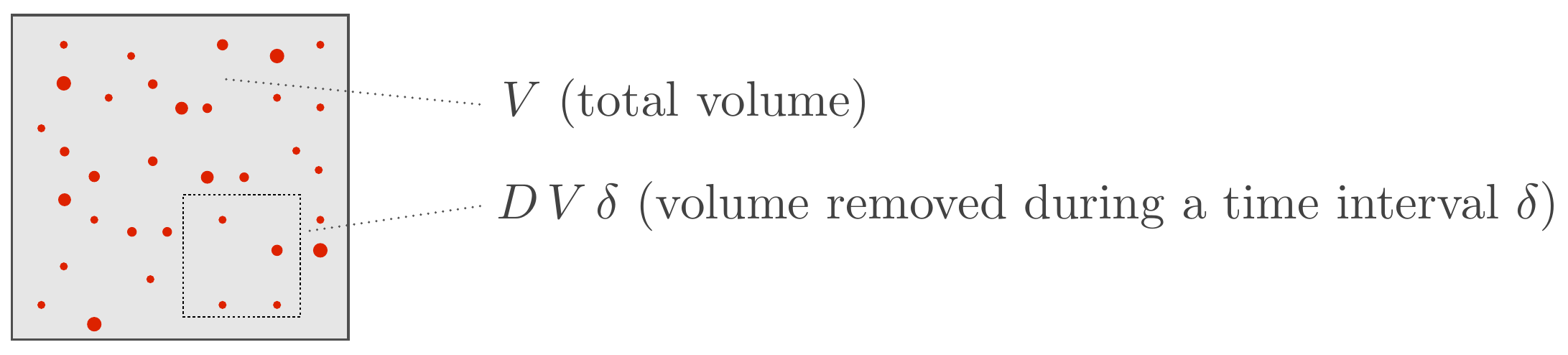}
\end{center}
and therefore, if we assume that all individuals have the same volume considered as negligible, during this time interval $\delta$, an individual has a probability $D\,\delta$ to be withdrawn from the chemostat, $D$ is the dilution rate. This rate could possibly depend on the mass of the individual.

\end{enumerate}
When the division of an individual occurs, the size of the population instantaneously jumps   from $N_{t}$ to $N_{t}+1$; when an individual is withdrawn from the vessel, the size of the population jumps instantaneously from $N_{t}$ to $N_{t}-1$; between each discrete event the size $N_{t}$ remains constant and the chemostat evolves according to the following two continuous mechanisms:
\begin{enumerate}
\setcounter{enumi}{2}
\vskip0.8em
\item {\textbf{Growth of each individual}} --
\emph{Each individual of mass $x$ growths at speed $\rhog(S_{t},x)$}:
\begin{align}
\label{eq.masses}
   \dot x^i_{t} 
   =
   \rhog(S_{t},x^i_{t})\,, \quad i=1,\dots,N_{t}
\end{align}
where $\rhog:\R^2_{+}\mapsto\R_{+}$ is given. For the simulation we will consider the following Gompertz model:
\begin{align*}
   \rhog(s,x) \eqdef r(s)\,\log\Big(\frac{\mmax}{x}\Big)\,x
\end{align*}
where the growth rate $r(s)$ depends on the substrate concentration according to the Monod kinetics:
\begin{align*}
    r(s)=\rmax\,\frac{s}{k_r+s}
\end{align*}
here $\mmax$ is the maximum weight that an individual can reach. In Section
\ref{subsec.modeles.deterministes} we also present an example of a function $\rhog(s,x)$ linear in $x$ which will lead to the classical model of chemostat.

\vskip0.8em
\item{\textbf{Dynamic of the substrate concentration}} -- 
\emph{The substrate concentration evolves according to the ordinary differential equation:
\begin{align}
\label{eq.substrat}
  \dot S_t = \rhos(S_{t},\nu_{t})
\end{align}}
where
\begin{align*}
  \rhos (s,\nu)
  &\eqdef
  D(\Sin-s)-k\, \mu(s,\nu)\,,
\\
  \mu(s,\nu)
  &\eqdef
  \frac{1}{V} \int_{\XX} \rhog(s,x)\,\nu(\rmd x)
  =    
  \frac{1}{V} \sum_{i=1}^{N} \rhog(s,x^i)
\end{align*}
with $\nu=\sum_{i=1}^N \delta_{x^i}$;
$D$ is the dilution rate (1/h), 
$\Sin$ is the input concentration (mg/l),
$k$ is the stoichiometric coefficient (inverse of the yield coefficient),
and  $V$ is the representative volume (l).
Mass balance leads to Equation \eqref{eq.substrat} and the initial condition $S_{0}$ may be random.
\end{enumerate}

To ensure the existence and uniqueness of solutions of the ordinary differential equations \eqref{eq.masses} and \eqref{eq.substrat}, we assume
that application $\rhog(s,x)$ is Lipschitz continuous w.r.t. $s$ uniformly in $x$:
\begin{align}
\label{hyp.rhog.lipschitz}
  \bigl|\rhog(s_1,x)-\rhog(s_2,x)\bigr|
  \leq 
  k_g \,|s_1-s_2|
\end{align}
for all $s_1,\,s_2 \geq 0$ and all $x\in\XX$. It is further assumed that:
\begin{align}
\label{hyp.rhog.borne}
  0
  \leq
  \rhog(s,x)
  \leq 
  \bar g
\end{align}
for all $(s,x)\in\RR_+ \times \X$, and that 
 in the absence of substrate the bacteria do not grow:
\begin{align}
\label{hyp.rhog.nulle.en.0}
  \rhog(0,x)
  =
  0
\end{align}
for all $x\in \X$. To ensure that the mass of a bacterium  
stays between $0$ and $\mmax$, it is finally assumed that:
\begin{align}
\label{hyp.rhog.nulle.en.x=mmax}
  \rhog(s,\mmax)
  =
  0
\end{align}
for any $s\geq 0$.

\subsection{Algorithm}

\begin{algorithm}
\begin{center}
\begin{minipage}{14cm}
\begin{algorithmic}
\STATE $t\ot 0$
\STATE sample $(S_0,\nu_0=\sum_{i=1}^{N_{0}}\delta_{x^i_{t}})$
\WHILE {$t\leq \tmax$}
  \STATE $N \ot \crochet{\nu_{t},1}$
  \STATE $\tau \ot (\bar\lambda+D)\,N$
  \STATE $\Delta t \sim \Exp(\tau)$
  \STATE integrate the equations for the mass \eqref{eq.masses} 
    and the substrate \eqref{eq.substrat} over $[t,t+\Delta t]$
  \STATE $t \ot t+\Delta t$
  \STATE draw $x$ uniformly in $\{x^i_{t}\,;\,i=1,\dots,N_{t}\}$  
  \STATE $u\sim U[0,1]$ 
  \IF {$u\leq \lambda(S_{t},x)/(\lambdab+D)$}
      \STATE $\alpha \sim Q$
      \STATE $\nu_{t} \ot \nu_{t} -\delta_{x}+\delta_{\alpha\,x}
              +\delta_{(1-\alpha)\,x}$
      \COMMENT{division}
  \ELSIF{$u\leq (\lambda(S_{t},x)+D)/(\lambdab+D)$}
      \STATE $\nu_{t} \ot \nu_{t} -\delta_{x}$
      \COMMENT{up-take}
  \ENDIF
\ENDWHILE
\end{algorithmic}
\end{minipage}
\end{center}
\caption{\itshape ``Exact'' Monte Carlo  simulation of the individual-based model:
approximations only lie in the numerical integration of the ODEs and in 
the pseudo-random numbers generators.}
\label{algo.ibm}
\end{algorithm}

In the model described above, the division rate $\lambda(s,x)$ depends on the concentration of substrate $s$ and on the mass  $x$ of each individual which continuously evolves according to the system of coupled ordinary differential equations \eqref{eq.masses} and \eqref{eq.substrat}, so to simulate the division of the cell we make use of a rejection sampling technique. It is assumed that there exists $\lambdab<\infty$ such that:
\[
  \lambda(s,x)\leq \lambdab
\]
hence an upper bound for the rate of event, division and up-take combined, at the population level is given by:
\[
    \tau \eqdef (\bar\lambda+D)\,N\,.
\]
At time $t+\Delta t$ with $\Delta t\sim \Exp(\tau)$, we determine if an event has occurred and what is its type by acceptance/rejection.
To this end, the masses of the $N$ individuals and the substrate concentration evolve according to the coupled ODEs \eqref{eq.masses} and \eqref{eq.substrat}.
Then we chose uniformly at random an individual within the population $\nu_{(t+\Delta t)^-}$, that is the population at time  $t+\Delta t$ before any possible event, let $x_{(t+\Delta t)^-}$ denotes its mass, then:
\begin{enumerate}

\item With probability:
\[
   \frac{\bar\lambda}{(\bar\lambda+D)}
\]
we determine if there has been division by acceptance/rejection:
\begin{itemize}
\item division occurs, that is:
\begin{align}
\label{eq.event.division}
   \nu_{t+\Delta t}
   =
   \nu_{(t+\Delta t)^-}
   -\delta_{x_{(t+\Delta t)^-}}  
   +\delta_{\alpha\,x_{(t+\Delta t)^-}}  
   +\delta_{(1-\alpha)\,x_{(t+\Delta t)^-}}  
   \qquad
   \textrm{with }\alpha\sim Q
\end{align}
with probability $\lambda(S_{t},x_{(t+\Delta t)^-})/\bar\lambda$;
\item 
no event occurs with probability $1-\lambda(S_{t},x_{(t+\Delta t)^-})/\bar\lambda$.
\end{itemize}
In conclusion, the event \eqref{eq.event.division} occurs with probability:
\[
  \frac{\lambda\bigl(S_{t},x_{(t+\Delta t)^-}\bigr)}{\bar\lambda}
  \,\frac{\bar\lambda}{(\bar\lambda+D)} 
  = 
  \frac{\lambda\bigl(S_{t},x_{(t+\Delta t)^-}\bigr)}{(\bar\lambda+D)}\,.
\]

\item With probability:
\[
   \frac{D}{(\bar\lambda+D)}
   =
   1-\frac{\bar\lambda}{(\bar\lambda+D)}
\]
the individual is withdrawn, that is:
\begin{align}
\label{eq.event.soutirage}
   \nu_{t+\Delta t}
   =
   \nu_{(t+\Delta t)^-}
   -\delta_{x_{(t+\Delta t)^-}}  
\end{align}
\end{enumerate}
Finally, the events and the associated probabilities are:
\begin{itemize}
\item 
division \eqref{eq.event.division} with probability
$\lambda(S_{t},x_{(t+\Delta t)^-})/(\bar\lambda+D)$,
\item
up-take \eqref{eq.event.soutirage} with probability 
${D}/{(\bar\lambda+D)}$
\end{itemize}
and no event (rejection) with the remaining probability.
The details are given in Algorithm~\ref{algo.ibm}.

Technically, the numbering of individuals is as follows: at the initial time individuals are numbered from $1$ to $N$, in case division the daughter cell $\alpha \,x$ keeps the index of the parent cell and the daughter cell $(1-\alpha)\,x$ takes the index $N+1$; in case of the up-take, the individual $N$ acquires  the index of the withdrawn cell.

\section{Notations}
\label{sec.notations}

Before proposing an explicit mathematical description of the process $(\nu_{t})_{t\geq 0}$ we introduce some notations.

\subsection{Punctual measures}

Notation \eqref{eq.nu} designating the bacterial population seems somewhat abstract but it will bridge the gap between the ``discrete'' -- counting punctual measures --
and the ``continuous'' -- continuous measures of the population densities --  in the context of the asymptotic large population analysis. Indeed for any measure $\nu(\rmd x)$ defined on $\R_{+}$ and any function $\varphi:\R_{+}\mapsto\R$, we define:
\[
  \crochet{\nu,\varphi}
  \eqdef
  \int_{\R_{+}}\varphi(x)\,\nu(\rmd x)\,.
\]
This notation is valid for continuous measures as well as for punctual measures $\nu_{t}(\rmd x)$ defined by \eqref{eq.nu}, in the latter case 
$\crochet{\nu_{t},\varphi} = \sum_{i=1}^{N_{t}} \varphi(x^i_{t})$.

Practically, this notation allows us to link to macroscopic quantities, e.g. 
at time $t$ the \emph{population size} is:
\[
  N_{t} = \crochet{\nu_{t},1}
\]
and the  \emph{total biomass} is:
\[
  X_{t} \eqdef \crochet{\nu_{t},I}
  = \sum_{i=1}^{N_{t}} x^i_{t}
\]
where $1(x)\equiv 1$ and $I(x)\equiv x$. Finally:
\[
   x\in\nu_{t}=\sum_{i=1}^{N_{t}}\delta_{x^i_{t}}(\rmd x)
\]
will denote any individual among $\{x^1_{t},\dots,x^{N_{t}}_{t}\}$.

The set of finite and positive measures on $\X$ is denoted  $\M_{F}(\XX)$,
and $\M(\XX)$ is the subset of punctual finite measures on $\X$:
\[
  \M(\XX)
  \eqdef
  \left\{
    	\sum_{i=1}^N \delta_{x^i} \;; \; N \in \NN, \, x^i \in \X 
  \right\}
\]
where by convention $\sum_{i=1}^0\delta_{x^i}$ is the null measure.

\subsection{Growth flow}

Let:
\[
\begin{array}{rrcl}
   A_{t} &: 
     \R_{+}\times\MM(\XX)  
     &\xrightarrow[]{\mbox{}\hskip1em} & 
     \R_{+}\times\MM(\XX)
     \\
   &(s,\nu)&\xrightarrow[]{\mbox{}\hskip1em} &
   A_{t} (s,\nu)
\end{array}
\]
be the differential flow associated with the couple system of ODEs \eqref{eq.substrat}--\eqref{eq.masses} apart from any event (division or up-take), i.e.:
\begin{align}
\label{eq.flot.A}
   A_{t}(s,\nu)
   =
   \Biggl(
   		A^0_{t}(s,\nu) \;,\; \sum_{i=1}^N\delta_{A^i_{t}(s,\nu)}
   \Biggr)
   \quad\textrm{with }\nu=\sum_{i=1}^N\delta_{x^i}
   \,.
\end{align}
where $A^0_{t}(s,\nu)$ and $(A^i_{t}(s,\nu)\,;\,i=1,\dots,N)$ are the coupled solutions of \eqref{eq.substrat}--\eqref{eq.masses} taken at time $t$ from the initial condition $(s,\nu)$, that is:
\begin{align*}
  \frac{\rmd}{\rmd t} 
   A_{t}^0(s,\nu)
   &= \rhos\Bigl(A_{t}^0(s,\nu),\sum_{i=1}^N \delta_{A_{t}^i(s,\nu)}\Bigr)
\\
   &= D\,(\Sin-A_{t}^0(s,\nu))
      -\frac{k}{V} \sum_{i=1}^N \rhog(A_{t}^0(s,\nu),A_{t}^i(s,\nu))\,,
   &
   A_{0}^0(s,\nu)=s\,,
\\[0.5em]
  \frac{\rmd}{\rmd t} 
   A_{t}^i(s,\nu)
   &=
   \rhog(A_{t}^0(s,\nu),A_{t}^i(s,\nu))\,,
   &
   A_{0}^i(s,\nu)=x^i
\end{align*}
for  $i=1,\dots,N$. Hence the flow $A_{t}(s,\nu)$ depends implicitly on the  size $N=\crochet{\nu,1}$ of the population $\nu$. 

\medskip

The stochastic process $(\nu_{t})_{t\geq 0}$ features a jump dynamics (division and up-take) and follows the dynamics of the flow $A_{t}$ between the jumps.
We can therefore generalize a  well-known formula for the pure jump process:
\begin{align}
\label{eq.dyn.saut.flot}
  \Phi(S_{t},\nu_{t})
  &=
  \Phi(A_{t}(S_{0},\nu_{0}))
  +
  \sum_{u\leq t} 
    \bigl[ 
       \Phi(A_{t-u}(S_{u},\nu_{u}))-\Phi(A_{t-u}(S_{u},\nu_{\umoins}))
    \bigr]  
  \,,\quad t\geq 0
\end{align}
for any function  $\Phi$ defined on $\R\times\MM(\XX)$. 

The sum $ \sum_{u\leq t}$ contains only a finite number of terms as the process $(\nu_{t})_{t\geq 0}$ admits only a finite number of jumps over any finite time interval. Indeed, the number of jumps in the process $(\nu_{t})_{t\geq 0}$ is bounded by a linear birth and death process with \textit{per capita} birth rate $\lambdab$
and \textit{per capita} death rate $D$ \citep{allen2003a}.

\section{Microscopic process}
\label{sec.processus.microscopique}

Let $(S_0,\nu_0)$ denote the initial condition of the process, it is a random variable with values in $\R_+ \times \M(\X)$.

The equation \eqref{eq.dyn.saut.flot} includes information on the flow, i.e. the dynamics between the jumps, but no information on the jumps themselves. To obtain an explicit equation for $(S_t,\nu_t)_{t\geq 0}$ we introduce Poisson random measures which manage the incoming of new individuals by cell division on the one hand, and the withdrawal of individuals by up-take on the other. To this end we consider two punctual Poisson random  measures  $N_1(\dif u, \dif j, \dif \alpha, \dif \theta)$ and $N_2(\dif u, \dif j)$
respectively defined on $\RR_+ \times  \NN^* \times \X \times [0,1]$ and $\RR_+ \times  \NN^*$ with respective intensity measures:
\begin{align*}
   n_1(\dif u, \dif j, \dif \alpha, \dif \theta)
   &\eqdef
   \lambdab\, \dif u \,\Big(\sum_{k \geq 0} 
       \delta_k(\dif j)\Big) \, Q(\dif \alpha) \, \dif \theta\,,
\\
   n_2(\dif u, \dif j)
   &\eqdef
   D\,\dif u \,\Big(\sum_{k \geq 0} \delta_k(\dif j)\Big)\,.
\end{align*}
Suppose that $N_1$, $N_{2}$, $S_{0}$ and $\nu_{0}$ are mutually independent.
Let $(\F_t)_{t \geq 0}$ be the canonical filtration generated by  $(S_0,\nu_0)$, $N_1$ and $N_2$.
According to \eqref{eq.dyn.saut.flot}, for any function  $\Phi$ defined on  
$\R\times\MM(\XX)$:
\begin{align}
\nonumber
  &\Phi(S_{t},\nu_{t})
  =
  \Phi(A_{t}(S_{0},\nu_{0}))
\\
\nonumber
  &
  \quad+
  \iiiint\limits_{[0,t]\times\N^*\times[0,1]^2}
        1_{\{j\leq N_{\umoins}\}} \, 
        1_{\{0\leq \theta \leq \lambda(S_{u},x_\umoins^j)/\lambdab\}}\,
        \bigl[\Phi(A_{t-u}(S_{u},\nu_{\umoins}
                -\delta_{x^j_\umoins}
                +\delta_{\alpha\,x^j_\umoins}
                +\delta_{(1-\alpha)\,x^j_\umoins}))
  \\[-0.7em]&\qquad\qquad\qquad\qquad\qquad\qquad\qquad
\nonumber
             -\Phi(A_{t-u}(S_{u},\nu_{\umoins}))\bigr]\,
		 		N_1(\dif u, \dif j, \dif \alpha, \dif \theta) 
\\[0.2em]
\label{eq.Phi.S.nu}
  &\quad+
		\iint\limits_{[0,t]\times\N^*} 
			1_{\{j\leq N_{\umoins}\}} \, 
			 \bigl[\Phi(A_{t-u}(S_{u},\nu_{\umoins}-\delta_{x^j_\umoins}))
			   -\Phi(A_{t-u}(S_{u},\nu_{\umoins}))\bigr]
			 			 \,	N_2(\dif u, \dif j)\,.
\end{align}
In particular, we obtain the following equation for the couple  $(S_{t},\nu_{t})$:
\begin{align}
\nonumber
  &(S_{t},\nu_{t})
  =
  A_{t}(S_{0},\nu_{0})
\\
\nonumber
  &
  \qquad+
  \iiiint\limits_{[0,t]\times\N^*\times[0,1]^2}
        1_{\{j\leq N_{\umoins}\}} \, 
        1_{\{0\leq \theta \leq \lambda(S_{u},x_\umoins^j)/\lambdab\}}\,
        \bigl[A_{t-u}(S_{u},\nu_{\umoins}
                -\delta_{x^j_\umoins}
                +\delta_{\alpha\,x^j_\umoins}
                +\delta_{(1-\alpha)\,x^j_\umoins})
\\[-0.7em]
\nonumber
&\qquad\qquad\qquad\qquad\qquad\qquad\qquad
             -A_{t-u}(S_{u},\nu_{\umoins})\bigr]\,
		 		N_1(\dif u, \dif j, \dif \alpha, \dif \theta) 
\\[0.2em]
\label{def.proc.S.nu}
  &\qquad+
		\iint\limits_{[0,t]\times\N^*} 
			1_{\{j\leq N_{\umoins}\}} \, 
			 \big[A_{t-u}(S_{u},\nu_{\umoins}-\delta_{x^j_\umoins})
			   -A_{t-u}(S_{u},\nu_{\umoins})\big]
			 			 \,	N_2(\dif u, \dif j)\,.
\end{align}

\medskip

From now on, we consider test functions $\Phi$ of the form:
\[
   \Phi(s,\nu)
   =
   F(s,\crochet{\nu,f})
\]
with $F \in C^{1,1}(\RR^+ \times \RR)$ and $f \in C^1(\X)$.

\begin{lemma}
\label{lem.gen.inf}
For any $t>0$:
\begin{align}
\nonumber
  &
  F(S_t,\crochet{\nu_t, f})
  = 
  F(S_0,\crochet{\nu_0, f})
\\
\nonumber
  &\quad
   + \int_0^t  \left[
	 		\rhos (S_u,\nu_u) \, 
	 		\partial_s F\left(S_u, \, \left<\nu_u, f\right>\right)
	 		+ \left< \nu_u , \, \rhog(S_u,.)\,f'\right> \,
	 		\partial_x F\left(S_u, \, \left<\nu_u, f\right>\right)
	 	\right] \dif u 
\\
\nonumber  
  &\quad
  + \iiiint\limits_{[0,t]\times\N^*\times[0,1]^2} 
		1_{\{j\leq N_{\umoins}\}} \,
		1_{\{0\leq \theta \leq \lambda(S_{u},x_\umoins^j)/
				\lambdab\}}\, 
		\bigl[
			F(S_u,\crochet{\nu_{\umoins}
					-\delta_{x_\umoins^j}
					+\delta_{\alpha \, x_\umoins^j}
					+\delta_{(1-\alpha)\, x_\umoins^j}, f} 
			)
        \\[-0.9em] 
        \nonumber
        & \qquad\qquad\qquad\qquad\qquad\quad
                     \qquad\qquad\qquad\qquad\qquad
			-
			F(S_u, \crochet{\nu_{\umoins}, f}) 
	    \bigr] 
	    \;
		 N_1(\dif u, \dif j, \dif \alpha, \dif \theta)  
\\
\label{eq.F.f}
  &\quad  
  + \iint\limits_{[0,t]\times\N^*} 
	  1_{\{j\leq N_{\umoins}\}} \, 
	  \bigl[
		F(S_u,  \crochet{\nu_{\umoins}-\delta_{x_\umoins^j}, f})
	    -
	    F\left(S_u, \, \left<\nu_{\umoins}, f\right>\right) 
	  \bigr] \, N_2(\dif u, \dif j).
\end{align}
\end{lemma}

\begin{proof}
From \eqref{eq.Phi.S.nu}:
\begin{align*}
  &
  \crochet{\nu_t, f}
  = 
  \sum_{i=1}^{N_0} f(A_t^i(S_0,\nu_0))
\\
  &\quad
  + \iiiint\limits_{[0,t]\times\N^*\times[0,1]^2}
     1_{\{j\leq N_{\umoins}\}} \, 
     1_{\{0\leq \theta \leq \lambda(S_{u},x_{\umoins}^j)/\lambdab\}}\,
\\[-0.5em]
	& \qquad\qquad\qquad
     \times\Bigl[
   	    \sum_{i=1}^{N_{\umoins}+1}
		 	f(A_{t-u}^i(S_u,
		 		\nu_{\umoins}-\delta_{x_{\umoins}^j}
				+\delta_{\alpha\,x_{\umoins}^j}
		 	    +\delta_{(1-\alpha)\,x_{\umoins}^j}))
\\[-1em]
	& \qquad\qquad\qquad\qquad\qquad\qquad\qquad\qquad
		-\sum_{i=1}^{N_{\umoins}}
		 	f(A_{t-u}^i(S_u,\nu_{\umoins}))
    \Bigl] \,
    N_1(\dif u, \dif j, \dif \alpha, \dif \theta) 
\\
  &\quad
  + \iint\limits_{[0,t]\times\N^*} 
		1_{\{j\leq N_{\umoins}\}} \, 
		\Bigl[
		  \sum_{i=1}^{N_{\umoins}-1}
		 	    f(A_{t-u}^i(S_u,
		 			\nu_{\umoins}-\delta_{x_{\umoins}^j}))
		  -
		  \sum_{i=1}^{N_{\umoins}} f(A_{t-u}^i(S_u,\nu_{\umoins})) 
		\Bigl] \, N_2(\dif u, \dif j).
\end{align*}
According to the chain rule formula, for any  $\nu=\sum_{i=1}^{N}\delta_{x^i}$:
\begin{align*} 
  f(A_{t-u}^i(s,\nu))
  & = 
  f(x^i)
  +
  \int_u^t 	\rhog(A_{\tau-u}^0(s,\nu),A_{\tau-u}^i(s,\nu)) \, f'(A_{\tau-u}^i(s,\nu))
		 \,\dif \tau
\\
  & = 
  f(x^i)
  +
  \int_u^t \varphi(A_{\tau-u}^0(s,\nu),A_{\tau-u}^i(s,\nu)) \,\dif \tau
\end{align*}
for $i \leq N$, with:
\begin{align*}
  \varphi(s,x) \eqdef \rhog(s,x) \, f'(x)\,.
\end{align*}
Hence:
\begin{align*}
  &\crochet{\nu_t, f}
  = 
  \crochet{\nu_0, f} 
\\
  & \qquad 
  + \iiiint\limits_{[0,t]\times\N^*\times[0,1]^2}
		1_{\{j\leq N_\umoins\}} \, 
		1_{\{0\leq \theta \leq \lambda(S_{u},x_\umoins^j)/\lambdab\}}
\\[-0.8em]
  &\qquad\qquad\qquad\qquad\qquad
  \times  \,
			 \left[
			    f(\alpha \, x_\umoins^j)
			    + f((1-\alpha) \, x_\umoins^j)
				- f(x_\umoins^j)
		     \right] \, 
  N_1(\dif u, \dif j, \dif \alpha, \dif \theta) 
\\
  &\qquad  
  - \iint\limits_{[0,t]\times\N^*}  
			1_{\{j\leq N_\umoins\}} \, f(x_\umoins^j) \,
			 \,	N_2(\dif u, \dif j)
  +T_0+T_1+T_2
\end{align*}
where:
\begin{align*}
  T_0 
  &\eqdef
  \sum_{i=1}^{N_0} \int_0^t \varphi(A_\tau^0(S_0,\nu_0), A_\tau^i(S_0,\nu_0)) \, \rmd\tau	
\\[1em]
  T_1
  &\eqdef 
  \iiiint\limits_{[0,t]\times\N^*\times[0,1]^2}
    1_{\{j\leq N_\umoins\}} \,
    1_{\{0\leq \theta \leq \lambda(S_{u},x_\umoins^j)/\lambdab\}} \, 
\\[-0.7em]
  &\qquad\qquad\qquad
  \times\int_u^t \Bigl[
	 \sum_{i=1}^{N_\umoins+1}
		 	\varphi(A_{\tau-u}^0(S_u,
		 			\nu_\umoins-\delta_{x_\umoins^j}
					+\delta_{\alpha \, x_\umoins^j}
		 			+\delta_{(1-\alpha) \, x_\umoins^j}),
\\[-1em]
  &\qquad\qquad\qquad\qquad\qquad\qquad\qquad\qquad 
		       A_{\tau-u}^i(S_u,\nu_\umoins
		           -\delta_{x_\umoins^j}
		           +\delta_{\alpha \, x_\umoins^j}
		 		   +\delta_{(1-\alpha) \, x_\umoins^j}))
\\[-0.6em]
  &\qquad\qquad\qquad\qquad\qquad\qquad
	 -\sum_{i=1}^{N_\umoins}
		 \varphi(A_{\tau-u}^0(S_u,\nu_\umoins),A_{\tau-u}^i(S_u,\nu_\umoins))
			 \Bigr]\,\dif \tau 	
\\
	&\qquad\qquad\qquad
	\times N_1(\dif u, \dif j, \dif \alpha, \dif \theta)
\\[1em]
  T_2
  &\eqdef  
  \iint\limits_{[0,t]\times\N^*} 1_{\{j\leq N_\umoins\}}
  \; \int_u^t
	\Bigl[ \sum_{i=1}^{N_\umoins-1}
		 	\varphi(A_{\tau-u}^0(S_u,
		 			\nu_\umoins-\delta_{x_\umoins^j}), A_{\tau-u}^i(S_u,
		 			\nu_\umoins-\delta_{x_\umoins^j}))
\\[-0.6em]
  &\qquad\qquad\qquad\qquad\qquad\qquad
			-\sum_{i=1}^{N_\umoins}
		 	\varphi(A_{\tau-u}^0(S_u,\nu_\umoins),A_{\tau-u}^i(S_u,\nu_\umoins)) 
			 \Bigr] \, \dif \tau
  \;  N_2(\dif u, \dif j).
\end{align*}
Fubini's theorem applied to $T_1$ and $T_2$ leads to:
\begin{align*}
  T_1
  &= 
  \int_0^t \iiiint\limits_{[0,\tau]\times\N^*\times[0,1]^2}
   1_{\{j\leq N_\umoins\}} \, 
   \times 1_{\{0\leq \theta \leq \lambda(S_{u},x_\umoins^j)/\lambdab\}} \, 
\\[-0.7em]
  &\qquad\qquad\qquad
  \times \Bigl[
		\sum_{i=1}^{N_\umoins+1}
		 	\varphi(A_{\tau-u}^0(S_u,
		 			\nu_\umoins
					-\delta_{x_\umoins^j}
					+\delta_{\alpha \, x_\umoins^j}
		 			+\delta_{(1-\alpha) \, x_\umoins^j}),
        \\[-1em]
        &\qquad\qquad\qquad\qquad\qquad\qquad\qquad\qquad
        A_{\tau-u}^i(S_u,
             \nu_\umoins-\delta_{x_\umoins^j}
                +\delta_{\alpha \, x_\umoins^j}
		 		+\delta_{(1-\alpha) \, x_\umoins^j}))
\\[-0.7em]
  &\qquad\qquad\qquad\qquad\qquad
		-\sum_{i=1}^{N_\umoins}
		 	\varphi(A_{\tau-u}^0(S_u,\nu_\umoins),A_{\tau-u}^i(S_u,\nu_\umoins))
  \Bigr] 
\\
  &\qquad\qquad\qquad
  \times N_1(\dif u, \dif j, \dif \alpha, \dif \theta) \;\dif \tau
\\[1em]
  T_2
  &\eqdef
  \int_0^t \iint\limits_{[0,\tau]\times\N^*} 
  1_{\{j\leq N_\umoins\}} \, 
  \Bigl[
		\sum_{i=1}^{N_\umoins-1}
		 	\varphi(A_{\tau-u}^0(S_u,
		 			\nu_\umoins-\delta_{x_\umoins^j}), A_{\tau-u}^i(S_u,
		 			\nu_\umoins-\delta_{x_\umoins^j}))
    \\[-1em]
	& \qquad\qquad\qquad\qquad\qquad\qquad 
		-\sum_{i=1}^{N_\umoins}
		 	\varphi(A_{\tau-u}^0(S_u,\nu_\umoins),A_{\tau-u}^i(S_u,\nu_\umoins))
   \Bigr] 
   \,	N_2(\dif u, \dif j)\; \dif \tau
\end{align*}
so, according to \eqref{eq.Phi.S.nu}:
\begin{align*}
  T_0+T_1+T_2
  &= 
  \int_0^t \crochet{\nu_\tau,\varphi(S_\tau,.)} \, \dif\tau \,.		 
\end{align*}
Finally,
\begin{align*}
  \crochet{\nu_t,f}
  &=
  \crochet{\nu_0, f} 
  + 
  \int_0^t \crochet{\nu_u,\rhog(S_u,.)\,f'} \, \dif u 
\\
  &\quad  
  + \iiiint\limits_{[0,t]\times\N^*\times[0,1]^2} 
	1_{\{j\leq N_\umoins\}} \, 
	1_{\{0\leq \theta \leq \lambda(S_{u},x_\umoins^j)/\lambdab\}}\,
	\bigl[f(\alpha \,x_\umoins^j)+f((1-\alpha)\,x_\umoins^j)
				-f(x_\umoins^j)\bigr] 
\\[-1em]
  &\qquad\qquad\qquad\qquad\qquad\qquad\qquad\qquad\qquad
  \times N_1(\dif u, \dif j, \dif \alpha, \dif \theta) 
\\
  &\quad
  - \iint\limits_{[0,t]\times\N^*}  
			1_{\{j\leq N_\umoins\}} \, f(x_\umoins^j) \,
			 \,	N_2(\dif u, \dif j).
\end{align*}

Since $f$ and $f'$ are continuous and bounded (bounded as defined on a compact set), we can conclude this proof by using the Itô formula for stochastic integrals with respect to Poisson random measures \citep{rudiger2006a} to develop the differential of $F(S_{t},\crochet{\nu_t,f})$ using Equation \eqref{eq.substrat} and the previous equation.
\carre
\end{proof}

\bigskip

Consider the compensated Poisson random  measures associated with $N_1$ and $N_2$:
\begin{align*}
  \tilde N_1(\dif u, \dif j, \dif y, \dif \theta)
  &\eqdef
  N_1(\dif u, \dif j, \dif y, \dif \theta)
  -
  n_1(\dif u, \dif j, \dif y, \dif \theta)\,,
\\
 \tilde N_2(\dif u, \dif j)
  &\eqdef
  N_2(\dif u, \dif j)-n_2(\dif u, \dif j)\,.
\end{align*}

\medskip

As the integrands in the Poissonian integrals of \eqref{eq.F.f} are predictable, one can make use of the result of \citet[p. 62]{ikeda1981a}:
\begin{proposition}
\label{prop.martingale}
Let:
\begin{align*}
  M^1_t
  &\eqdef 
  \iiiint\limits_{[0,t]\times\N^*\times[0,1]^2} 
     1_{\{j\leq N_\umoins\}} \, 
     1_{\{0\leq \theta \leq \lambda(S_{u},x_\umoins^j)/\lambdab\}} 
\\[-1em]
  &\qquad\qquad\qquad\qquad
    \times
    \left[
	 F\left(S_u,\left<\nu_\umoins-\delta_{x_\umoins^j}
			 +\delta_{\alpha \, x_\umoins^j}+\delta_{(1-\alpha)\, x_\umoins^j}, f\right> \right)
				-F\left(S_u, \, \left<\nu_\umoins, f\right>\right) 
	 \right]
\\
  &\qquad\qquad\qquad\qquad
  \times \tilde N_1(\dif u, \dif j, \dif \alpha, \dif \theta)\,,
\\
  M^2_t
  &\eqdef 
  \iint\limits_{[0,t]\times\N^*}
	 1_{\{j\leq N_\umoins\}} \, 
	 \left[
	  F\left(S_u, \left<\nu_\umoins-\delta_{x_\umoins^j}, f\right>\right)
	  -
	  F\left(S_u, \, \left<\nu_\umoins, f\right>\right) 
	 \right] \,	\tilde N_2(\dif u, \dif j)\,.
\end{align*}
We have the following properties of martingales:
\begin{enumerate}

\item if for any $t\geq 0$:
\begin{align*}
  &\EE\Bigl(
    \int_0^t \int_{\X} \lambda(S_{u}, x) \int_0^1
	 \bigl|
         F(S_u,\crochet{\nu_u-\delta_x+\delta_{\alpha \, x}
                +\delta_{(1-\alpha)\,x},f})
\\
  &\qquad\qquad\qquad\qquad\qquad\qquad\qquad
  -F(S_u,\crochet{\nu_u,f})
	\bigr| \, Q(\dif \alpha)\, \nu_u(\dif x)\,\dif u
  \Bigr) < + \infty
\end{align*}
then  $(M^1_t)_{t\geq 0}$ is a martingale;
\item if for any $t\geq 0$
\begin{align*}
  \EE\Bigl(
    \int_0^t \int_{\X} 
	 \bigl|
		F(S_u,\crochet{\nu_u-\delta_x,f})
		-
		F(S_u,\crochet{\nu_u,f})
	 \bigr|
			\, \nu_u(\dif x)\,\dif u
	\Bigr) < +\infty
\end{align*}
then  $(M^2_{t})_{t\geq 0}$ is a martingale;

\item if for any $t\geq 0$
\begin{align*}
  &\EE\Bigl(
    \int_0^t \int_{\X} \lambda(S_{u}, x) \int_0^1 
	  \bigl|
	     F(S_u,\crochet{
	        \nu_u-\delta_x+\delta_{\alpha \, x}+\delta_{(1-\alpha)\,x},f
	        })
      \\ 
      &  \qquad\qquad\qquad\qquad\qquad\qquad\qquad
		 -
		 F(S_u,\crochet{\nu_u,f})
	  \bigr|^2 \,
	  Q(\dif \alpha)\,\nu_u(\dif x)\,\dif u
  \Bigr)
  < + \infty
\end{align*}
then  $(M^1_{t})_{t\geq 0}$ 
is a square integrable martingale and predictable quadratic variation:
\begin{align*}
  &\crochet{M^1}_t	
  \eqdef
  \int_0^t \int_{\X} \lambda(S_{u}, x) \, \int_0^1 
    \bigl[
		F(S_u,
		  \crochet{
		    \nu_u-\delta_x+\delta_{\alpha \, x}+\delta_{(1-\alpha)\,x}
		    ,f})
      \\ 
      &  \qquad\qquad\qquad\qquad\qquad\qquad\qquad
		-
		F(S_u,\crochet{\nu_u,f})
	\bigr]^2 \,
	Q(\dif \alpha)\, \nu_u(\dif x)\,\dif u\,;
\end{align*}

\item if for any $t\geq 0$
\begin{align*}
  \EE\Bigl(\int_0^t \int_{\X} 
     \left|
	    F(S_u,\crochet{\nu_u-\delta_x,f})
		-
		F(S_u,\crochet{\nu_u,f})
	\right|^2 \, \nu_u(\dif x)\,\dif u
  \Bigr) < + \infty
\end{align*}
then  $(M^2_{t})_{t\geq 0}$ 
is a square integrable martingale and predictable quadratic variation:
\begin{align*}
  \left<M^2\right>_t	
  \eqdef
  D \, \int_0^t \int_{\X} 
  \left[
    F(S_u,\crochet{\nu_u-\delta_x,f})
	-
	F(S_u,\crochet{\nu_u,f})
  \right]^2 \, 
  \nu_u(\dif x)\,\dif u\,.
\end{align*}
\end{enumerate}
\end{proposition}

\bigskip

\begin{lemma}[Control of the population size]
\label{lem.taille.pop}
Let $T>0$, if there exists $p\geq 1$ such that $\EE(\crochet{\nu_0,1}^p) <\infty$, then:
\begin{align*}
  \EE\left( 
    \sup_{t \in [0,T]} \left< \nu_t,1 \right>^p 
  \right) 
  \leq C_{p,T}
\end{align*}
where $C_{p,T}<\infty $ depends only on $p$ and $T$.
\end{lemma}

\begin{proof}
For any $n \in \NN$, define the following stopping time:
\begin{align*}
  \tau_n \eqdef \inf \{t \geq 0, \, N_t \geq n \}\,.
\end{align*}
Lemma \ref{lem.gen.inf} applied to $F(s,x)=x^p$ and $f(x)=1$ leads to:
\begin{multline*}
  \sup_{u\in [0,t \wedge \tau_n]} \crochet{\nu_u,1}^p
  \leq 
  \crochet{\nu_0,1}^p
\\
  + \iiiint\limits_{[0,t\wedge \tau_n]\times\N^*\times[0,1]^2} 
	 	1_{\{j\leq N_\umoins\}}\, 
		1_{\{0\leq \theta \leq \lambda(S_{u}, x_\umoins^j)/ \lambdab\}}\,
	 	\bigl[ 
		  (\crochet{ \nu_\umoins,1}+1)^p-\crochet{\nu_\umoins,1}^p 
		\bigr] 
\\[-1em]
	\times N_1(\dif u, \dif j,\dif \alpha,  \dif \theta)\,.
\end{multline*}
From inequality  $(1+y)^p-y^p\leq C_p\,(1+y^{p-1})$ we get:
\begin{multline*}
  \sup_{u\in [0,t \wedge \tau_n]} \crochet{\nu_u,1}^p
  \leq  
  \crochet{\nu_0,1}^p 
\\
  + C_p\,\iiiint\limits_{[0,t\wedge \tau_n]\times\N^*\times[0,1]^2} 
  1_{\{j\leq N_\umoins\}}\,
  1_{\{0\leq \theta \leq \lambda(S_{u}, x_\umoins^j)/ \lambdab\}}\, 
  \bigl[(1+\crochet{\nu_\umoins,1}^{p-1}\bigr] 
\\[-1em]
  \times N_1(\dif u, \dif j,\dif \alpha,  \dif \theta)\,.
\end{multline*}
Proposition \ref{prop.martingale}, together with the inequality  $(1+y^{p-1})\,y \leq 2\,(1+y^p)$ give:
\begin{align*}
  \EE \Biggl(
    \sup_{u\in [0,t \wedge \tau_n]} \crochet{\nu_u,1}^p 
  \Biggr)
  \leq 
  \EE(\crochet{\nu_0,1}^p)
  + 2\,\lambdab\,C_p\,
    \EE\int_0^{t} 
		\bigl( 1+\crochet{ \nu_{u\wedge \tau_n},1}^p\bigr)\,\dif u\,.
\end{align*}
Fubini's theorem and Gronwall's inequality allow us to conclude that for any $T< \infty$:
\begin{align*}
  \EE\Biggl(
    \sup_{t\in [0,T \wedge \tau_n]} \crochet{\nu_t,1}^p 
  \Biggr)
  \leq 
  	\Bigl(\EE\bigl(\crochet{\nu_0,1}^p\bigr) + 2\,\lambdab\,C_p\,T \Bigr)  \,
  	\exp(2\,\lambdab\,C_p\,T)
  \leq
  C_{p,T}
\end{align*}
where $C_{p,T}<\infty$ as  $\EE(\crochet{\nu_0,1}^p)<\infty$.

In addition, the sequence of stopping times $\tau_n$ tends to infinity,  otherwise there would exist $T_0<\infty$ such that $\PP(\sup_n \tau_n <T_0)=\varepsilon_{T_0}>0$ hence $\EE(\sup_{t \in [0,T_0 \wedge \tau_n]} \crochet{\nu_t,1}^p)\geq \varepsilon_{T_0}\,n^p$  which contradicts the above inequality. Finally, Fatou's lemma gives:
\begin{align*} 
  \EE\Biggl( \sup_{t \in [0,T]} \crochet{\nu_t,1}^p \Biggr)
  =
  \EE \Biggl( 
	  \liminf_{n \to \infty} 
	    \sup_{t\in [0,T \wedge \tau_n]} 
	    \crochet{\nu_t,1}^p
  \Biggr)
  \leq 
  \liminf_{n \to \infty}
   \EE \Biggl(
     \sup_{t\in [0,T \wedge \tau_n]} \crochet{\nu_t,1}^p 
   \Biggr)
  \leq 
  C_{p,T}\,.
\end{align*}
\vskip-1em\carre
\end{proof}

\bigskip

\begin{remark}
In particular, if $\EE\crochet{\nu_0,1}<\infty$ and if the function $ F $ is bounded, then by Lemma \ref{lem.taille.pop} and Proposition \ref{prop.martingale}, $(M^1_{t})_{t\geq 0}$ and $(M^2_{t})_{t\geq 0}$ are martingales.
\end{remark}

\begin{lemma}
\label{lem.integrabilite_Sigma}
If $\EE\crochet{\nu_0,1}+\EE(S_0)<\infty$ then:
\begin{align*}
  \EE\biggl(\int_0^t |\rhos (S_u,\nu_u) |\,\dif u \biggr)
  \leq	
  D\, t\,\EE (S_0\vee \Sin)
  + 
  \frac kV \,\bar{g}\;
  \EE\biggl( \int_0^t\crochet{\nu_u,1} \, \dif u\biggr)
  <\infty\,.
\end{align*}
\end{lemma}

\begin{proof}
As $S_u\geq 0$ and $\rhog$ is a non negative function,
\begin{align*}
  \rhos (S_u,\nu_u)
  \leq 
  D \, \Sin\,.
\end{align*}
Furthermore, for any $(s,x)\in\RR_+ \times \X$,  $\rhog(s,x)\leq \bar g$, and 
$S_u\leq S_0\vee \Sin$ so:
\begin{align*}
  \rhos (S_u,\nu_u)
  \geq 
  -D \, (S_0\vee \Sin)-\frac kV \, \bar{g}\,\crochet{\nu_u,1}\,.
\end{align*}
We therefore deduce that:
\begin{align*}
  \int_0^t |\rhos (S_u,\nu_u) |\, \dif u
  \leq	
  D\,t\,(S_0\vee \Sin)
  + 
  \frac kV \, \bar{g}\, \int_0^t \crochet{\nu_u,1}\, \dif u\,. 
\end{align*}
According to Lemma \ref{lem.taille.pop}, the last term is integrable 
which concludes the proof.
\carre
\end{proof}

\bigskip

\begin{theorem}[Infinitesimal generator]
The process $(S_t,\nu_t)_{t\geq 0}$ is Markovian with values in  $\R_{+}\times\MM(\XX)$ and its infinitesimal generator is:
\begin{multline}
  \L \Phi(s,\nu)
  \eqdef 
  \bigl(D(\Sin-s)-k\,\mu(s,\nu) \bigr)\,
  \partial_s F(s, \crochet{\nu, f}) 		
  +
  \crochet{\nu ,\rhog(s,.) \,  f'} \,
  \partial_x F(s, \crochet{\nu, f}) 
\\
  +  
  \int_{\X} \lambda(s,x) \,
     \int_0^1 			  
	   \bigl[ 
	      F(s,
	         \crochet{\nu-\delta_x+\delta_{\alpha\,x}
		    +
		    \delta_{(1-\alpha)\,x},f}
		  ) 
		  -
		  F(s, \crochet{\nu,f}) 
	   \bigr]			
			  Q(\dif \alpha) \, \nu(\dif x) 
\\
\label{eq.generateur.infinitesimal}
		+ D\,  \int_{\X} 
			 \bigl[
			F(s,\crochet{\nu-\delta_x,f})
			-
			F(s,\left<\nu,f\right>) \bigr]\,
			 \,	\nu(\dif x) 
\end{multline}
for any $\Phi(s,\nu)=F(s,\crochet{\nu,f})$ with  $F \in C_b^{1,1}(\RR^+ \times \RR)$ and $f \in C^{1}(\X)$. Thereafter $\L \Phi(s,\nu)$ is denoted 
$\L F(s,\crochet{\nu,f})$.
\end{theorem}

\begin{proof}
Consider deterministic initial conditions  $S_{0}=s\in\R_{+}$ and $\nu_{0}=\nu\in\MM(\XX)$. 
According to Lemma \ref{lem.gen.inf}:
\begin{align*}
  &
  \EE\bigl(F(S_t,\crochet{\nu_t, f})\bigr)
  =
  F(s, \, \left<\nu, f\right>)
  + 
  \EE\biggl(\int_0^t 
	   \rhos (S_u,\nu_u) \, 
	   \partial_s F(S_u, \, \crochet{\nu_u, f})\, \dif u 
    \biggr)
\\
  &\ 
  +\EE \biggl(\int_0^t
	   \crochet{ \nu_u , \rhog(S_u,.)\,f'} \;
	   \partial_x F\left(S_u, \, \left<\nu_u, f\right>\right)\,
	   \dif u
   \biggr) 
\\
  &\ 
  + \EE\biggl(\ \ 
	  \iiiint\limits_{[0,t]\times\N^*\times[0,1]^2} \!\!\!\!
		 1_{\{j\leq N_\umoins\}} \, 
		 1_{\{0\leq \theta \leq \lambda(S_{u}, x_\umoins^j)/
				\lambdab\}}      
			\Bigl[
			 F\bigl(
			   S_u
			   , 
			   \crochet{\nu_\umoins
			       -\delta_{x_\umoins^j}
			       + \delta_{\alpha \, x_\umoins^j}
			       + \delta_{(1-\alpha)\, x_\umoins^j}
			     , f}
			 \bigr)
             \\[-1em]
             &\qquad\qquad\qquad\qquad\qquad
              \qquad\qquad\qquad\qquad\qquad 	      
			 -
			 F\bigl(
			   S_u
			   , 
			   \crochet{\nu_\umoins, f}\bigr) \Bigr]\,
		 	   N_1(\dif u, \dif j, \dif \alpha, \dif \theta)
			\bigr)
\\
  &\ 
  + \EE\biggl( \ \ 
     \iint\limits_{[0,t]\times\N^*} 
		1_{\{j\leq N_\umoins\}} \, 
		\Bigl[
			 F(S_u, \crochet{\nu_\umoins-\delta_{x_\umoins^j}, f})
			 -F(S_u,\crochet{\nu_\umoins, f}) 
	    \Bigr]\,
	    N_2(\dif u,\dif j) 
	 \biggr).
\end{align*}
As  functions $F$, $\partial_sF$, $\partial_x F$, $f'$ and $\rhog$  are bounded, from Lemmas \ref{lem.taille.pop} and \ref{lem.integrabilite_Sigma} and Proposition \ref{prop.martingale}, the expectation the right side of the above equation are finite.

Furthermore, from Proposition \ref{prop.martingale}:
\begin{align*}
  &
  \EE\left(F\left(S_t, \, \left<\nu_t, f\right>\right)\right)
  =
  F(s,\crochet{\nu, f})
  + 
  \EE (\Psi(t))
\end{align*}
where:
\begin{align*}
 & \Psi(t)
  \eqdef
   	 \int_0^t 
	 	\rhos (S_u,\nu_u) \, 
	 	\partial_s F(S_u, \crochet{\nu_u,f})\, \dif u   
\\
  &\quad 	
  +
  \int_0^t
	 \crochet{\nu_u,\rhog(S_u,.)\,f'} \;
	 \partial_x F(S_u, \crochet{\nu_u,f})\,
	 \dif u 
\\
  &\quad  
  + 
     \int_0^t \int_{\X} \int_0^1 
		\lambda(S_{u}, x)\, 
		\left[
			 F\left(
			   S_u
			   ,
			   \crochet{\nu_u -\delta_{x}
			   		+\delta_{\alpha\, x}
			   		+\delta_{(1-\alpha)\,x}
			   \,,\, 
			   f}
			 \right)
			 \right.   
			\\[-0.7em]
			&\qquad\qquad\qquad\qquad\qquad
			 \qquad\qquad\qquad\qquad
			\left. 
			- F(S_u,\crochet{\nu_u,f}) 
		\right]	\,
		Q(\dif \alpha)\,\nu_u(\dif x) \, \dif u  
\\
  &\quad  
  + 
  D\, 
    \int_0^t \int_{\X} 
  	\bigl[
		F(S_u,\crochet{\nu_u-\delta_{x}, f}) - F(S_u,\crochet{\nu_u, f}) 
	\bigr]\,
	\nu_u(\dif x) \, \dif u
	\,.
\end{align*}
Also:
\begin{align*}
  &
  \frac{\partial}{\partial t} 
  \Psi(t)\Bigr|_{t=0}
  =   
  \bigl(D(\Sin-s) - k\,\mu(s,\nu)\bigr) \, 
  \partial_s F(s, \crochet{\nu,f}) 
  + \crochet{\nu,\rhog(s,.)\,f'} \;
	\partial_x F(s,\crochet{\nu,f}) 
\\
  &\qquad
  + \int_{\X}\int_0^1 
	  \lambda(s, x) \,  
	  \Bigl[
		F\left(
		  s
		  ,
		  \crochet{\nu-\delta_x
		                +\delta_{\alpha \, x}+\delta_{(1-\alpha)\,x}, f
		  }
		\right)
		-
		F(s,\crochet{\nu,f}) 
	  \Bigr]\,
	  Q(\dif\alpha) \, \nu(\dif x) 
\\
  &\qquad 
  + D\,  \int_{\X} 
	  \Bigl[
			F(s,\crochet{\nu-\delta_x,f})
			-
			F(s,\crochet{\nu,f}) 
	  \Bigr]\,
	  \nu(\dif x),
\end{align*}
hence:
\begin{multline*}
  \biggl|
    \frac{\partial}{\partial t} \Psi(t)\Bigr|_{t=0}
  \biggr|
  \leq  
  D\,(\Sin+s)
  +
  \left(
	 \frac kV\,\bar g  \, \norme{\partial_s F}
	 + \bar g \, \norme{f'}\,\norme{\partial_x F}
	  + 2\,(\lambdab+D)\, \norme{F} 
  \right) \,
  \crochet{\nu,1} \, .
\end{multline*}
The right side of the last equation is finite.
One may apply the theorem of differentiation under the integral sign, hence the application
$t\mapsto \EE(F(S_t, \crochet{\nu_t, f}))$ is differentiable at $t=0$ with derivative $\L F(s,\crochet{\nu, f})$ defined by  \eqref{eq.generateur.infinitesimal}.
\mbox{}\hfil\carre
\end{proof}

\begin{remark}
We define the washout time as the stopping time:
\begin{align*}
  \tauw
  \eqdef
  \inf\{t\geq 0\,;\, N_{t}=\crochet{\nu_{t},1}=0\}
\end{align*}
with the convention $\inf\emptyset=+\infty$.
Before $\tauw$  the infinitesimal generator is given by \eqref{eq.generateur.infinitesimal}, after this time $\nu_{t}$ is the null measure, i.e. the chemostat does not contain any bacteria, and the infinitesimal generator is simply reduced to the generator associated with the ordinary  differential equation $\dot S_{t}=D\,(\Sin-S_{t})$ coupled with the null measure given by $\crochet{\nu_{t},f}=0$ for all $f$.
\end{remark}

\section{Convergence in distribution of the individual-based model}
\label{sec.convergence}

\subsection{Renormalization}

In this section we will prove that the coupled process of the  substrate concentration and  the bacterial population converges in distribution to a deterministic process in the space:
\[
 \CC([0,T],\R_{+}) 
 \times
 \D([0,T],\MM_{F}(\X)) 
\]
equipped with the product metric: 
\fenumi\ the uniform norm on $\CC([0,T],\R_{+})$; \fenumii\ the Skorohod metric
on  $\D([0,T],\MM_{F}(\X))$ where $\MM_{F}(\X)$ is equipped with the topology
of the weak convergence of measures (see Appendix \ref{appendix.skorohod}).

Renormalization must have the effect that the density of the bacterial population must grow to infinity. To this end,  we first consider a growing volume, i.e. in the previous model the volume is replaced by:
\[
  V_{n} = n\,V
\]
and $(S^n_{t},\nu^n_{t})_{t\geq 0}$ will denote the process \eqref{def.proc.S.nu} where $V$ is replaced by $V_{n}$ and
 $x_t^{n,1},\dots,x_t^{n,N_t^n}$ the $N_t^n$ individuals of $\nu_t^n$; second
 we introduce the rescaled process:
\begin{align}
\label{def.renormalisation}
  \nub_t^n \eqdef \frac{1}{n}\nu_t^n\,,\quad t\geq 0
\end{align}
and we suppose that:
\begin{align*}
   \nub_0^n 
   =
   \frac{1}{n}\nu_0^n
   \xrightarrow[n\to\infty]{} 
   \xi_0 
   \textrm{ in distribution in }\MM_F(\X) \,.
\end{align*}
$\xi_0$ is the limit measure after renormalization of the population density at the initial time. It may be random, but we will assume without loss of generality that it is deterministic, moreover we suppose that $\crochet{\xi_{0},1}>0$.

\bigskip

Therefore, this asymptotic consists in simultaneously letting the volume of chemostat  and the size of the initial population tend to infinity. 

As the substrate concentration is maintained at the same value, it implies that the population tends to infinity. We will show that the rescaled process $(S^n_t,\bar\nu^n_t)_{t\geq 0}$  defined by \eqref{def.renormalisation}  converges in distribution to a process $(S_t,\xi_t)_{t\geq 0}$ introduced later.

\bigskip

The process  $(S^n_t,\nu^n_t)_{t\geq 0}$ is defined by:
\begin{align*}
  \dot S_t^n 
  &=
  D\,(\Sin-S_t^n)-\frac k{V_{n}} \, 
    \int_{X} \rhog(S_t^n,x)\,\nu_t^n(\dif x) 
\\
  &= 
  D(\Sin-S_t^n)-\frac {k}{V} \, 
    \int_{X} \rhog(S_t^n,x)\,\nub_t^n(\dif x)
  =
  \rhos(S_{t}^n,\nub_t^n)
\end{align*}
and
\begin{align*}
  \nub_t^n
  &= 
  \frac 1n \, \sum_{j=1}^{n} \delta_{A_t^j(S_0^n,\nu_0^n)} 
\\
  &\qquad
  + \frac 1n \, 
	  \iiiint\limits_{[0,t]\times\N^*\times[0,1]^2} 
		 1_{\{i\leq N^n_\umoins\}} \, 
		 1_{
		     \{0\leq \theta \leq 
		         \lambda(S^n_u, x_\umoins^{n,i})/\lambdab\}
		   }\,
  \Bigl[
	-\sum_{j=1}^{N^n_\umoins} 
		             \delta_{A_{t-u}^j(S^n_u,\nu^n_\umoins)}
\\[-1em]
  &\qquad\qquad\qquad\qquad\qquad
	+\sum_{j=1}^{N^n_\umoins+1}
		           	\delta_{
			           A_{t-u}^j(S^n_u \, ,\,
		 			   \nu^n_\umoins-\delta_{x_\umoins^{n,i}}
					   +\delta_{\alpha x_\umoins^{n,i}}
		 			   +\delta_{(1-\alpha) x_\umoins^{n,i}})
					}
  \Bigr] \;
  N_1(\dif u, \dif i, \dif \alpha, \dif \theta)
\\[0.5em]
  &\qquad
  + \frac 1n \,  
    \iint\limits_{[0,t]\times\N^*}
    1_{\{i\leq N^n_\umoins\}} \, 
    \Bigl[
		-\sum_{j=1}^{N^n_\umoins}
		 	\delta_{A_{t-u}^j(S^n_u,\nu^n_\umoins)}
		+\sum_{j=1}^{N^n_\umoins-1}
		 	\delta_{A_{t-u}^j(S^n_u,
		 			\nu^n_\umoins-\delta_{x_\umoins^{n,i}})}
	\Bigr]
			 \,	N_2(\dif u, \dif i)
\end{align*}

\begin{remark}
\label{remark.cv.nu.seule}
Due to the structure of the previous system and specifically the above equation, 
it will be sufficient to prove the convergence in distribution of the component $\nub_t^n$ to deduce also the convergence of the component  $S_{t}^n$.
\end{remark}

\subsection{Preliminary results}

\begin{lemma}
\label{lem.renorm}
For all $t>0$,
\begin{align*}
  &
  F(S_t^n,\crochet{\nub_t^n, f})
  =
  F(S_0^n,\crochet{\nub_0^n, f})
\\
  &\qquad
  + 
  \int_0^t \textstyle\Bigl(
    D\,(\Sin-S_u^n)
    - \frac {k}{V}\,\int_X \rhog(S_u^n,x)\,\nub_u^n(\dif x) 
  \Bigr)
  \;
  \partial_s F(S_u^n, \, \crochet{\nub_u^n, f})\,\dif u
\\
  &\qquad 
  + \int_0^t
	   \crochet{\nub_u^n , \rhog(S_u^n,.)\,f'} \;
	   \partial_x F(S_u^n, \, \crochet{\nub_u^n, f}) \, \dif u 
\\
  &\qquad  
  + n\,\int_0^t \int_{\X} \lambda(S^n_u, x)\, 
	 \int_0^1
	   \textstyle
	   \Bigl[
	     F\bigl(
	       S_u^n
	       \,,\,
	       \crochet{\nub_u^n,f} +\frac 1n f(\alpha \, x )
	       +\frac 1nf((1-\alpha) \, x)-\frac 1nf(x)
	   \bigr)  
\\[-0.5em]
  & \qquad\qquad\qquad\qquad\qquad\qquad\qquad\qquad\qquad
			 -F\bigl(S_u^n, \crochet{\nub_u^n, f}\bigr) 
	   \Bigr]
	   \,Q(\dif \alpha)\, \nub_u^n(\dif x)\,\dif u  
\\			
  &\qquad  
  + D\,n\, 
  \int_0^t \int_{\X} 
  \textstyle\Bigl[
	 F\bigl(S_u^n, \, \crochet{\nub_u^n, f}-\frac 1nf(x) \bigr)
	 -
	 F\bigl(S_u^n, \, \crochet{\nub_u^n, f} \bigr) 
  \Bigr]\, \nub_u^n(\dif x)\,\dif u
	+ Z_t^{F,f,n}
\end{align*}
where
\begin{align}
\nonumber
  Z_t^{F,f,n}
  &\eqdef 
  \iiiint\limits_{[0,t]\times\N^*\times[0,1]^2}
	1_{\{i\leq N_\umoins^n\}} \, 
	1_{\{0\leq \theta \leq \lambda(S^n_u, x_\umoins^{n,i})/\lambdab\}}
\\[-1em]
\nonumber
  &\qquad\qquad\qquad\qquad
	\textstyle
	\Bigl[
		F\bigl(
		  S_u^n
		  \,,\, 
		  \crochet{\nub_\umoins^n, f}
		  + \frac 1n f(\alpha \,x_\umoins^{n,i})
		  +\frac 1nf((1-\alpha)\,x_\umoins^{n,i})
		  -\frac 1nf(x_\umoins^{n,i})
		\bigr) \, 
\\[-0.5em]
\nonumber
  &\qquad\qquad\qquad\qquad\qquad\qquad\qquad\qquad
      -F\bigl(S_u^n, \crochet{\nub_\umoins^n, f}\bigr) 	
	\Bigr]			
	\,\tilde N_1(\dif u, \dif i, \dif \alpha, \dif \theta)  
\\
\label{eq.Z,F,f}
  & \qquad 
  + \iint\limits_{[0,t]\times\N^*}
  	\textstyle 
	1_{\{i\leq N_\umoins^n\}} \, 
	\Bigl[
	  F\bigl(S_u^n, \crochet{\nub_\umoins^n, f}
	         -\frac 1nf(x_\umoins^{n,i}) \bigr)
	  -
	  F\bigl(S_u^n, \crochet{\nub_\umoins^n, f}\bigr) 
	\Bigr] 
	\,	\tilde N_2(\dif u, \dif i)
\end{align}
\end{lemma}

\begin{proof}
It is sufficient to note that
$
F(S_t^n,\crochet{ \nub_t^n,f })
	 = F(S_t^n,\crochet{\nu_t^n,\frac{1}{n}f})
$
and to apply Lemma \ref{lem.gen.inf}.
\carre
\end{proof}

\bigskip

\begin{lemma}
\label{lem.esp}
If $\sup_{n\in\N} \EE(\crochet{\nub_0^n,1}^p)<\infty$ for some $p \geq 1$, then:
\[
  \sup_{n\in\N} \EE\Biggl(\sup_{u\in [0,t]} \crochet{\nub_u^n,1}^p\Biggr)
  <
  C_{t,p}
\]
where $C_{t,p}$ depends only on $t$ and $p$.
\end{lemma}

\begin{proof}
Define the stopping time:
\[	
	\tau_N^n
	\eqdef
	\inf \{t\geq 0, \crochet{\nub_t^n,1} \geq N \}\,.
\]
According to Lemma \ref{lem.renorm}:
\begin{align*}
  &
  \sup_{u \in [0, t \wedge \tau_N^n]} \crochet{\nub_u^n,1}^p
  \leq 
  \crochet{\nub_0^n,1}^p
\\
  &\qquad\qquad 
  + n\,\int_0^{t \wedge \tau_N^n} \int_{\X}
			\lambda(S^n_u, x)\,
			\textstyle
			 \left[\left(\crochet{ \nub_u^n,1} + \frac{1}{n}\right)^p
			 -\crochet{ \nub_u^n,1}^p \right]\, 
			\nub_u(\dif x)\,\dif u
\\
  & \qquad\qquad
  + \iiiint\limits_{[0,t\wedge \tau_N^n]\times\N^*\times[0,1]^2}
		1_{\{i \leq N_\umoins^n\}}\,
		1_{\{0\leq \theta \leq \lambda(S^n_u, x_\umoins^{i,n})/ \lambdab\}}
\\[-1em]
  & \qquad\qquad\qquad\qquad\qquad\qquad\qquad
  \times \,
		\textstyle
		\left[\left(\crochet{\nub_\umoins^n,1} + \frac{1}{n}\right)^p
			 -\crochet{\nub_\umoins^n,1}^p \right]\, 
	 \tilde N_1(\dif u, \dif i,\dif \alpha,  \dif \theta)\,.
\end{align*}
From the inequality 
$(1+y)^p-y^p \leq C_p\, (1+y^{p-1})$,
one can easily check that $(\frac 1n+y)^p-y^p \leq 
\frac{C_p}{n}\, (1+y^{p-1})$. Taking expectation in the previous inequality  and applying Proposition \ref{prop.martingale} lead to:
\begin{align*}
  \EE \Biggl( 
    \sup_{u \in [0, t \wedge \tau_N^n]} \crochet{\nub_u^n,1}^p 
  \Biggr)
  &\leq 
  \EE\bigl(\crochet{ \nub_0^n,1}^p \bigr)
  + 
  \EE \int_0^{t \wedge \tau_N^n} 
    C_p\,
    \bigl(1+\crochet{\nub_u^n,1}^{p-1}\bigr)
	\int_\X \lambda(S^n_u, x)\,\nub_u^n(\dif x)\, \dif u
\\
	&\leq 
	\EE\bigl(\crochet{\nub_0^n,1 }^p \bigr)
	+ 
	\lambdab \, C_p\, 
	\int_0^t \EE \left(
	  \crochet{ \nub_{u \wedge \tau_N^n}^n,1 }
	  +
	  \crochet{\nub_{u \wedge \tau_N^n}^n,1}^p 
	\right) \, \dif u\,.
\end{align*}
As:
\[
  \crochet{ \nub_{u \wedge \tau_N^n}^n,1}  
  +
  \crochet{ \nub_{u \wedge \tau_N^n}^n,1}^p  
  \leq 
  2 \, \left(1+ \crochet{\nub_{u \wedge \tau_N^n}^n,1}^p \right)\,,
\]
we get:
\begin{align*}
  \EE \Biggl( 
    \sup_{u \in [0, t \wedge \tau_N^n]} \crochet{\nub_u^n,1}^p 
  \Biggr)
  &\leq 
  \EE\left(\crochet{\nub_0^n,1}^p \right)
			+ 2\, \lambdab \, C_p\, t
			+ 2\, \lambdab \, C_p\, \int_0^t 
			\EE \left(\sup_{u \in [0, u \wedge \tau_N^n]} 
			  \crochet{\nub_u^n,1}^p \right) \, \dif u
\end{align*}
and from Gronwall's inequality we obtain:
\begin{align*}
  \EE \Biggl( 
    \sup_{u \in [0, t \wedge \tau_N^n]} \crochet{\nub_u^n,1}^p 
  \Biggr)
  \leq 
  \Bigl(
  	\EE(\crochet{\nub_0^n,1}^p ) + 2\,\lambdab \, C_p\, t
  \Bigr)
  \,
  \exp(2\,\lambdab \, C_p\, t)\,.
\end{align*}
The sequence of stopping times $\tau_N^n$ tends to infinity as $N$ tends to infinity 
for the same reasons as those set  in the proof of Lemma \ref{lem.taille.pop}.
From Fatou's lemma we deduce:
\begin{align*}
  \EE \Biggl(  
       \sup_{u \in [0, t]} \crochet{\nub_u^n,1}^p 
  \Biggr)
  & =
  \EE \Biggl( 
     \liminf_{N \to \infty} 
       \sup_{u \in [0, t \wedge \tau_N^n]} \crochet{\nub_u^n,1}^p 
  \Biggr)
\\
  &\leq 
  \liminf_{N \to \infty}\EE 
     \Biggl(
        \sup_{u \in [0,t\wedge\tau_N^n]} \crochet{\nub_u^n,1}^p 
     \Biggr)
\\
  &\leq 
  \Bigl(
    \EE\left(\crochet{\nub_0^n,1}^p \right) 
    + 
	2\, \lambdab \, C_p\, t
  \Bigr)\, 
  \exp(2\,\lambdab\,C_p\, t)
\end{align*}
and as $\sup_n \EE\left(\crochet{\nub_0^n,1}^p \right) <\infty$, we deduce the proof of the lemma.
\carre
\end{proof}

\bigskip

\begin{corollary}
Let $f \in C^1(\X)$, suppose that  $\EE(\crochet{\nub_0^n, 1}^2)<\infty$,
then for all  $t>0$:
\begin{align}
\nonumber
  \crochet{\nub_t^n, f}
  &=
  \crochet{\nub_0^n, f}
  + \int_0^t 
	 		\left< \nub_u^n , \, \rhog(S_u^n,.)\,f'\right> 
	 		\, \dif u 
\\
\nonumber 
  &\qquad
  + \int_0^t \int_{\X} \lambda(S^n_u,x) \, 
	 \int_0^1  
			  \bigl[ f(\alpha \, x)+f((1-\alpha)\,x)-f(x)\bigr] \, 
			  Q(\dif \alpha)\,\nub_u^n(\dif x) \, \dif u			  
\\		
\label{renormalisation}
  &\qquad  
  - D\, \int_0^t \int_{\X} f(x) \, \nub_u^n(\dif x) \, \dif u   
  + Z_t^{f,n} 
\end{align}
where
\begin{align}
\nonumber
  Z_t^{f,n}
  &\eqdef
  \frac 1n \, 
  \iiiint\limits_{[0,t]\times\N^*\times[0,1]^2} 
  	1_{\{i\leq N_{\umoins}^n\}} \, 
	1_{\{0\leq \theta \leq \lambda(S^n_u, x_{\umoins}^{i,n})/\lambdab\}} 
\\[-1em]
\nonumber
  &\qquad\qquad\qquad\qquad		
	\times 
	\left[
	   f(\alpha\,x_{\umoins}^{i,n})
	   + f((1-\alpha)\,x_{\umoins}^{i,n})
	   - f(x_{\umoins}^{i,n})
	\right] \, 
	\tilde N_1(\dif u, \dif i, \dif \alpha, \dif \theta)  
\\[0.5em]
\label{eq.Zfn}
  &\qquad\qquad\qquad
  - \frac 1n \, \iint\limits_{[0,t]\times\N^*}
			1_{\{i\leq N_{\umoins}^n\}} \, 
			f(x_{\umoins}^{i,n})
			 \,	\tilde N_2(\dif u, \dif i)
\end{align}
is a martingale with the following predictable quadratic variation:
\begin{align} 
\nonumber
  \crochet {Z^{f,n}}_t
  &=
  \frac 1n  
  \int_0^t \int_{\X} \lambda(S^n_u, x) \, 
     \int_0^1 
        \left[ f(\alpha \, x)+f((1-\alpha) \, x)-f(x) \right]^2 \, 
			  Q(\dif \alpha) \,\nub_u^n(\dif x) \, \dif u  
\\
\label{var_qua}
	&\qquad\qquad\qquad
	 + \frac 1n \, D\, \int_0^t \int_{\X} 
			f(x)^2 \, \nub_u^n(\dif x) \, \dif u\,. 
\end{align}
\end{corollary}

\begin{proof}
Equation \eqref{renormalisation} is obtained by applying Lemma \ref{lem.renorm} 
with  $F(s,x)=x$, then $Z^{f,n}$ is $Z^{F,f,n}$ defined by
\eqref{eq.Z,F,f}. Moreover as the random measures $\tilde{N_1}$ and $\tilde{N_2}$ are independent, we have:
\begin{align*}
\crochet {Z^{f,n}}_t
	 = &
	 	\frac 1{n^2} \, \crochet {M^1}_t
	 	+\frac 1{n^2} \, \crochet {M^2}_t
\end{align*}
where $M^1$ and $M^2$ are defined at Proposition \ref{prop.martingale}. 
From this latter proposition and Lemma  \ref{lem.esp} we deduce the proof of the corollary.
\carre
\end{proof}

\bigskip

\begin{remark}
The infinitesimal generator of the renormalized process $(S^n_t,\bar\nu^n_t)_{t\geq 0}$ is:
\begin{multline*}
  \L^n \Phi(s,\nu)
  \eqdef 
  \bigl(D(\Sin-s)-k\,\mu(s,\nu) \bigr)\,
  \partial_s F(s, \crochet{\nu, f}) 		
  +
  \crochet{\nu ,\rhog(s,.) \,  f'} \,
  \partial_x F(s, \crochet{\nu, f}) 
\\
  \quad+  
  n\,\int_{\X} \lambda(s,x) \,
     \int_0^1 			  
	   \bigl[ 
	      F\bigl(s,
	         \crochet{\textstyle\nu - \frac 1n \delta_x + \frac 1n \delta_{\alpha\,x} 
	         + \frac 1n \delta_{(1-\alpha)\,x} \, , \, f}
		  \bigr) 
		  -
		  F(s, \crochet{\nu,f}) 
	   \bigr]	\;		
			  Q(\dif \alpha) \, \nu(\dif x) 
\\
	\quad
	+ n\, D\, \int_{\X} 
	 	\bigl[
			F\bigl(s,\crochet{\textstyle\nu- \frac 1n \delta_x \, , \, f} \bigr)
			-
			F(s,\left<\nu,f\right>) 
		\bigr]\,
			 \,	\nu(\dif x) 
\end{multline*}
for any $\Phi(s,\nu)=F(s,\crochet{\nu,f})$ with $F \in C_b^{1,1}(\RR^+ \times \RR)$ and $f \in C^{1}(\X)$. 
Note that this generator has the same ``substrat'' part than that of the initial generator 
\eqref{eq.generateur.infinitesimal} which again justifies the 
Remark \ref{remark.cv.nu.seule}.
\end{remark}

\bigskip

To prove  the uniqueness of the solution of the limit IDE, we have
to assume that the application $\lambda(s,x)$ is Lipschitz continous w.r.t. $s$ uniformly in  $x$:
\begin{align}
\label{hyp.lambda.lipschitz}
 \bigl|\lambda(s_1,x)-\lambda(s_2,x)\bigr|\leq k_\lambda \,|s_1-s_2|
\end{align}
for all $s_1,\,s_2 \geq 0$ and all $x\in\XX$.  
This hypothesis as well as Hypothesis \ref{hyp.rhog.lipschitz} 
will also be used to demonstrate the convergence of IBM, see Theorem \ref{theo.cv.ibm}.

\subsection{Convergence result}

\begin{theorem}[Convergence of the IBM towards the IDE]
\label{theo.cv.ibm}
Under the assumptions described above, the process $(S_t^n,\nub_t^n)_{t\geq 0}$ converges in distribution in the product space 
$\CC([0,T],\R_{+}) 
 \times
 \D([0,T],\MM_{F}(\X))$
towards the process $(S_t,\xi_t)_{t\geq 0}$ solution of:
\begin{align}
\label{eq.limite.substrat.faible}
	S_t  
	&= 
	S_{0}
	+
	\int_{0}^t\biggl[
	D\,(\Sin-S_u)-\frac kV \int_\X \rhog(S_u,x)\,\xi_u(\dif x)
	\biggr]\,\rmd u\,,
\\[1em]
\nonumber 
	\crochet{\xi_t,f}
	&=
	\crochet{\xi_0,f}
	+
	\int_{0}^t\biggl[
	\int_{\X} \rhog(S_u,x)\,f'(x) \,	\xi_u(\dif x)
\\
\nonumber 
	&\qquad
	+ \int_{\X} \int_0^1\lambda(S_u, x)\,  
	\Bigl[f(\alpha\, x)+ f((1-\alpha)\,x)-f(x)\Bigr] \, 	 			
			Q(\dif \alpha) \,\xi_u(\dif x) 
\\
\label{eq.limite.eid.faible}
	&\qquad
	- D\, \int_{\X} f(x) \,	\xi_u(\dif x) 
	\biggr]\,\rmd u\,,
\end{align}
for any $f \in C^{1}(\X)$.
\end{theorem}

The proof is in three steps\footnote{Note that our situation is simpler than that studied by \cite{roelly1986a} and \cite{meleard1993a} since in our case $\XX$ is compact: in fact in our case the weak  topology -- the smallest topology which makes the applications $\nu\to\crochet{\nu,f}$ continuous for any $f$ continuous and bounded -- and the vague topology -- the smallest topology which makes the applications $\nu\to\crochet{\nu,f}$ continuous for all $f$ continuous with compact support -- are identical.}:
first the uniqueness of the solution of the limit equation \eqref{eq.limite.substrat.faible}-\eqref{eq.limite.eid.faible}, 
second the tightness  (of the sequence of distribution) of  $\bar\nu^n$
and lastly the convergence in distribution of the sequence.

\subsubsection*{Step 1: uniqueness of the solution of   \eqref{eq.limite.substrat.faible}-\eqref{eq.limite.eid.faible}}

Let $(S_t,\xi_t)_{t\geq 0}$ be a solution of \eqref{eq.limite.substrat.faible}-\eqref{eq.limite.eid.faible}.
We first show that $(\xi_t)_t$ is of finite mass for all $t\geq 0$:
\begin{align*}
	\crochet{\xi_t, 1}
	& = 
	\crochet{\xi_0, 1}
	+ 
	\int_0^t \int_{\X} \int_0^1 \lambda(S_u, x)\, Q(\dif \alpha) \, 	 	
			\xi_u(\dif x)\, \dif u  
		- D\, \int_0^t \int_{\X} \xi_u(\dif x) \, \dif u\\ 
	& \leq  \left<\xi_0, 1\right>+(\bar \lambda - D) \int_0^t \left<\xi_u, 1\right> \, \dif u
\end{align*} 
and according to Gronwall's inequality: 
$\crochet{\xi_t, 1}  \leq \crochet{\xi_0, 1}  \, e^{(\bar \lambda - D)\,t} 
	< \infty
$.

We introduce the following norm on $\MM_{F}(\XX)$:
\[
  \normm{\nub}
  \eqdef
  \sup 
	\Bigl\{
		|\crochet{\nub,f}|
		\,;\,
		f \in C^1(\X),\, \norme{f}_\infty \leq 1,\,\norme{f'}_\infty \leq 1
	\Bigr\}
\]
and consider two solutions  $(S^1_t,\xi^1_t)_{t\geq 0}$ and $(S^2_t,\xi^2_t)_{t\geq 0}$ 
of  \eqref{eq.limite.substrat.faible}-\eqref{eq.limite.eid.faible}.

It was previously shown that $\xi^1_t$ and $\xi^2_t$  are of finite mass on  $\RR_+$, so we can define:
\[
  C_t
  \eqdef \sup_{0\leq u \leq t} \crochet{\xi^1_u+\xi^2_u , 1}
  \,.
\]
According to \eqref{eq.limite.eid.faible}, for any  $f \in C^1(\X)$ such that 
$\norme{f}_\infty \leq 1$ and $\norme{f'}_\infty \leq 1$ we have:
\begin{align*}
	|\crochet{\xi^1_t-\xi^2_t,f}|
	&\leq 
	\int_0^t \biggl|\int_\X
		f'(x) \, 
		\Bigl[	 			
	 	\rhog(S^1_u,x)\, [\xi^1_u(\dif x)-\xi^2_u(\dif x)] 
\\
		&\qquad\qquad\qquad
		- [\rhog(S^2_u,x)-\rhog(S^1_u,x)]\, \xi^2_u(\dif x)
		\Bigr]
	 	\biggr|\,\dif u 
\\
		&\qquad +	\int_0^t  \biggl|
			\int_\X \int_0^1 
				\left[f(\alpha\,x)+f((1-\alpha)\,x)-f(x) \right] \,
				Q(\dif \alpha)
\\
		& \qquad\qquad\qquad 
			\left[
				\lambda(S^1_u, x) \, [\xi^1_u(\dif x)-\xi^2_u(\dif x)]
				- [\lambda(S^2_u, x)-\lambda(S^1_u, x)]\, 
				\xi^2_u(\dif x)
			\right] \biggr| \, \dif u
\\
		 &\qquad
		   + D \, \int_0^t \biggl|
				\int_\X f(x) \, (\xi^1_u(\dif x)-\xi^2_u(\dif x))
			\biggr| \, \dif u
\\
	&\leq
	(\bar g +3\, \bar \lambda +D)\,
			\int_0^t  \normm{\xi^1_u-\xi^2_u} \, \dif u 
			+ C_t\,( k_g + 3 \, k_\lambda) \,
			 \int_0^t | S^1_u- S^2_u |\, \dif u .
\end{align*}
Taking the supremum over the functions $f$, we obtain:
\begin{align*}
	\normm{\xi^1_t-\xi^2_t}	
	&\leq
	(\bar g +3\, \bar \lambda +D)\,
	\int_0^t  \normm{\xi^1_u-\xi^2_u} \, \dif u
	+ C_t\,( k_g + 3 \, k_\lambda) \,  
		\int_0^t | S^1_u-S^2_u |\, \dif u
	\,.
\end{align*}
Moreover, from \eqref{eq.limite.substrat.faible} we get:
\begin{align*}
	&
	|S^1_t- S^2_t|
	\leq  
	D \, \int_0^t |S^1_u- S^2_u|\,\dif u  
\\
	&\quad	
	+ \frac kV \int_0^t \biggl|
		\int_\X 
		\Bigl(	 			
	 			\rhog(S^1_u,x)\, [\xi^1_u(\dif x)-\xi^2_u(\dif x)] 
	 			- 
				[\rhog(S^2_u,x)-\rhog(S^1_u,x)]
				\, \xi^2_u(\dif x)
		\Bigr)
	\biggr| \, \dif u 
\\
	&\quad	
	\leq 
	\Bigl(D+\frac kV \, C_t\,k_g\Bigr) \, 
	\int_0^t |S^1_u- S^2_u| \, \dif u
	+ \frac kV \,\bar g \int_0^t \normm{\xi^1_u-\xi^2_u} \, \dif u \,.
\end{align*}
We define:
\begin{align*}
	M_t
	\eqdef
	\max \left\{\bar g +3\, \bar \lambda +D+ \frac kV \,\bar g
		\, , \, C_t\,( k_g + 3 \, k_\lambda)+D+\frac kV \, C_t\,k_g\right\}
\end{align*}
hence:
\begin{align*}
	\normm{\xi^1_t-\xi^2_t}
	+|S^1_t-S^2_t|
	& \leq
	M_t \, \int_0^t \Bigl( 
		\normm{\xi^1_u-\xi^2_u}+|S^1_u-S^2_u|	
	\bigr) \, \dif u 
\end{align*}
Finally from Gronwall's inequality we get 
$\normm{\xi^1_t-\xi^2_t}+|S^1_t-S^2_t|=0$ for all $t\geq 0$, hence  $\xi^1_t = \xi^2_t$ and $S^1_t=S^2_t$.

\subsubsection*{Step 2: tightness of $(\bar\nu^n)_{n\geq 0}$}

The tightness of  
$\bar\nu^n$ is equivalent to the fact that from any subsequence one can extract a subsequence that converges in distribution in the space $\D([0,T],\MM_{F}(\X))$.
According to \citet[Th. 2.1]{roelly1986a} this amounts to proving the tightness 
of  $\crochet{\bar\nu^n,f}$ in $\D([0,T],\R)$ for all $f$ in a set dense in 
$\C(\XX)$, here we will consider  $f\in \C^1(\XX)$. To prove the latter result, it is sufficient to check the following Aldous-Rebolledo criteria \citep[Cor. 2.3.3]{joffe1986a}:
\begin{enumerate}
\item 
The sequence $(\langle \nub_t^n,f \rangle)_{n\geq 0}$ is tight for any $t\geq 0$.

\item 
Consider the following semimartingale decomposition:
\[
	\crochet{\nub_t^n,f}
	=
	\crochet{\nub_0^n,f}
	+
	A_t^n+Z_t^n\,.
\]
where $A_t^n$ is of finite variation and $Z_t^n$ is a martingale.
For all $t>0$, $\epsilon>0$, $\eta>0$ there exists $n_{0}$ such that for any sequence 
$\tau_n$ of stopping times with  $\tau_{n}\leq t$ we have:
\begin{align}
\label{eq.AR.2.1}
  \sup_{n\geq n_0} \sup_{\theta\in[0,\delta]}
  \P\Big(
    \big|
       A^n_{\tau_n+ \theta} - A^n_{\tau_n}
    \big|
    \geq \eta
  \Big)
  &
  \leq \epsilon\,,
\\
\label{eq.AR.2.2}
  \sup_{n\geq n_0} \sup_{\theta\in[0,\delta]}
  \P\Big(
    \big|
       \crochet{Z^n}_{\tau_n+ \theta} - \crochet{Z^n}_{\tau_n}
    \big|
    \geq \eta
  \Big)
  &
  \leq \epsilon\,.
\end{align} 
\end{enumerate}

\subsubsection*{\it Proof of \fenumi}

For any $K>0$,
\begin{align*}
	\P\bigl(|\crochet{\nub_t^n,f}| \geq K\bigr)
	&\leq 
	\frac{1}{K}\,\norme{f}_\infty \,
	  \sup_{n\in\N} \E\bigl(\crochet{\nub_t^n,1}\bigr)
\end{align*}
and using Lemma \ref{lem.esp}, we deduce  \fenumi.

\subsubsection*{\it Proof of {\rm(\textit{ii})}}

\begin{align*}
	A_t^n
	&=  
	\int_0^t  \crochet{\nub_u^n,\rhog(S_u^n,.)\,f'} \, \dif u 
\\
	&\qquad
	+ \int_0^t \int_{\X} \int_0^1  
		\lambda(S_u^n, x) \,  
		\bigl[ f(\alpha\,x)+f((1-\alpha)\,x)-f(x) \bigr] \, 
		Q(\dif \alpha) \,\nub_u^n(\dif x) \, \dif u
\\
	&\qquad 
	- D\,\int_0^t\int_{\X} f(x)\,\nub_u^n(\dif x) \, \dif u
\end{align*}
hence, according to Lemma \ref{lem.esp}:
\begin{align*}
	\EE | A_{\tau_n+\theta}^n-A_{\tau_n}^n|
	&\leq
	(
		\norme{f'}_\infty\, \bar{g}
		+ 3\, \norme{f}_\infty\,\lambdab
		+ D\,\norme{f}_\infty
	)
	\, C_{t,1}\, \theta\,.
\end{align*}
Using  \eqref{var_qua}, we also have:
\begin{align*}
	\EE | \crochet{Z^n}_{\tau_n+\theta}-\crochet{Z^n}_{\tau_n}|
	\leq 
	\frac{1}{n}\,\left(9 \, \bar \lambda+D \right) \, 
		\norme{f}_\infty^2\, C_{t,1} \, \theta\,.
\end{align*}
Hence
$\EE| A_{\tau_n+\theta}^n-A_{\tau_n}^n|+\EE | \crochet{Z^n}_{\tau_n+\theta}-\crochet{Z^n}_{\tau_n} | 
	\leq  C \, \theta$ 
and we obtain  \fenumii\ from the Markov inequality.

\bigskip

In conclusion, from the Aldous-Rebolledo criteria, the sequence $(\nub^n)_{n\geq 0}$ is tight.

\subsubsection*{Step 3: convergence of the sequence $(\bar\nu^n)_{n\in\N}$}

To conclude the proof of the theorem it is suffice to show that the sequence $(\bar\nu^n)_{n\in\N}$ has a unique accumulation point and that this point is equal to $\xi$ described in Step 1. In order to characterize $\xi$, the  solution of \eqref{eq.limite.eid.faible}, we introduce, for any given $f \in C^{1}(\X)$, the following function defined for all
$\zeta\in\D([0,T],\M_F(\X))$:
\begin{align}
\nonumber
  \Psi_{t}(\zeta)
  &\eqdef
  \crochet{\zeta_{t},f}-\crochet{\zeta_{0},f}
  -\int_{0}^t \biggl[
		\int_{\X} \rhog(S^\zeta_u,x)\,f'(x) \,	\zeta_u(\dif x)
\\
\nonumber
	&\qquad
	+ \int_{\X} \int_0^1\lambda(S^\zeta_u, x)\,  
	\bigl[f(\alpha\, x)+ f((1-\alpha)\,x)-f(x)\bigr] \, 	 			
			Q(\dif \alpha) \,\zeta_u(\dif x) 
\\
\label{eq.Psi}
	&\qquad
	- D\, \int_{\X} f(x) \,	\zeta_u(\dif x) 
		\biggr]\,\rmd u
\end{align}
where $S_t^\zeta$ is defined by:
\begin{align}
\label{eq.Psi.2}
	S_t^\zeta
	&\eqdef
	S_0
	+
	\int_0^t \Bigl(
		D\,(\Sin-S_u^\zeta)
		-\frac kV \int_\X \rhog(S_u^\zeta,x)\,\zeta_u(\dif x)
	\Bigr) \, \dif u\,.
\end{align}
Hence, if $\Psi_{t}(\zeta)=0$ for all $t\geq 0$ and all $f \in C^{1}(\X)$ then $(S^\zeta,\zeta)=(S,\xi)$ where $(S,\xi)$ is the unique solution of  \eqref{eq.limite.substrat.faible}-\eqref{eq.limite.eid.faible}.

\bigskip

We consider a subsequence $\bar\nu^{n'}$ of $\bar\nu^{n}$ which converges in distribution in the space  $\D([0,T],\M_F(\X))$ and $\tilde \nu$ its limit.

\subsubsection*{\it Sub-step 3.1: A.s. continuity of the limit~$\tilde \nu$.}

\begin{lemma}
$\tilde \nu(\omega)\in \CC([0,T],\M_F(\X))$ for  all $\omega\in\Omega$ a.s.
\end{lemma}

\begin{proof}
For any  $f\in\C(\X)$ such that $\norme{f}_\infty \leq 1$:
\begin{align*}
	\bigl|
	 	\crochet{\nub_t^{n'},f}-\crochet{\nub_{t^-}^{n'},f}
	\bigr| 
	\leq 
	\frac{1}{n'}\, 
	\bigl| 
		\crochet{\nu_t^{n'},1}-\crochet{\nu_{t^-}^{n'},1}
	\bigr|\,.
\end{align*}
But $| \crochet{\nu_t^{n'},1}-\crochet{\nu_{t^-}^{n'},1}|$ represents 
the difference between the number of individuals in $\nu_t^{n'}$ and in $\nu_{t^-}^{n'}$, which is at most 1. Hence:
\begin{align*}
	\sup_{t \in [0,T]} \, 
	\normtv{\nub_t^{n'}-\nub_{t^-}^{n'}}
	\leq \frac{1}{n'}
\end{align*}
which proves that the limit process $\tilde\nu$ is a.s. continuous
\cite[Th. 10.2 p. 148]{ethier1986a}  as the Prokhorov metric is dominated by the total variation metric.
\end{proof}
\carre

\subsubsection*{\it Sub-step 3.2: Continuity of $\zeta\to\Psi_{t}(\zeta)$ in any $\zeta$ continuous.}

\begin{lemma}
\label{lemma.Psi.continue}
For any given  $t\in[0,T]$ and $f\in C^{1}(\X)$, the function  $\Psi_{t}$ defined by
\eqref{eq.Psi} is continuous from  $\DD([0,T],\MM_{F}(\X))$ with values in  $\R$ in any point $\zeta\in\CC([0,T],\MM_{F}(\X))$. 
\end{lemma}

\proof
Consider a sequence  $(\zeta^{n})_{n\in\N}$ which converges towards  $\zeta$ in
$\DD([0,T],\MM_{F}(\X))$ with respect to the Skorohod topology. As the limit $\zeta$ is continuous we have that  $\zeta^{n}$ converges to $\zeta$ with the uniform topology:
\begin{align}
\label{eq.lemma.Psi.continue.a}
   \sup_{0\leq t\leq T} 
   \dpr({\zeta_t^{n},\zeta_{t}})
   \cv_{n\to\infty} 0
\end{align}
where $\dpr$ is the Prokhorov metric (see Appendix \ref{appendix.skorohod}).

The functions $\lambda(s,x)$ and $\rhog(s,x)$ are Lipschitz continuous functions w.r.t. $s$ uniformly in $x$ and also bounded, see \eqref{hyp.lambda.lipschitz} and \eqref{hyp.rhog.lipschitz}, so from \eqref{eq.Psi.2} we can easily check that:
\begin{align*}
	|S_t^{\zeta^{n}}-S_t ^\zeta|
	&\leq
	C\,\int_0^t \Bigl(
		|S_u^{\zeta^{n}}-S_u ^\zeta|
		+ \Bigl|
     		\int_\X \rhog(S_u^{\zeta^{n}},x)\,
				[\zeta^{n}_u(\dif x)-\zeta_u(\dif x)]
\\
	&\qquad\qquad\qquad\qquad\qquad\qquad
     		-
			\int_\X [\rhog(S_u^\zeta,x)-\rhog(S_u^{\zeta^{n}},x)]
			        \,\zeta_u(\dif x)
		\Bigr|
	\Bigr) \, \dif u
\\ 
	&\leq
	C\,\int_0^t \Bigl(
		|S_u^{\zeta^{n}}-S_u ^\zeta|
		+ |\crochet{\zeta^{n}_u-\zeta_u,1}|\Bigr) \, \dif u
\end{align*}
and the Gronwall's inequality leads to:
\begin{align}
\label {eq.lemma.Psi.continue.b} 
  |S_t^{\zeta^{n}}-S_t ^\zeta|
  \leq 
  C\,\int_{0}^t |\crochet{\zeta^n_{u}-\zeta_{u},1}|\,\rmd u\,.
\end{align}
Here and in the rest of the proof the constant $ C $ will depend only on $ T $, $ f $ and on the parameters of the models.
Hence, from \eqref{eq.Psi}:
\begin{align*}
	|\Psi_t(\zeta^{n})-\Psi_t(\zeta)|
	&\leq 
	C\,
	\Bigl[
	  	|\crochet{\zeta_t^{n}- \zeta_t,1}|
		+
		|\crochet{\zeta_0^{n}- \zeta_0,1}|
\\
   &\qquad\qquad\qquad
		+
		\int_0^t |S_u^{\zeta^{n}}-S_u ^\zeta| \, \dif u
		+
		\int_0^t |\crochet{\zeta_u^{n}-\zeta_u,1}| \, \dif u
	\Bigr]
\\
   &\leq
		C\,\sup_{0\leq t\leq T}|\crochet{\zeta_t^{n}-\zeta_t,1}|\,.
\end{align*}
Let $\delta_{t}=\dpr(\zeta_t^{n},\zeta_t)$, by definition of the Prokhorov metric:
\begin{align*}
  \zeta_t^{n}(\XX)-\zeta_t(\XX^{\delta_{t}})&\leq \delta_{t}\,,
  &
  \zeta_t(\XX)-\zeta_t^{n}(\XX^{\delta_{t}})&\leq \delta_{t}\,,
\end{align*}
but $\XX^{\delta_{t}}=\XX$ hence $|\zeta_t^{n}(\XX)-\zeta_t(\XX)|\leq \delta_{t}$. Note finally that $|\zeta_t^{n}(\XX)-\zeta_t(\XX)| = |\crochet{\zeta_t^{n}-\zeta_t,1}|$, so we get:
\[
   |\Psi_t(\zeta^{n})-\Psi_t(\zeta)| 
   \leq C\,
   \sup_{0\leq t\leq T}\dpr({\zeta_t^{n},\zeta_t})
\]
which tends to zero.
\carre


\subsubsection*{\it Sub-step 3.3: Convergence in distribution  of $\Psi_{t}(\bar\nu^{n'})$ to $\Psi_{t}(\tilde\nu)$.}

The sequence  $\bar\nu^{n'}$ converges in distribution to  $\tilde\nu$ and $\tilde\nu(\omega)\in\CC([0,T],\MM_{F}(\XX))$;
moreover the application  $\Psi_{t}$ is continuous in any point of $\CC([0,T],\MM_{F}(\XX))$, thus according to the continuous mapping theorem \cite[Th. 2.7 p. 21]{billingsley1968a} we get:
\begin{align} 
\label{eq.cv.loi.Psi}
  \Psi_{t}(\bar\nu^{n'})
  \xrightarrow[n\to\infty]{\textrm{loi}}
  \Psi_{t}(\tilde\nu)\,.
\end{align}

\subsubsection*{\it Sub-step 3.4: $\tilde\nu=\xi$ a.s.}

From \eqref{renormalisation}, for any $n\geq 0$ we have:
\begin{align*}
	\Psi_t(\nub^n) =  Z_t^{f,n}
\end{align*}
where $Z_t^{f,n}$ is defined by  \eqref{eq.Zfn}. Also, \eqref{var_qua} gives:
\begin{align*}
	\EE(|Z_t^{f,n}|^2)
	&= 
	\EE \crochet{Z^{f,n}}_t 
	\leq \frac{1}{n}\,(9\,\bar\lambda+D) \, \norm{f}_{\infty}^2\,C_{t,1} \, t\,.
\end{align*}
Hence $\Psi_{t}(\nub^n)$ converges to $0$ in $L^2$ but also in $L^1$. 
Furthermore, we easily show that:
\begin{align*}
|\Psi_{t}(\zeta)|
	\leq & C_{f,t} \sup_{0 \leq u \leq t} \crochet{\zeta_u,1}
\end{align*}
moreover, from Lemma \ref{lem.esp}, $(\Psi_{t}(\nub^{n'}))_{n'}$ is uniformly
integrable. The dominated convergence theorem and \eqref{eq.cv.loi.Psi} imply:
\[
  0
  = \lim_{n' \to \infty} \EE |\Psi_t(\nub^{n'})|
  = \EE |\Psi_t(\tilde \nu)|\,.
\]
So $\Psi_t(\tilde \nu)=0$ a.s. and $\tilde \nu$ is a.s. equal to $\xi$ where $(S,\xi)$ 
is the unique solution of \eqref{eq.limite.substrat.faible}-\eqref{eq.limite.eid.faible}.

\bigskip

This last step concludes the proof of Theorem \ref{theo.cv.ibm}.

\subsection{Links with deterministic models}
\label{subsec.modeles.deterministes}

Equation \eqref{eq.limite.eid.faible}  is actually a weak version of an 
integro-differential equation that can be easily identified.
Indeed suppose that the solution $\xi_t$ of Equation \eqref{eq.limite.eid.faible}
admits a density $p_t(x)\,\rmd x=\xi_{t}(\rmd x)$, and that $Q(\rmd \alpha)=q(\alpha)\,\rmd \alpha$, then the system of equations \eqref{eq.limite.substrat.faible}-\eqref{eq.limite.eid.faible} is a weak version of the following system:
\begin{align} 
\label{eq.limite.substrat.fort}
	&
	\frac{\rmd}{\rmd t} S_t  = 
	D\,(\Sin-S_t)-\frac kV \int_\X \rhog(S_t,x)\,p_t(x)\,\dif x\,,
\\
\nonumber
	&
	\frac{\partial}{\partial t} p_t(x)
	+\frac{\partial}{\partial x} \bigl( \rhog(S_t,x)\,p_t(x)\bigr)
	+ \bigl(\lambda(S_t, x)+D \bigr)\,p_t(x)
\\
\label{eq.limite.eid.fort}
  	&\qquad\qquad\qquad\qquad\qquad\qquad\qquad\qquad
	=  
	2\,\int_\X 
	\frac{\lambda(S_t, z)}{z}\,q\left(\frac xz \right)\,p_t(z)\,\dif z\,.
\end{align}
In fact, this is the population balance equation  introduced by \cite{fredrickson1967a} and \cite{ramkrishna1979a} for growth-fragmentation models.

\bigskip

It is easy to link  the model \eqref{eq.limite.substrat.fort}-\eqref{eq.limite.eid.fort}
to the classic chemostat model. Indeed suppose that
the growth function $x \mapsto \rhog(s,x)$ is proportional to $ x $, i.e.:
\[
  \rhog(s,x)=\tilde\mu(s)\,x\,.
\]
The results presented now are formal insofar as a linear growth function 
does not verify the assumptions made in this article.
We introduce the bacterial concentration:
\[
  Y_t \eqdef \frac1V \,\int_\X x\, p_t(x) \, \dif x\,.
\] 
As $\sup_{0\leq t\leq T}\crochet{p_{t},1}<\infty$, from \eqref{eq.limite.eid.fort}:
\begin{multline*}
	\frac{\dif}{\dif t} Y_t
	- \frac1V \int_\X x \, 
			\frac{\partial}{\partial x}\bigl( \rhog(S_t,x)\,p_t(x)\bigr) \, \dif x
	+ \frac1V \int_\X x \, \lambda(S_t, x) \, p_t(x)\,\dif x
	+ D\, Y_t
\\
	= \frac2V\, \int_\X x \, 
				\int_\X 
				\frac{\lambda(S_t, z)}{z}\,q(x/z)\,p_t(z)\,
				\dif z \, \dif x\,,
\end{multline*}
but
\begin{align*}
&\int_\X x \, 
	\int_\X 
		\frac{\lambda(S_t, z)}{z}\,
		q(x/z)\,p_t(z)\,\dif z \, \dif x
		= 		
		\int_\X \int_0^1			 
			z\, \lambda(S_t,z)\, \alpha \,q\left(\alpha \right)\,p_t(z)\,\dif \alpha \, \dif z
\\
		&\qquad\qquad =  \int_\X \int_0^1			 
			z\, \lambda(S_t, z)\, \alpha \,q\left(1-\alpha \right)\,p_t(z)\,\dif \alpha \, \dif z
			\tag{by symmetry of $q$}
\\
		&\qquad\qquad =	\int_\X \int_0^1			 
			z\, \lambda(S_t, z)\, (1-\alpha) \,q\left(\alpha \right)\,p_t(z)\,\dif \alpha \, \dif z
\\
		&\qquad\qquad = - \int_\X \int_0^1			 
			z\, \lambda(S_t, z)\, \alpha \,q\left(\alpha \right)\,p_t(z)\,\dif \alpha \, \dif z
		+ \int_\X 			 
			z\, \lambda(S_t, z)\,p_t(z)\, \dif z
\end{align*}
thus:
\begin{align*}
 2\, \int_\X x \, 
	\int_\X 
	\frac{\lambda(S_t, z)}{z}\,q(x/z)\,p_t(z)\,\dif z \, \dif x
		= \int_\X 			 
			z\, \lambda(S_t, z)\,p_t(z)\, \dif z.
\end{align*}
The function $x\mapsto p_t(x)$ is the population density at time $t$. On the one hand $p_0(x)$ has compact support. On the other hand the growth of each bacterium is defined by a differential equation whose right-hand side is bounded by a linear function in $x$, uniformly in $s$. Hence for all $t\leq T$, we can uniformly bound the mass of all the bacteria and $p_{t}(x)$ has a compact support, i.e. there exists $\mmax$ such that the support of $p_{t}(x)$ is included in $[0,\mmax]$ with $p_{t}(\mmax)=0$, so we choose $\X=[0,\mmax]$. Moreover $\rhog(S_t,0)=0$ hence:
\begin{align*}
	\int_\X 
		x \, \frac{\partial}{\partial x}\bigl( \rhog(S_t,x)\,p_t(x)\bigr) 
	\, \dif x
	= 
	-\int_\X  \rhog(S_t,x)\,p_t(x) \, \dif x.
\end{align*}
Finally:
\begin{align*}
	\frac{\dif}{\dif t} Y_t 
	& = 
		\frac1V \int_\X  \rhog(S_t,x)\,p_t(x) \, \dif x
		-D\, Y_t
	= \tilde\mu(S_{t})\,Y_{t} -D\, Y_t\,.
\end{align*}
We deduce that the concentrations $(Y_{t},S_{t})_{t\geq 0}$  of biomass and substrate are  the solution of the following closed system of ordinary differential equations:
\begin{equation}
\label{eq.chemostat.edo}
	\begin{split}
	\dot Y_t 	& = \bigl(\tilde\mu(S_t)-D \bigr)\, Y_t\,,
	\\
	\dot S_t	& = D\,(\Sin-S_t)-k \,\tilde\mu(S_t)\, Y_t\,.
	\end{split}
\end{equation}
which is none other than the classic  chemostat equation \citep{smith1995a}.

\section{Simulations}
\label{sec.simulations}

In this section we compare the behavior of the individual-based model (IBM) and two deterministic models: the integro-differential equation (IDE) \eqref{eq.limite.substrat.fort}-\eqref{eq.limite.eid.fort} and  classic chemostat model, represented by the ordinary differential equation \eqref{eq.chemostat.edo} (ODE).
Simulations of the IBM were performed following Algorithm \ref{algo.ibm}.
The resolution of the integro-differential equation was made following the numerical scheme given in Appendix \ref{appendix.schema.num}, with a discretization step in the mass space  of $\Delta x = 2 \times 10^{-7}$ and a discretization step in time of  $\Delta t = 5 \times 10^{-4}$.

\subsection{Simulation parameters}

In the simulations proposed in this section, the division rate of an individual is given by the following function:
\begin{align*}
	\lambda(s,x)
	& =
	\frac{\bar\lambda}{\log \bigl((\mmax-\mdiv) \, p_\lambda +1 \bigr)} 
	\,
	\log \bigl((x-\mdiv) \, p_\lambda +1 \bigr) \, 1_{\{x \geq \mdiv\}}
\end{align*}
which does not depend on the substrate concentration.

The  division kernel $Q(\dif \alpha)=q(\alpha)\,\dif \alpha$ is given by a symmetric beta distribution:
\begin{align*}
	q(\alpha)
	& = 
	\frac{1}{B(p_\beta)}\,
		\bigl( \alpha \,(1-\alpha) \bigr)^{p_\beta-1}
\end{align*}
where $B(p_\beta)
	= \int_0^1 \bigl( \alpha \,(1-\alpha) \bigr)^{p_\beta-1} \, \dif \alpha$
is a normalizing constant.

Individual growth follows a Gompertz model, with a growth rate depending on the substrate concentration:
\begin{align*}
	g(s,x)
	& = 
	\rmax\,\frac{s}{k_r+s} \,
	\log\Big(\frac{\mmax}{x}\Big)\,x\,.
\end{align*}
The masses of individuals at the initial time are sampled according to the following probability density function:
\begin{align}
\label{eq.d}
d(x)
	& =	
		\Biggl(
			\frac{x-0.0005}{0.00025}
			\,\left(1-\frac{x-0.0005}{0.00025}\right)
		\Biggr)^5 \,
		1_{\{0.0005 < x < 0.00075\}}\,.
\end{align}
This initial density will show a transient phenomenon that cannot be reproduced by the classical chemostat model described in terms of ordinary differential equations \eqref{eq.chemostat.edo}, see Figure \ref{fig.edo.ibm.eid}.

The simulations were performed using the parameters in Table  \ref{table.parametres}. The parameters $V$, $N_0$ and $D$ will be specified for each simulation.

\begin{table}[h]
\begin{center}
\begin{tabular}{|c|c|}
	\hline
    Parameters & Values \\
    \hline
    $S_0$				&	5 mg/l \\
    $\Sin$				&	10 mg/l \\
    $\mmax$				&	0.001 mg \\	
    $\mdiv$				&	0.0004 mg \\
	$\bar\lambda$		&	1 h$^{-1}$\\
    $p_\lambda$			&	1000 \\
    $p_\beta$			&	7 \\
    $\rmax$				&	1 h$^{-1}$\\
    $k_r$				&	10 mg/l\\
    $k$					&	1\\
    \hline
\end{tabular}
\end{center}
\caption{Simulation parameters.}
\label{table.parametres}
\end{table}

\subsection{Comparison of the IBM and the IDE}

To illustrate the convergence in large population asymptotic of the IBM to the IDE, we performed simulations at different levels of population size. To this end we vary the volume of the chemostat and the number of individuals at the initial time. We considered three cases:
\begin{enumerate}
\item small size: $V=0.05$ l and $N_0=100$,
\item medium size: $V=0.5$ l and $N_0=1000$,
\item large size: $V=5$ l and $N_0=10000$.
\end{enumerate} 
In each of these three cases we simulate:
\begin{itemize}
\item 60 independent runs of the IBM;
\item the numerical approximation of \eqref{eq.limite.substrat.fort}-\eqref{eq.limite.eid.fort} using the finite difference schemes 
detailed in Appendix \ref{appendix.schema.num}
\end{itemize}
with the same initial biomass concentration distribution.

\bigskip

The convergence of IBM to EID is clearly illustrated
in Figure \ref{evol.taille.concentrations}  where the evolutions of the population size, of the biomass concentration, and of the substrate concentration are represented.

In Figure \ref{fig.evol.eid}  the time evolution of the normalized mass distribution is depicted, i.e. the normalized solution of the IDE \eqref{eq.limite.eid.fort}. We have represented the simulation until time $T=10$ (h) to illustrate the transient phenomenon due to the choice of the initial distribution \eqref{eq.d}: after a few time iterations this  distribution is bimodal; the upper mode (large mass) grows in mass and disappears   before time $T=10$ (h). The lower mode (small mass) corresponds to the mass of the bacteria resulting from the division; the upper mode corresponds to the mass of the bacteria from the initial bacteria before their division.  Thus, the upper mode is set to disappear quickly by division or by up-take. The IBM realizes this phenomenon, see
Figure \ref{repartition.masse}. In contrast, the classical chemostat model presented below, see Equation \eqref{eq.chemostat.edo}, cannot account for this phenomenon.

Figure \ref{repartition.masse} presents this  normalized mass distribution at three different instants, $t=1,\,4,\, 80$ (h),  and the simulation of the IDE is compared 
to 60 independent runs of the IBM, again for the three levels of population sizes described above. Depending on whether the population is large, medium or small, 
we needed to adapt the number of bins of the histograms so that the resulting graphics are clear. The convergence of the IBM solution to the IDE in large population limit can be observed.

In conclusion, the IBM converges in large population limit to the IDE and variability ``around'' the asymptotic model is relatively large in small or medium population size; note that there is no reason why the IDE represents the mean value of the IBM.

\begin{figure}
\begin{center}
\begin{tabular}{ccc}
\includegraphics[width=4.5cm]{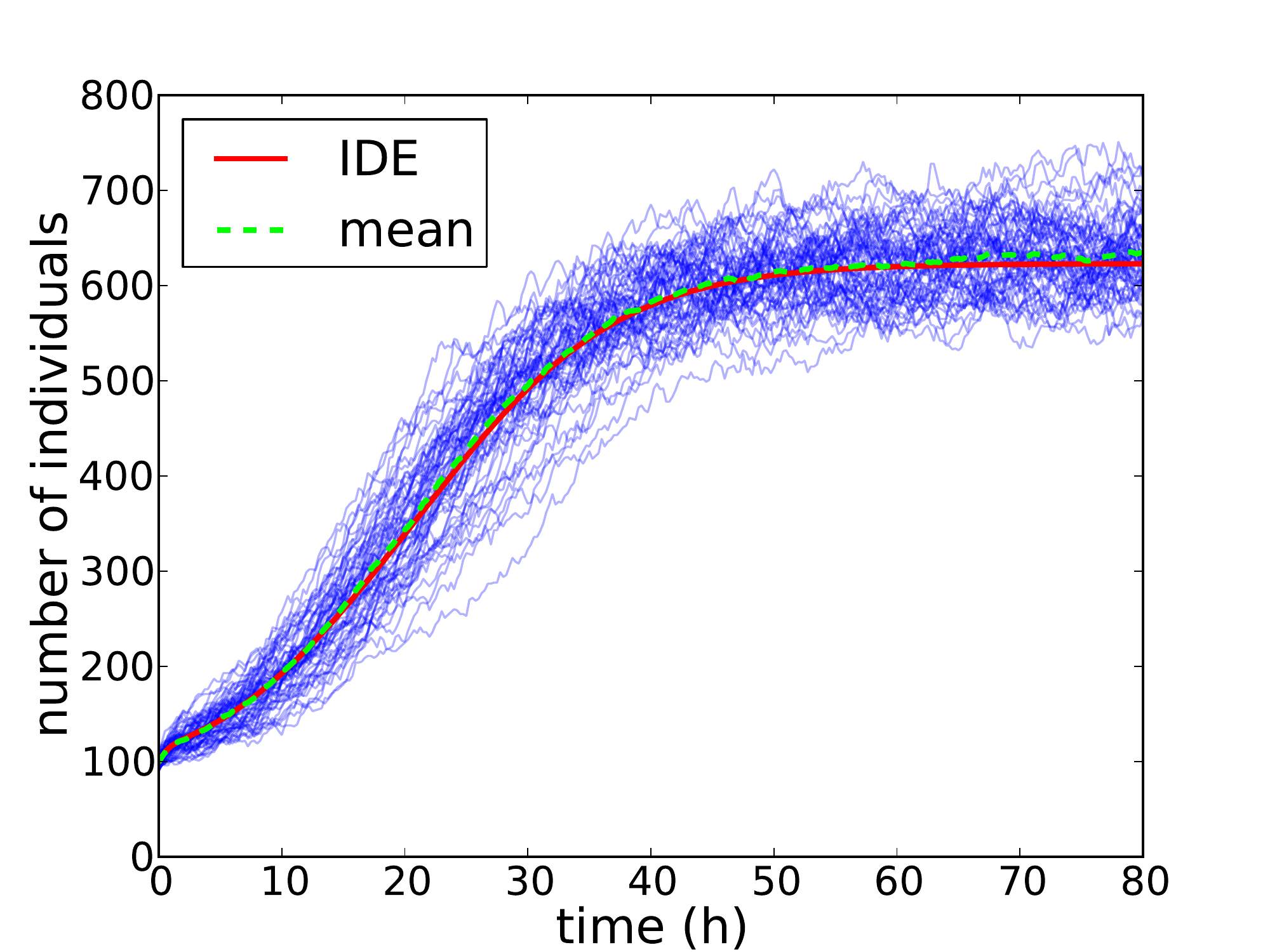}
&
\includegraphics[width=4.5cm]{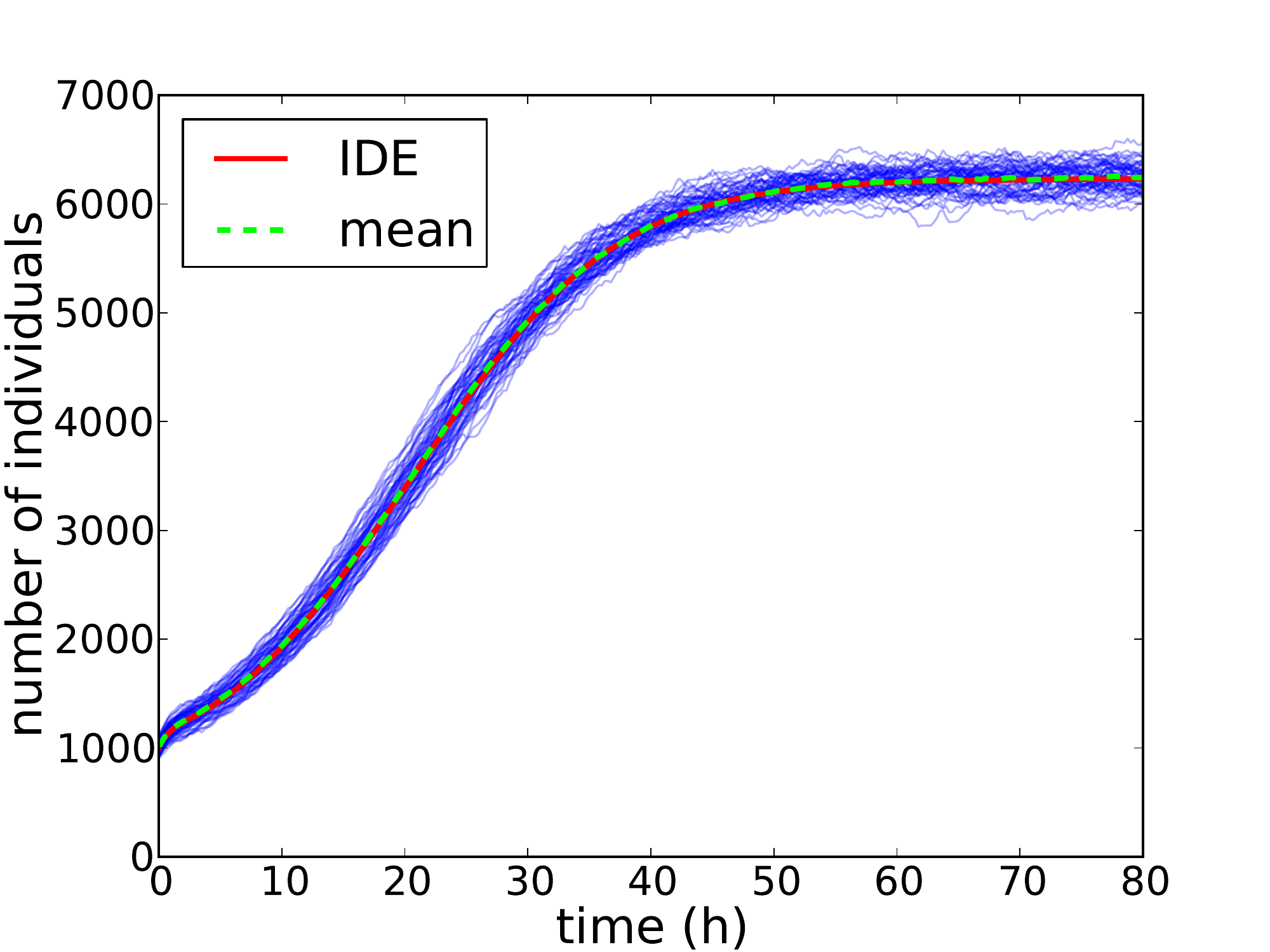}
&
\includegraphics[width=4.5cm]{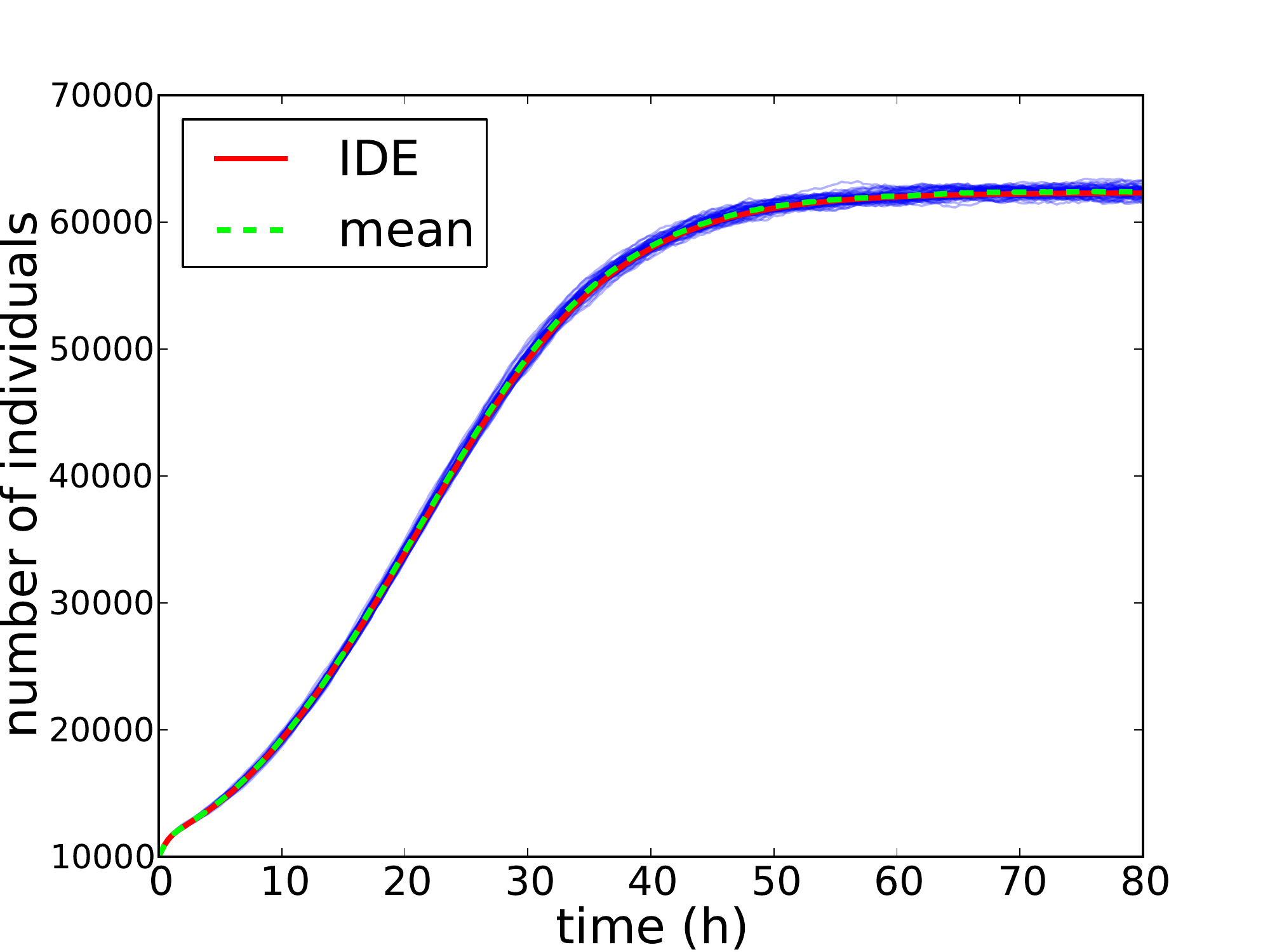}
\\ 
\includegraphics[width=4.5cm]{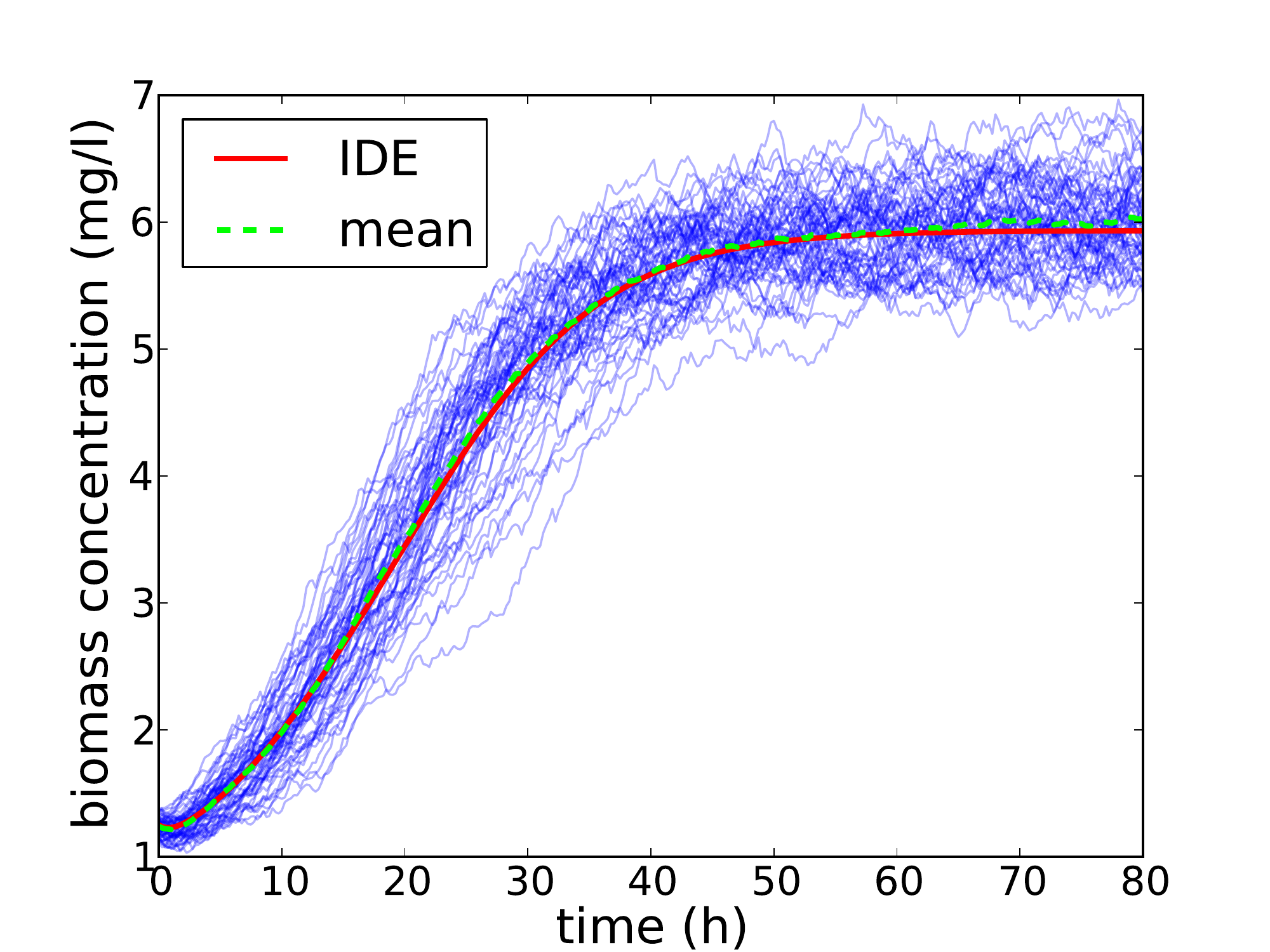}
&
\includegraphics[width=4.5cm]{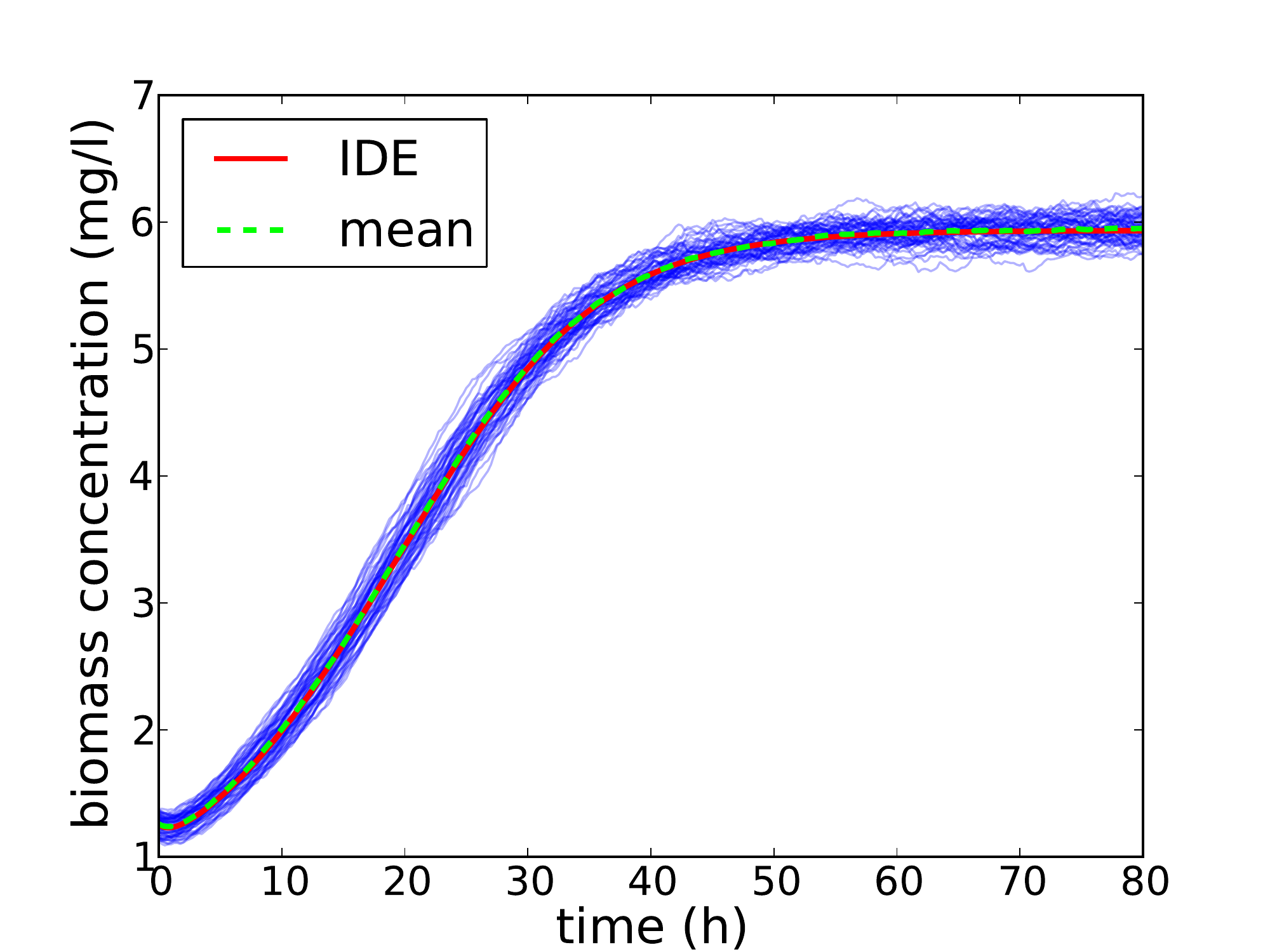}
&
\includegraphics[width=4.5cm]{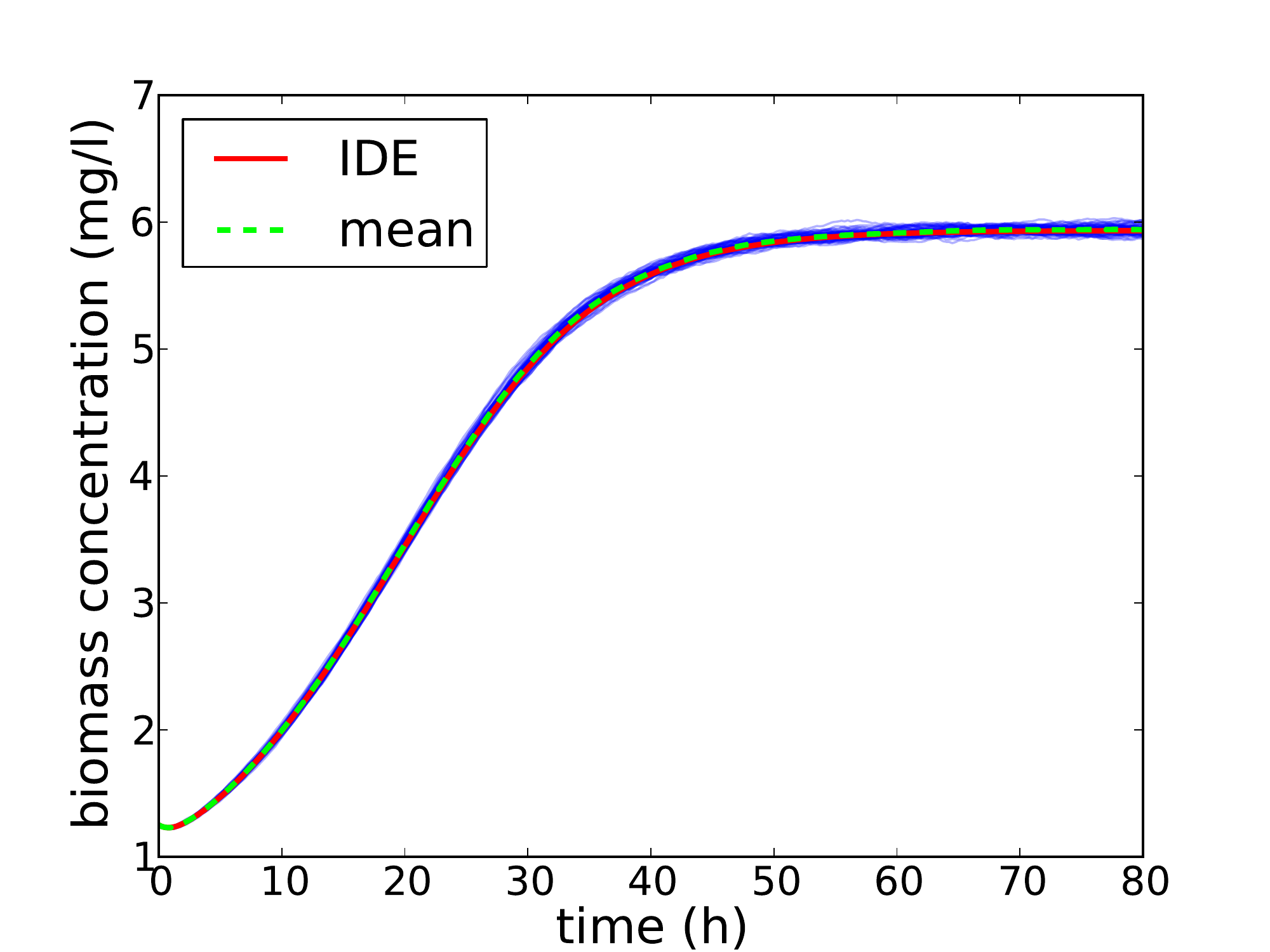}
\\ 
\includegraphics[width=4.5cm]{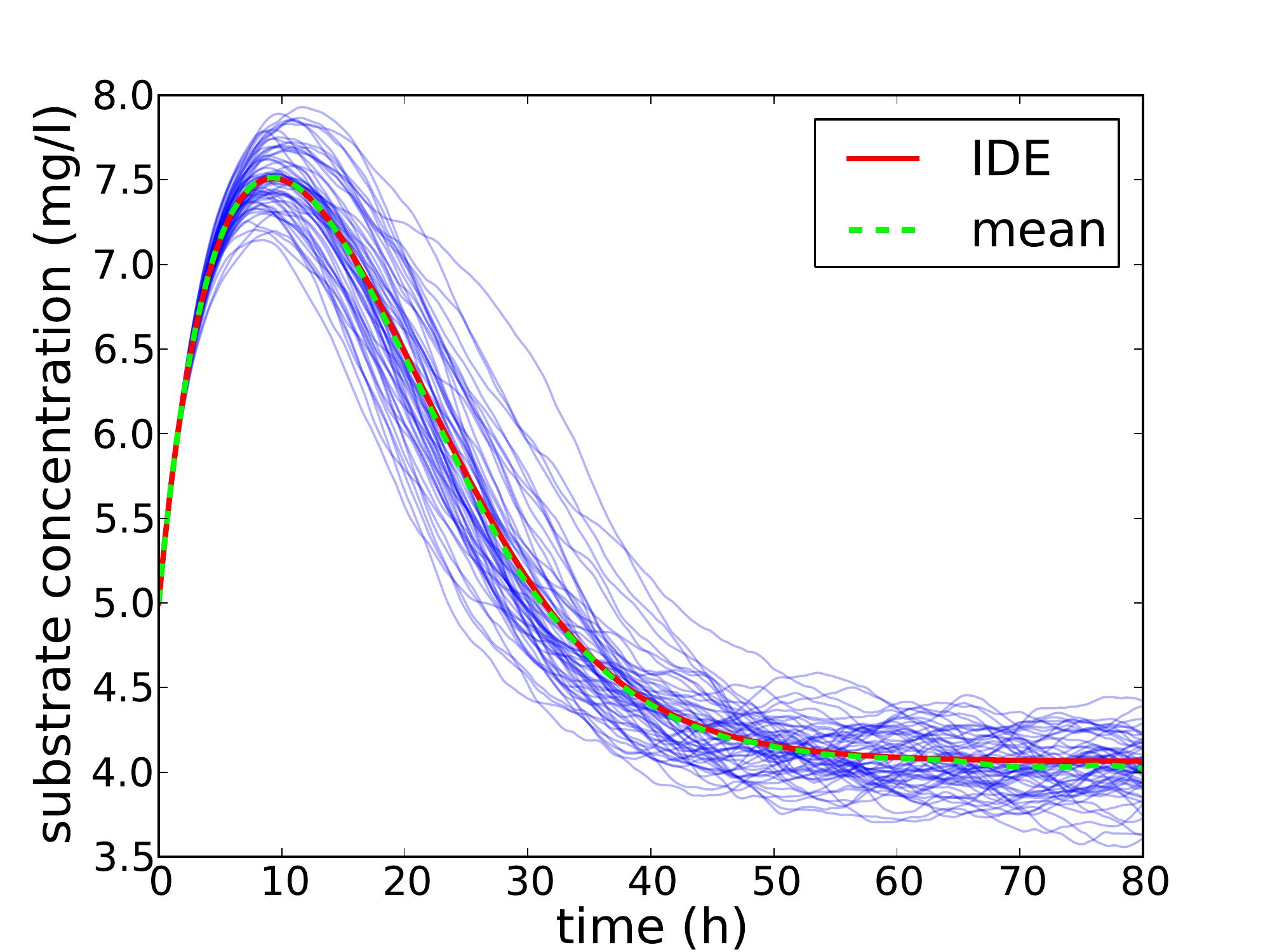}
&
\includegraphics[width=4.5cm]{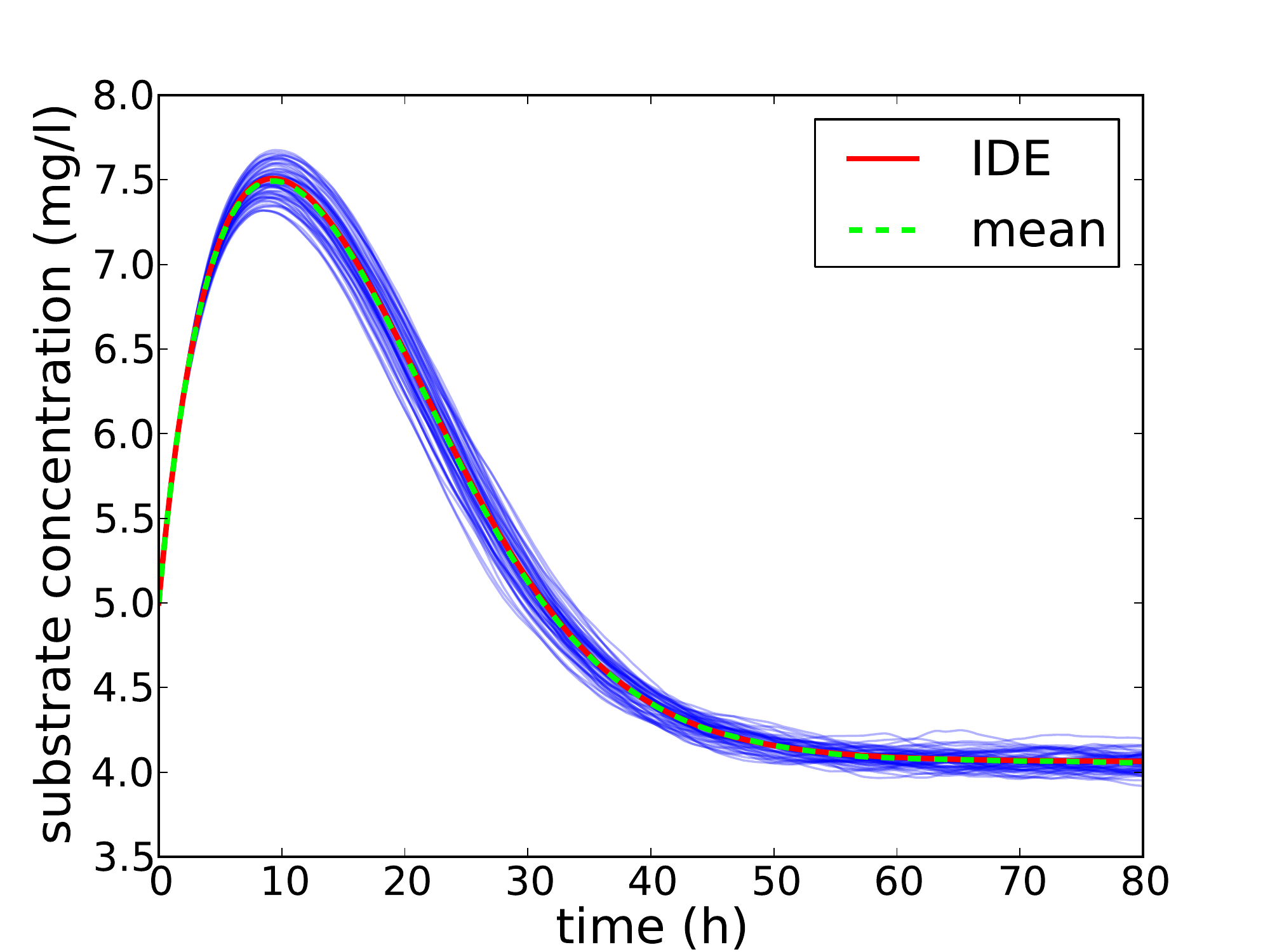}
&
\includegraphics[width=4.5cm]{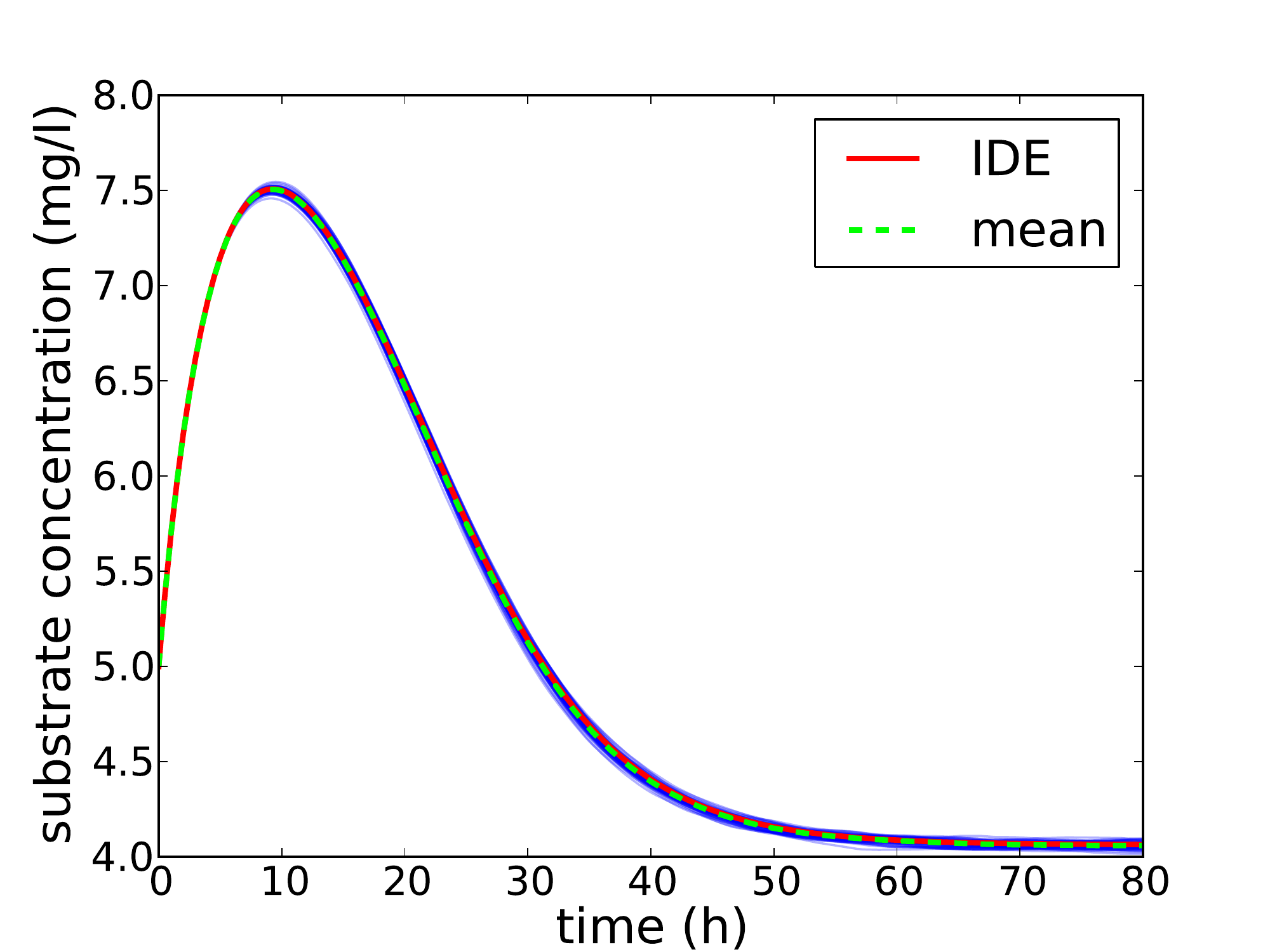}
\\ 
\includegraphics[width=4.5cm]{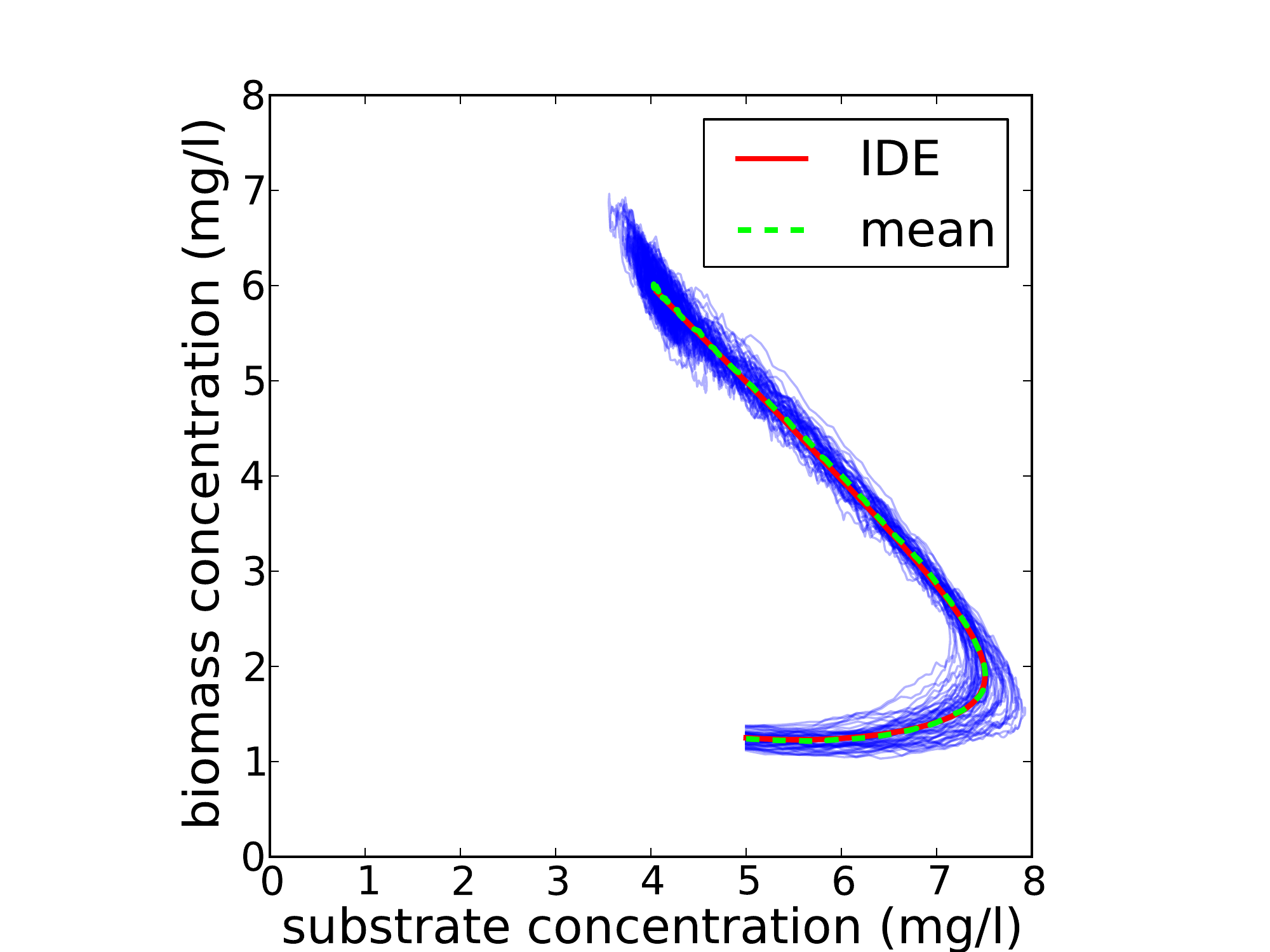}
&
\includegraphics[width=4.5cm]{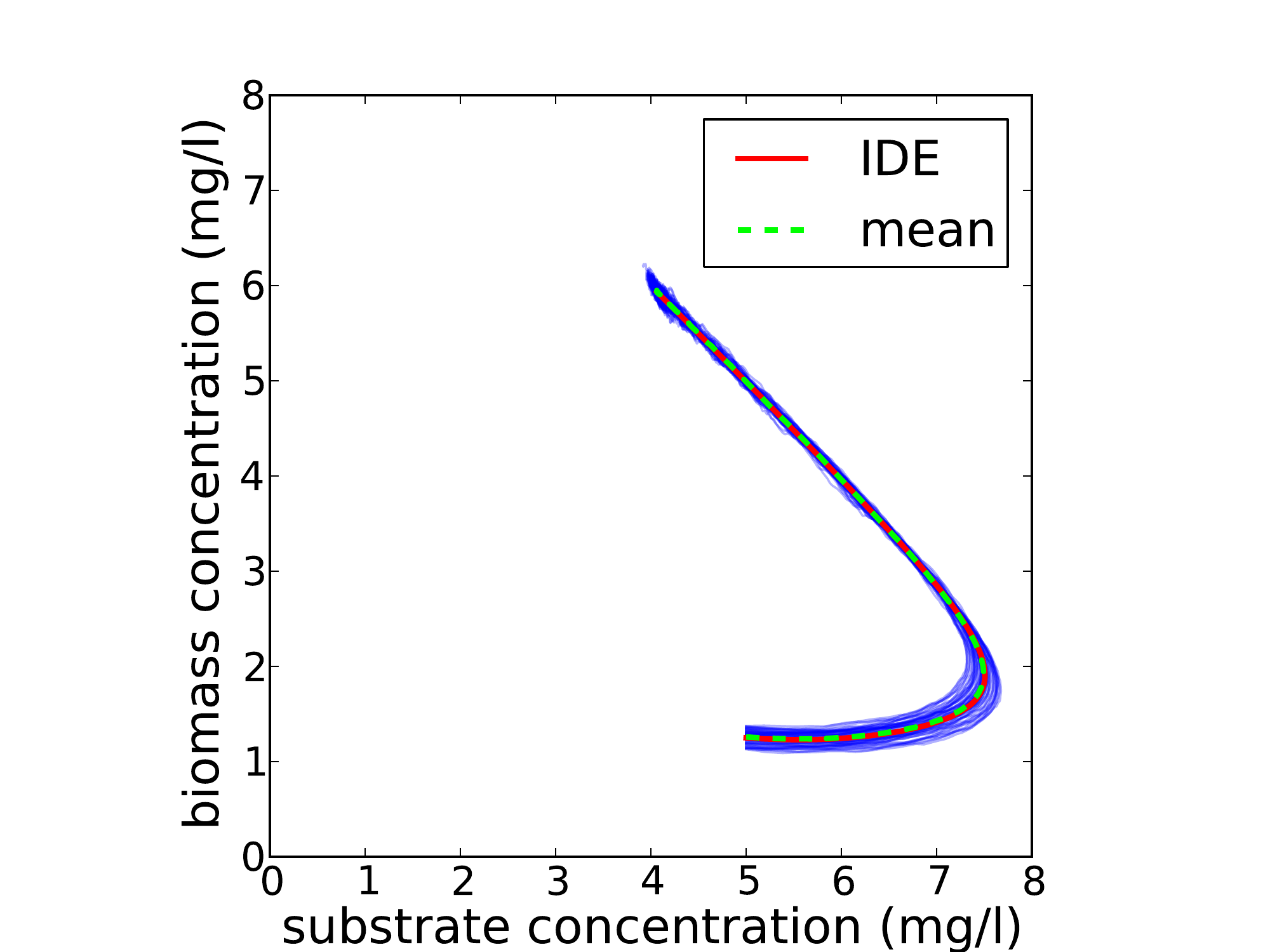}
&
\includegraphics[width=4.5cm]{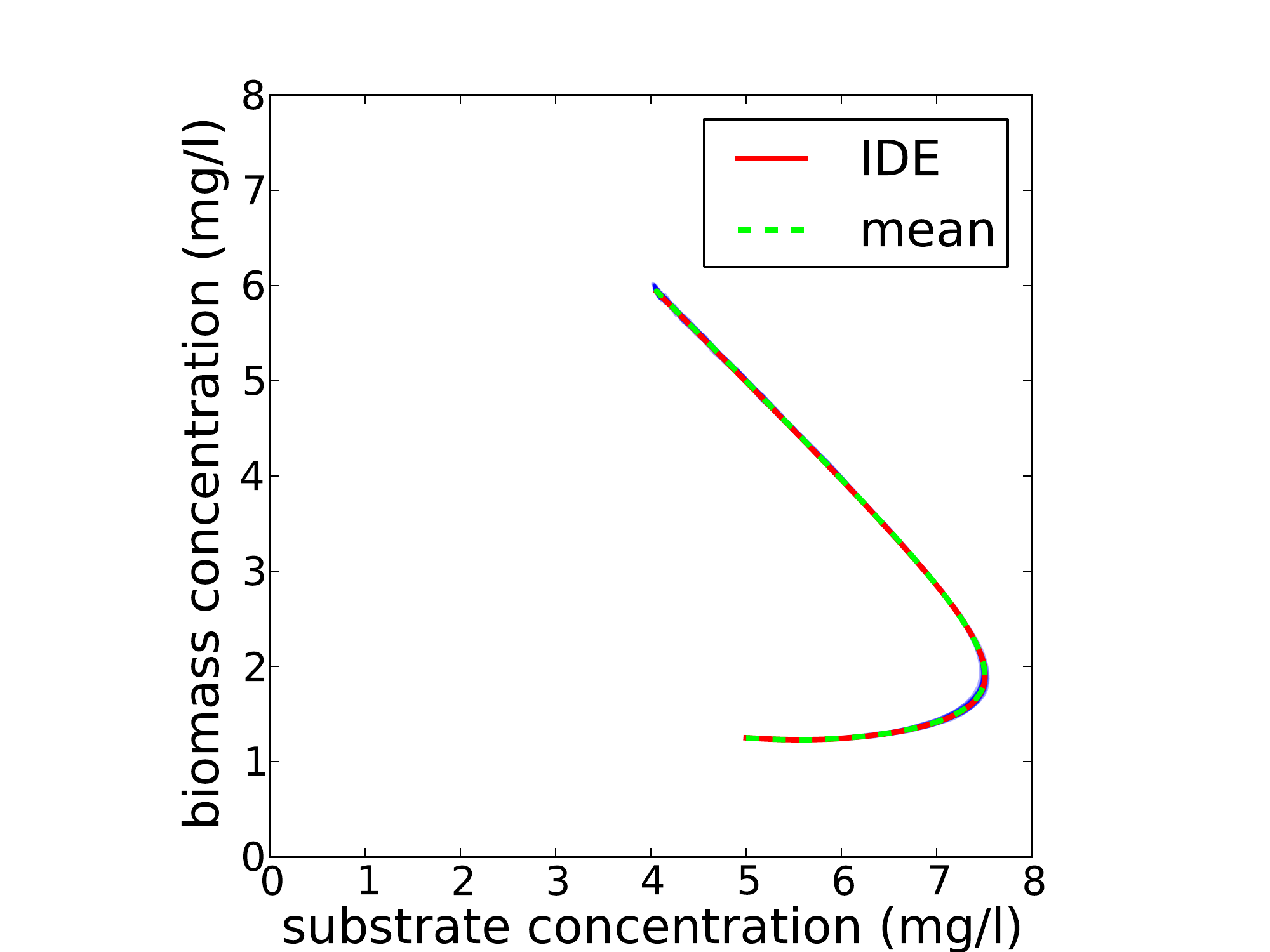}
\\ 
\scriptsize small population size
&
\scriptsize medium population size
&
\scriptsize large population size
\\
\scriptsize $V=0.05$ l, $N_0=100$
&
\scriptsize $V=0.5$ l, $N_0=1000$
&
\scriptsize $V=5$ l, $N_0=10000$
\end{tabular}
\end{center}
\vskip-1em
\caption{From top to bottom: time evolutions of the population size, the biomass concentration, the concentration substrate and  the concentrations phase portrait  for the three levels of population sizes (small, medium and large). The blue curves represent the trajectories of 60 independent runs of IBM. The green curve represents the mean value of these runs. The red curve represents the solution of the IDE. The rate $ D $ is 0.2 h$^{-1}$.}
\label{evol.taille.concentrations}
\end{figure}

\begin{figure}
\begin{center}
\includegraphics[trim=4.3cm 1.5cm 2.2cm 2.2cm,width=10cm]{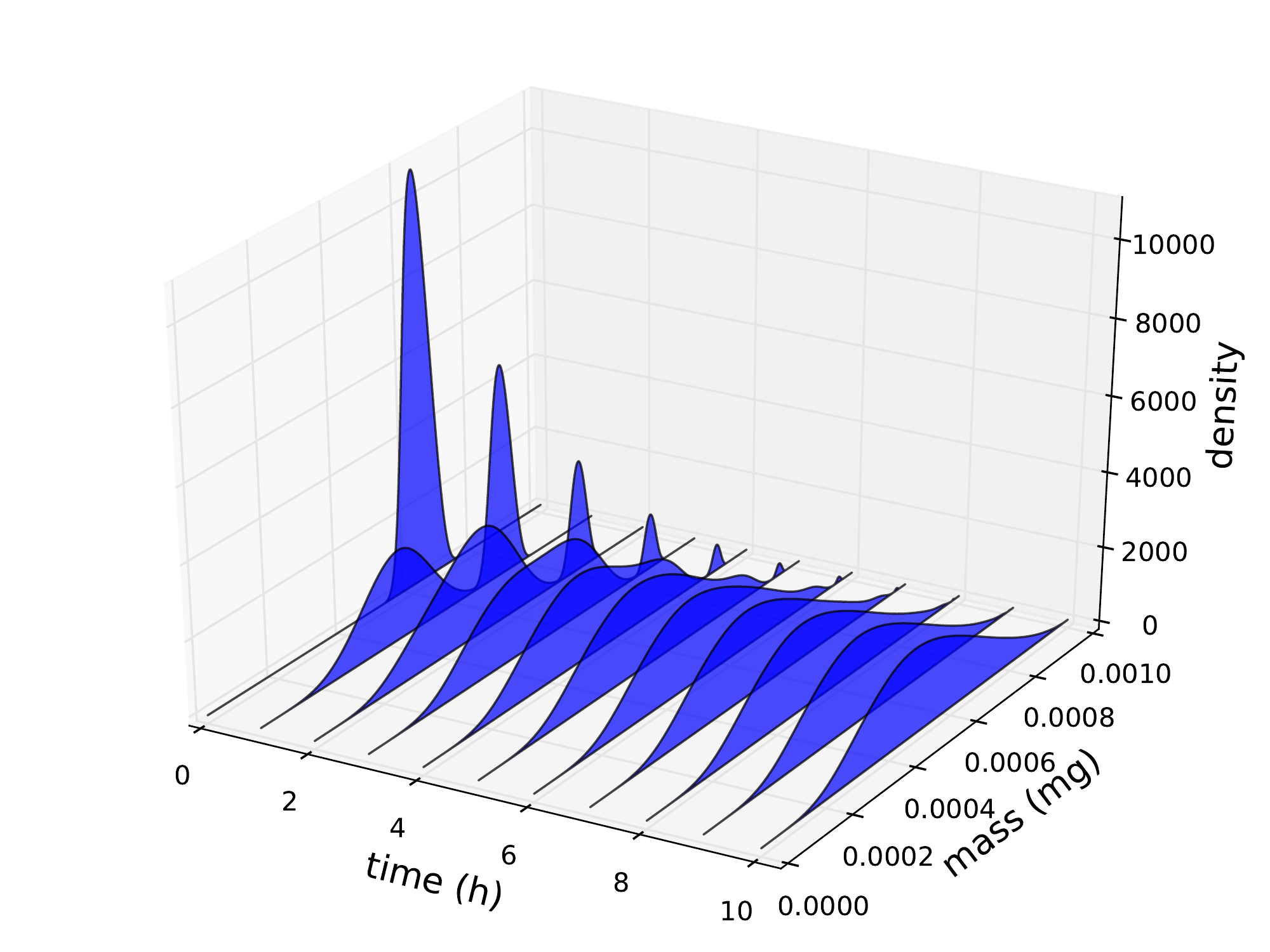}
\end{center}
\caption{Time evolution of the normalized mass distribution for the IDE \eqref{eq.limite.eid.fort}: we represent the simulation until time $T=10$ (h) only to illustrate the transient phenomenon due to the choice of the initial distribution \eqref{eq.d}. After a few iterations in time this distribution is bimodal, the upped mode growths in mass and disappears before $T=10$ (h).}
\label{fig.evol.eid}
\end{figure}

\begin{figure}
\begin{center}
\begin{tabular}{ccc}
\includegraphics[width=4.5cm]{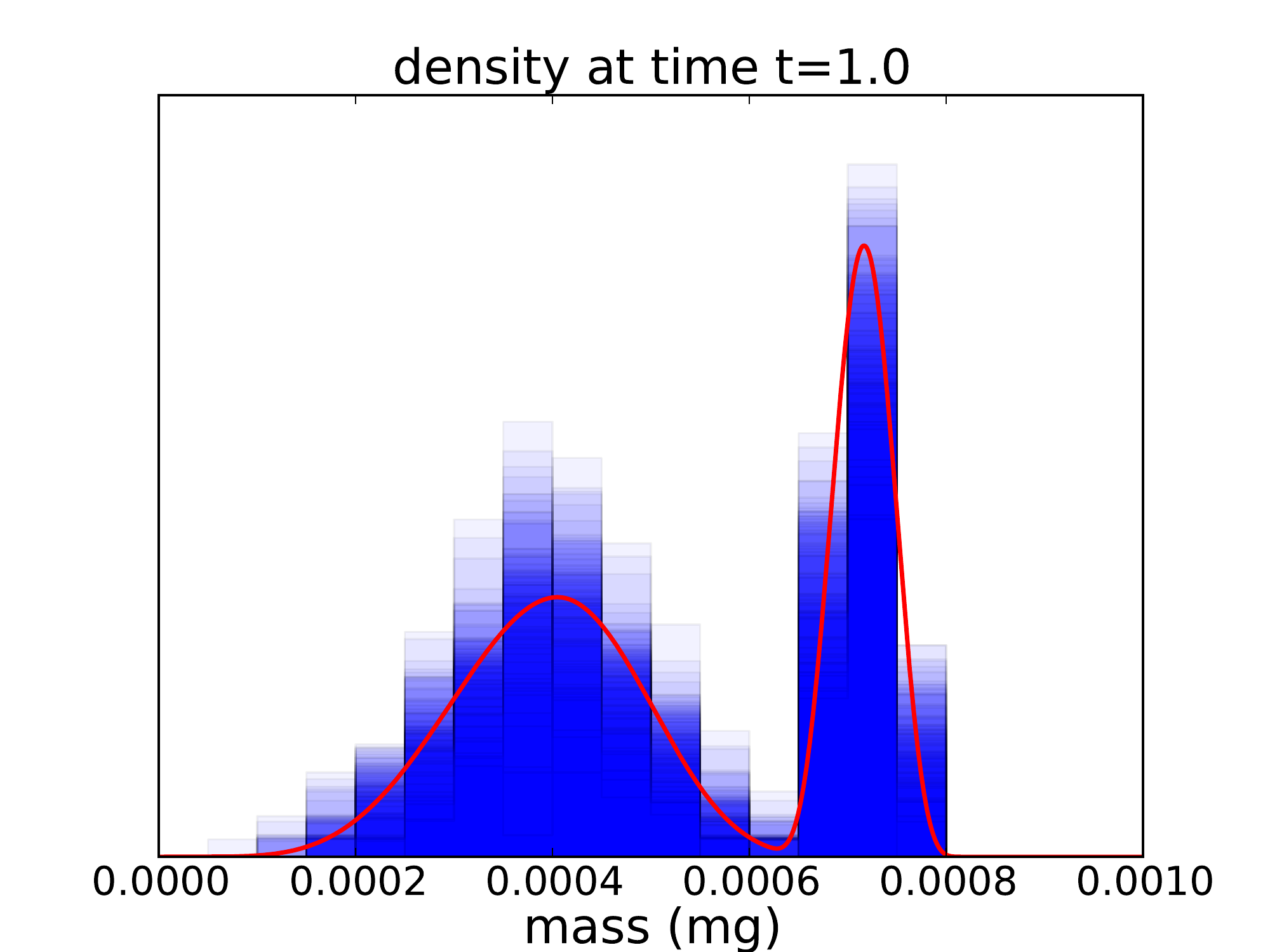}
&
\includegraphics[width=4.5cm]{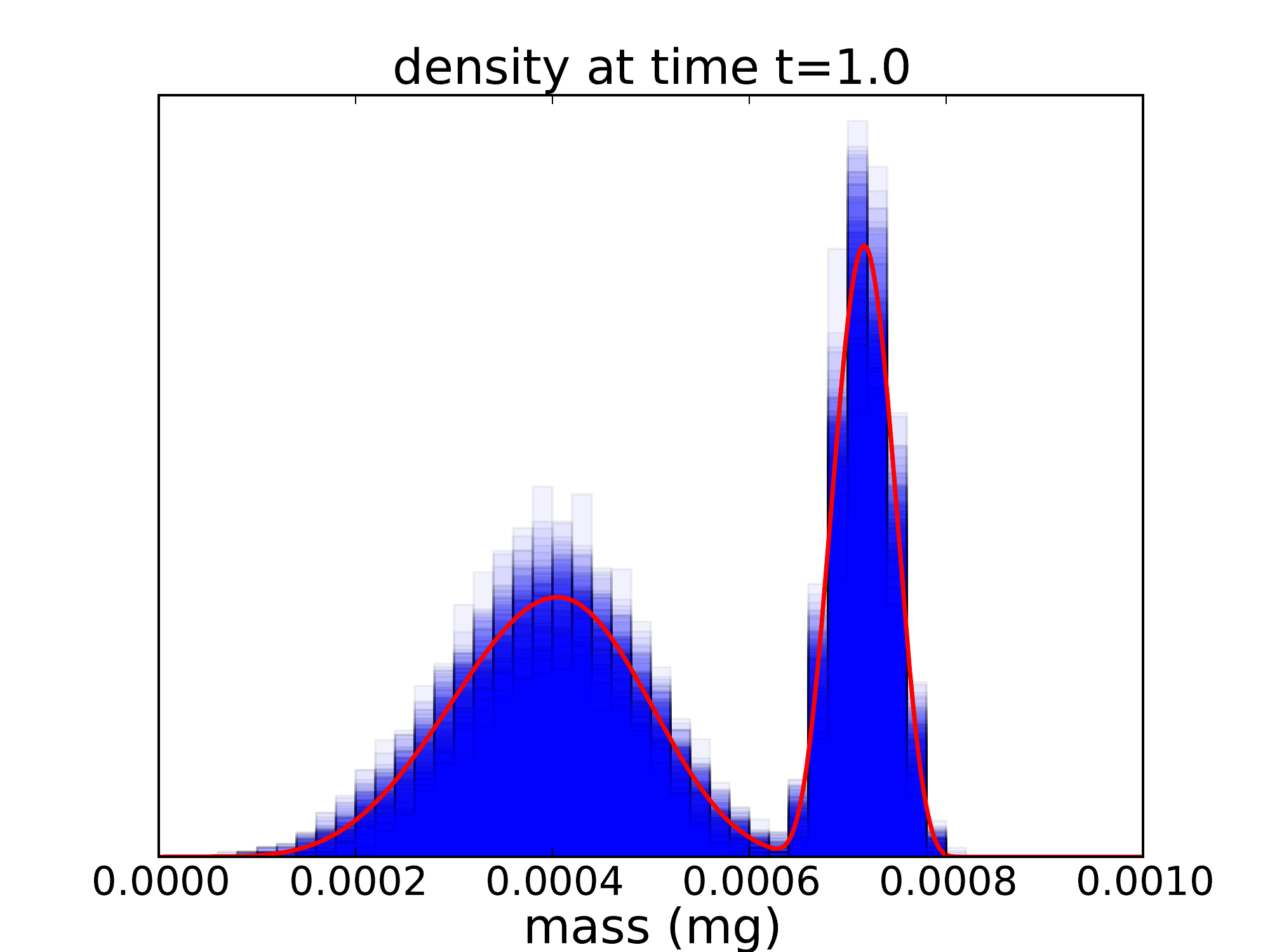}
&
\includegraphics[width=4.5cm]{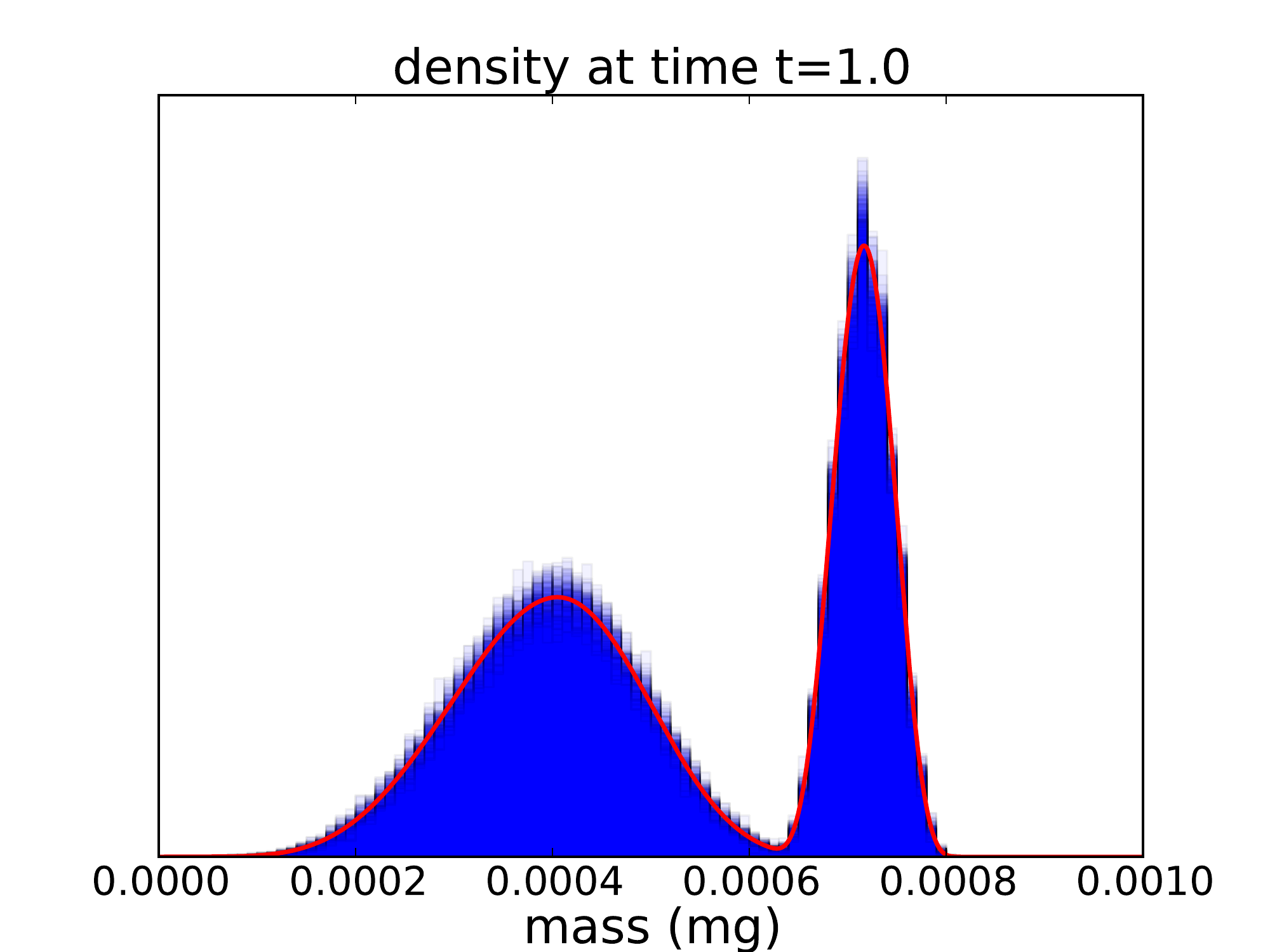}
\\ 
\includegraphics[width=4.5cm]{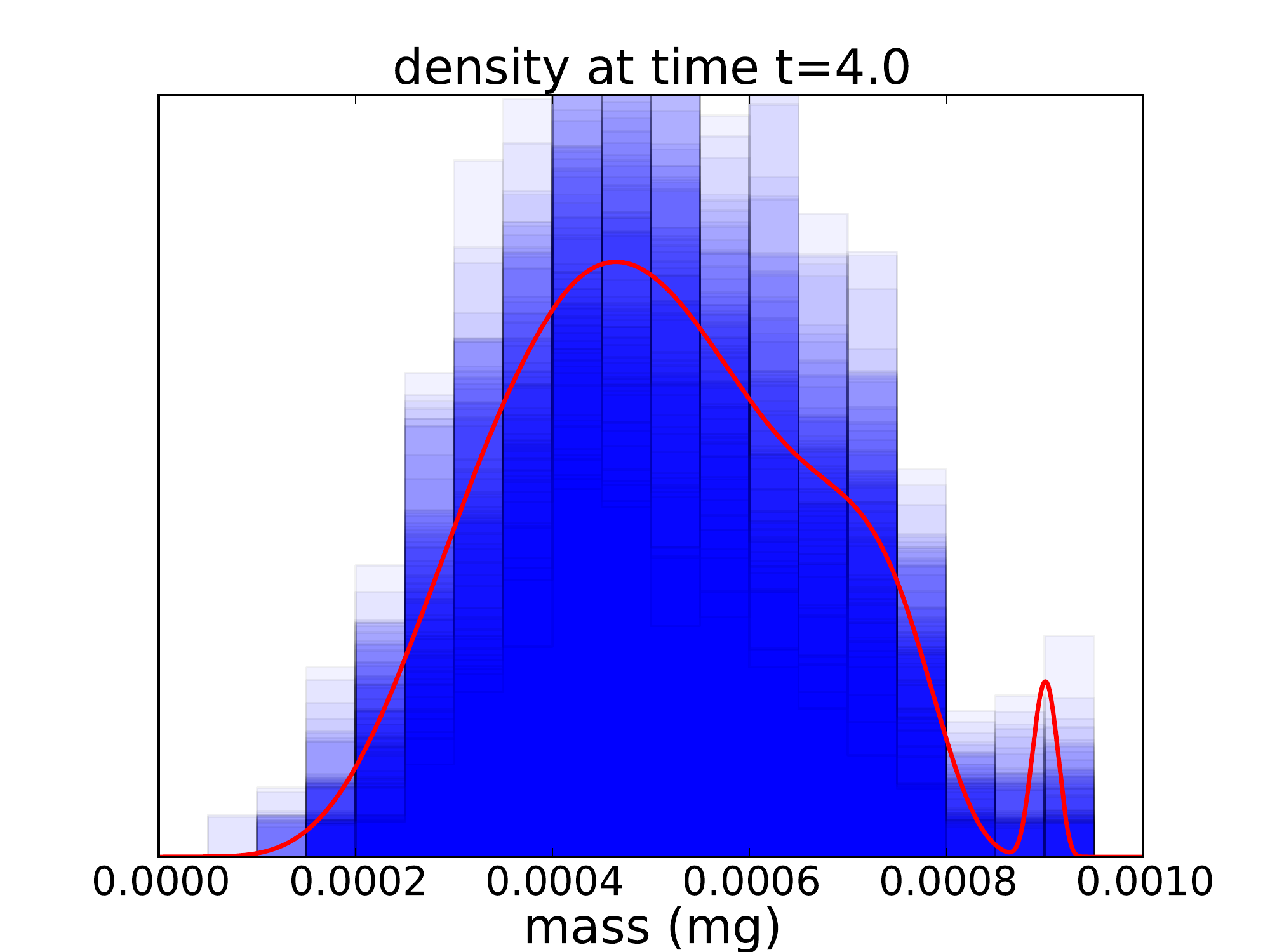}
&
\includegraphics[width=4.5cm]{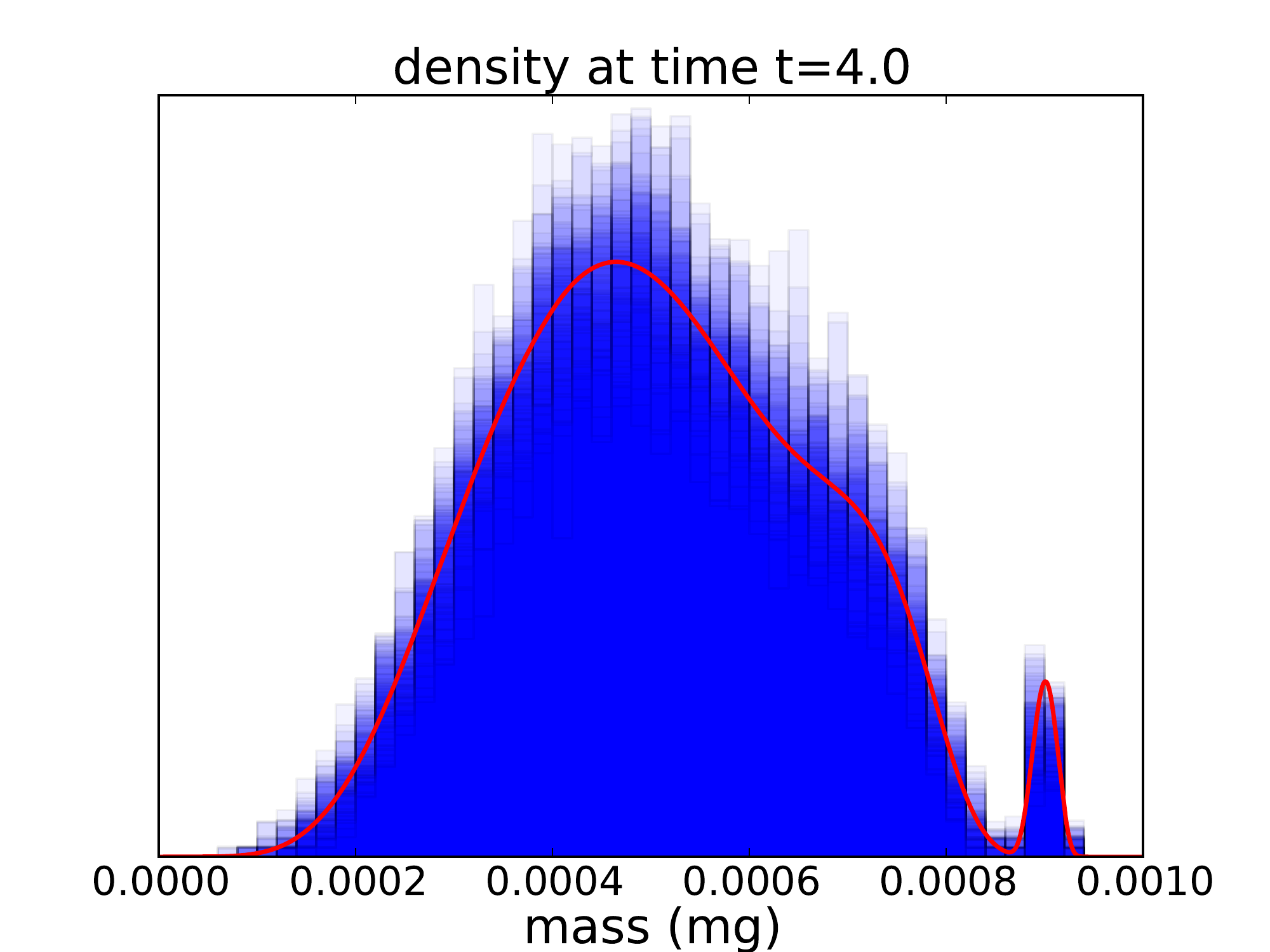}
&
\includegraphics[width=4.5cm]{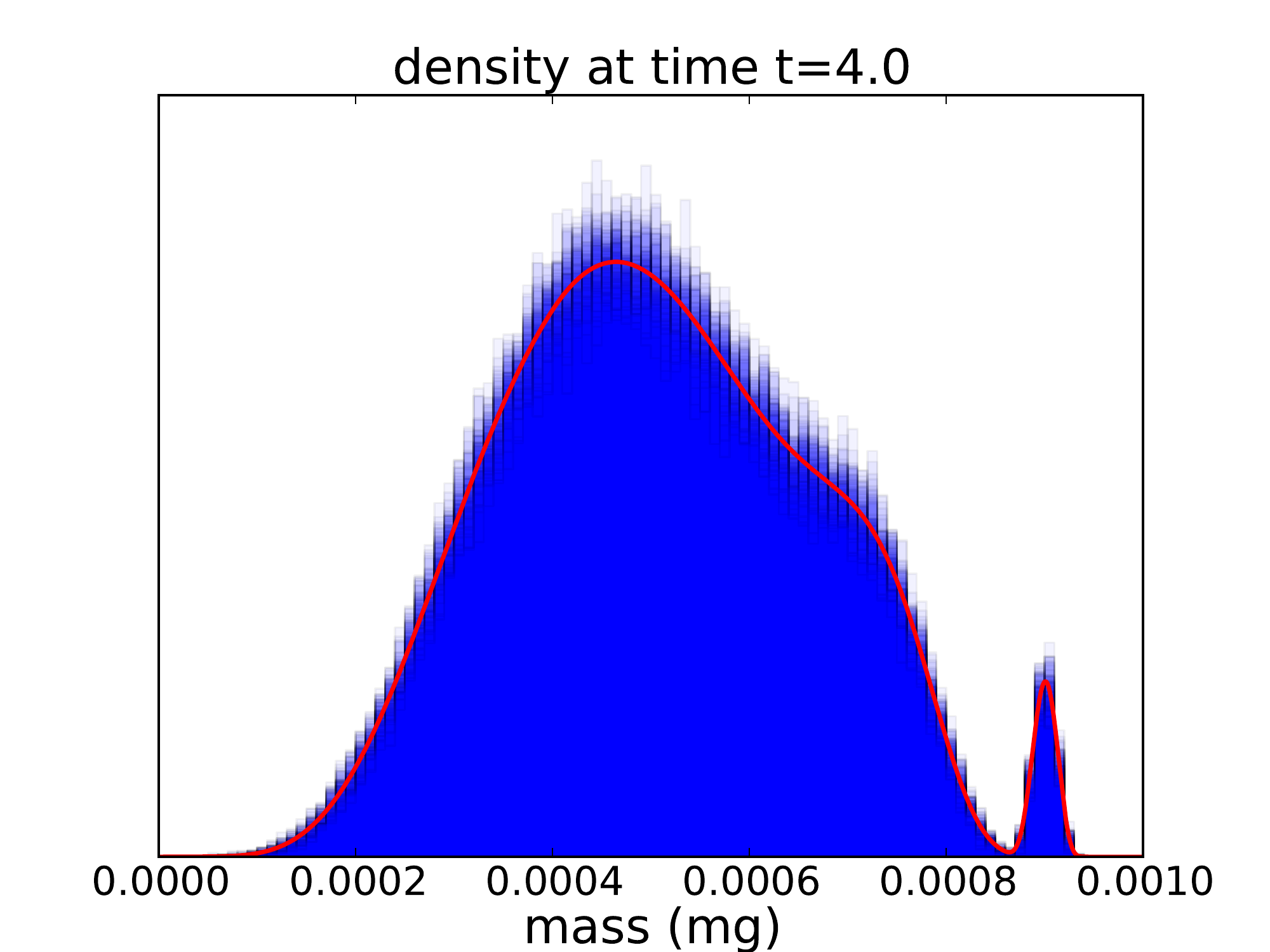}
\\ 
\includegraphics[width=4.5cm]{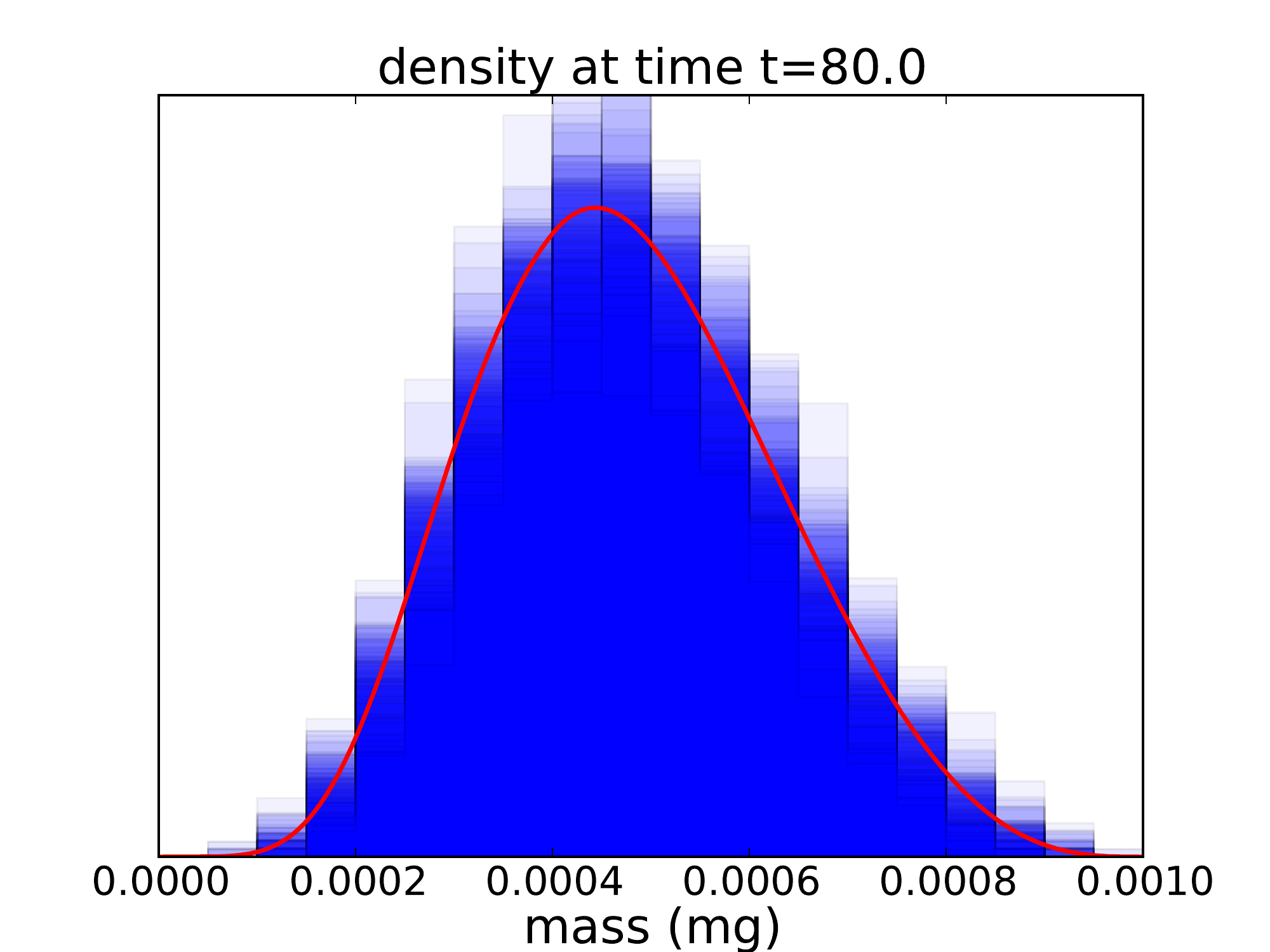}
&
\includegraphics[width=4.5cm]{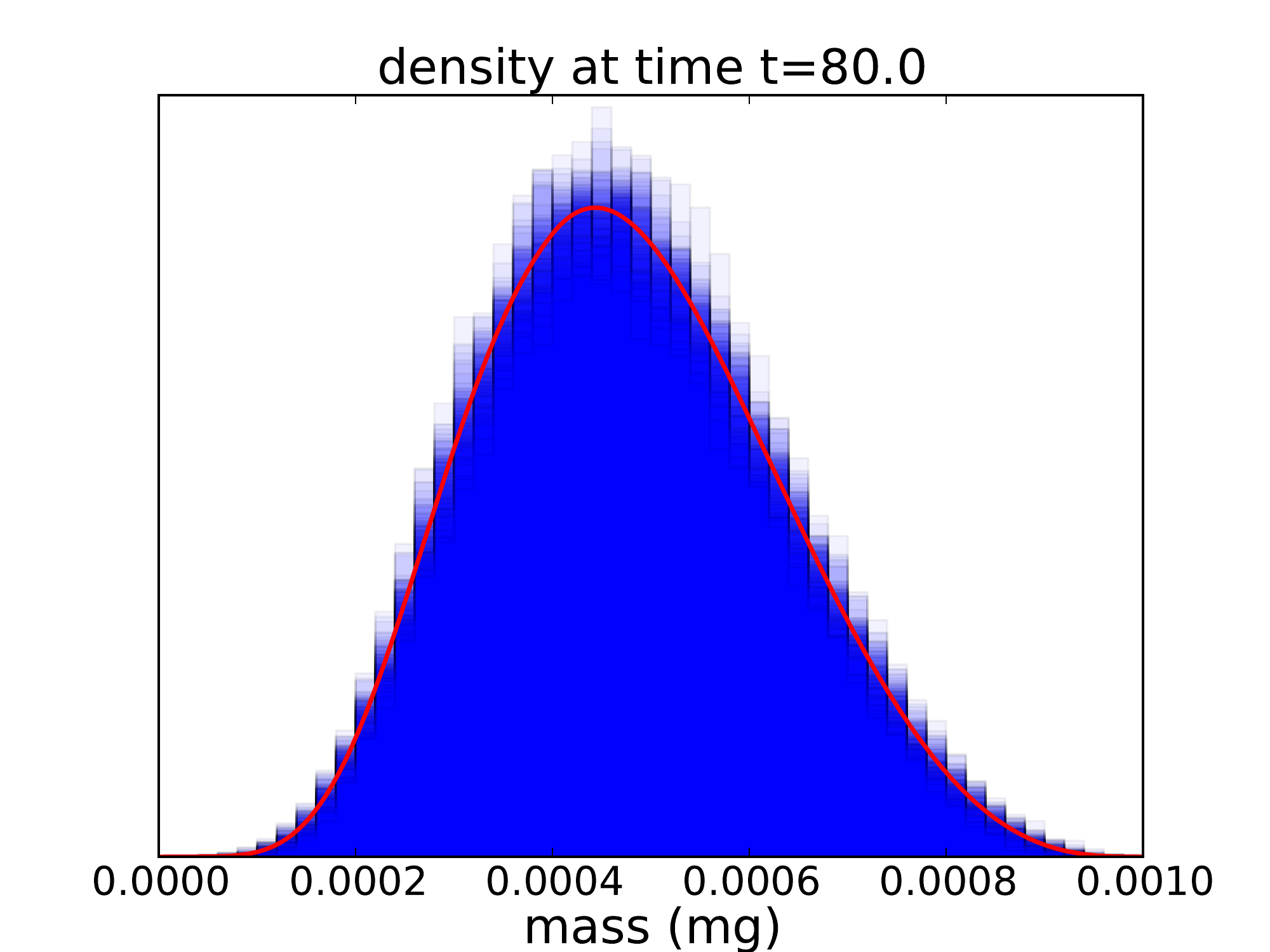}
&
\includegraphics[width=4.5cm]{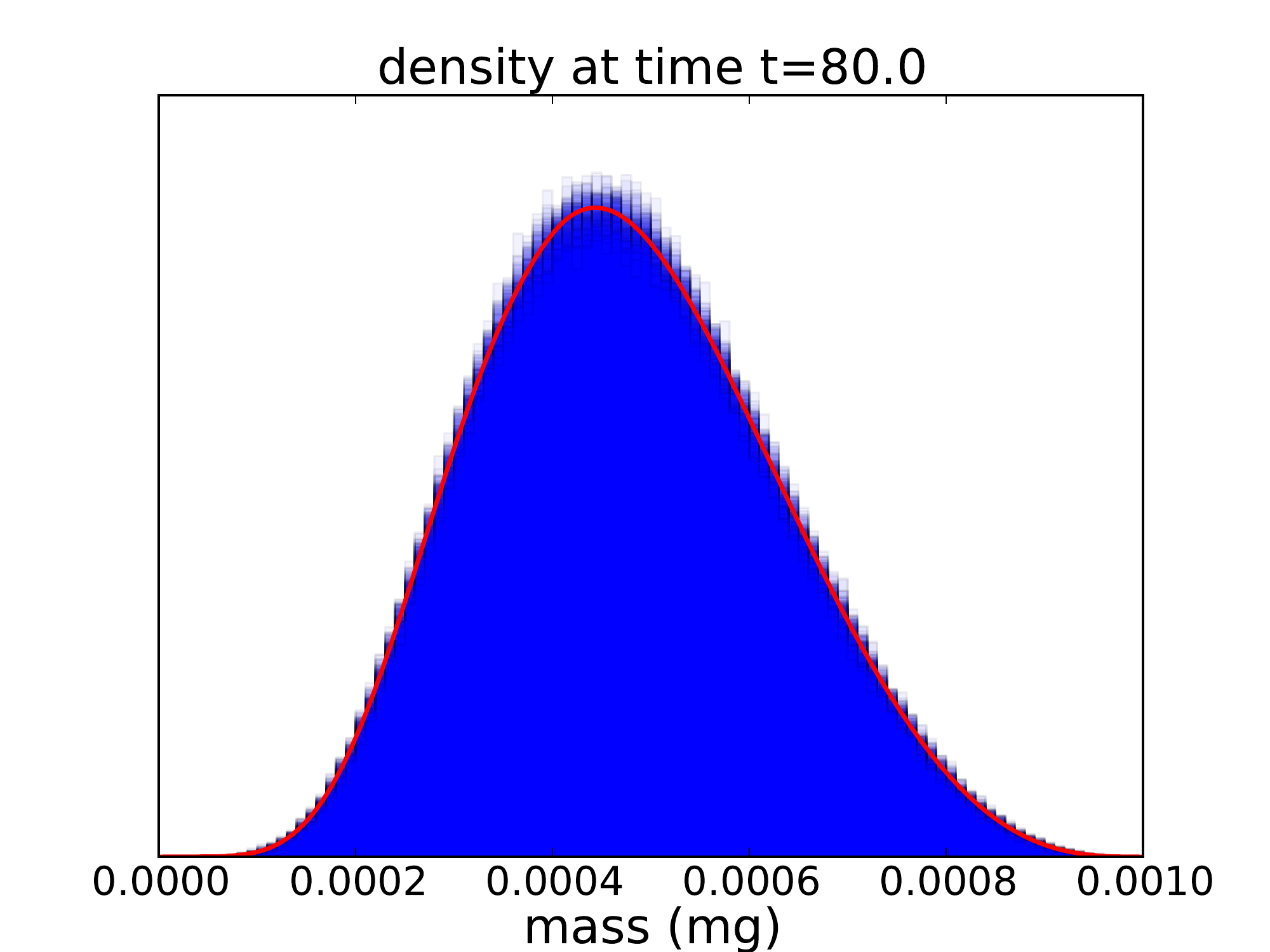}
\\ 
\scriptsize small population size
&
\scriptsize medium population size
&
\scriptsize large population size
\\
\scriptsize $V=0.05$ l, $N_0=100$
&
\scriptsize $V=0.5$ l, $N_0=1000$
&
\scriptsize $V=5$ l, $N_0=10000$
\end{tabular}
\end{center}
\caption{Mass distribution for the time $t=1$ (above), $t=4$ (middle) and $t=80$ (bottom) in small (left), medium (middle) and large (right) population size. For each graph, the blue histograms represent the empirical mass distributions of individuals for the 60 independent runs of IBM. In order to plot the histogram we have adapted the number of bins according to the population size. The red curve represents the mass distribution given by the IDE. The dilution rate $ D $ is 0.2 h$^{-1}$. Again the convergence of the IBM solution to the IDE in large population limit is observed.}
\label{repartition.masse}
\end{figure}

\subsection{Comparison of the IBM, the IDE and the ODE}

We now compare the IBM and the IDE to the classical chemostat model described by the system of ODE \eqref{eq.chemostat.edo}. The function $\tilde\mu$ is the specific growth rate. The growth model in both the IBM and the IDE is of Monod type, so for the ODE model we also consider the classical Monod kinetics:
\begin{align}
\label{eq.edo.monod}
	\tilde\mu(S)
	& = 
	\mu_{\max} \, \frac{S}{K_s+S}\,.
\end{align}
The parameters of this Monod law are not given in the initial model and we use 
a least squares method to determine the value of the parameters  $\mu_{\max}$ and $K_s$
which minimize the quadratic distance between $(S_{t},X_{t})_{t\leq T}$ given by \eqref{eq.chemostat.edo} and  $(S_{t},X_{t}=\int_{\X} x\,p_{t}(x)\,\rmd x)_{t\leq T}$ given by \eqref{eq.limite.substrat.fort}-\eqref{eq.limite.eid.fort}.

The numerical integration of the ODE \eqref{eq.chemostat.edo} presents no difficulties
and is performed by the function \texttt{odeint} of the module  \texttt{scipy.integrate} of \texttt{Python} with the default parameters.

\bigskip

First we consider  a simulation based on the initial mass density $d(x)$  defined by 
\eqref{eq.d}. With this initial density both the IDE and the IBM feature a transient phenomenon described in the previous section and illustrated in Figures  \eqref{fig.evol.eid} et~\eqref{repartition.masse}.

Figure \ref{fig.edo.ibm.eid} (left) shows a significant difference between the IBM and the IDE on the one hand and the ODE on the other hand, the latter model cannot account for the transient phenomenon. With the first two models, the individual bacteria are withdrawn uniformly and independently of their mass (large mass bacteria has the same probability of withdrawal as small mass bacteria) and 
at the beginning of the simulation there is  a decrease in biomass
as at initial state  $d(x)$ has a substantial proportion of large bacteria mass.
The ODE is naturally not able to account for this phenomenon.

\begin{figure}
\begin{center}
\begin{tabular}{cc}
\includegraphics[width=5cm]{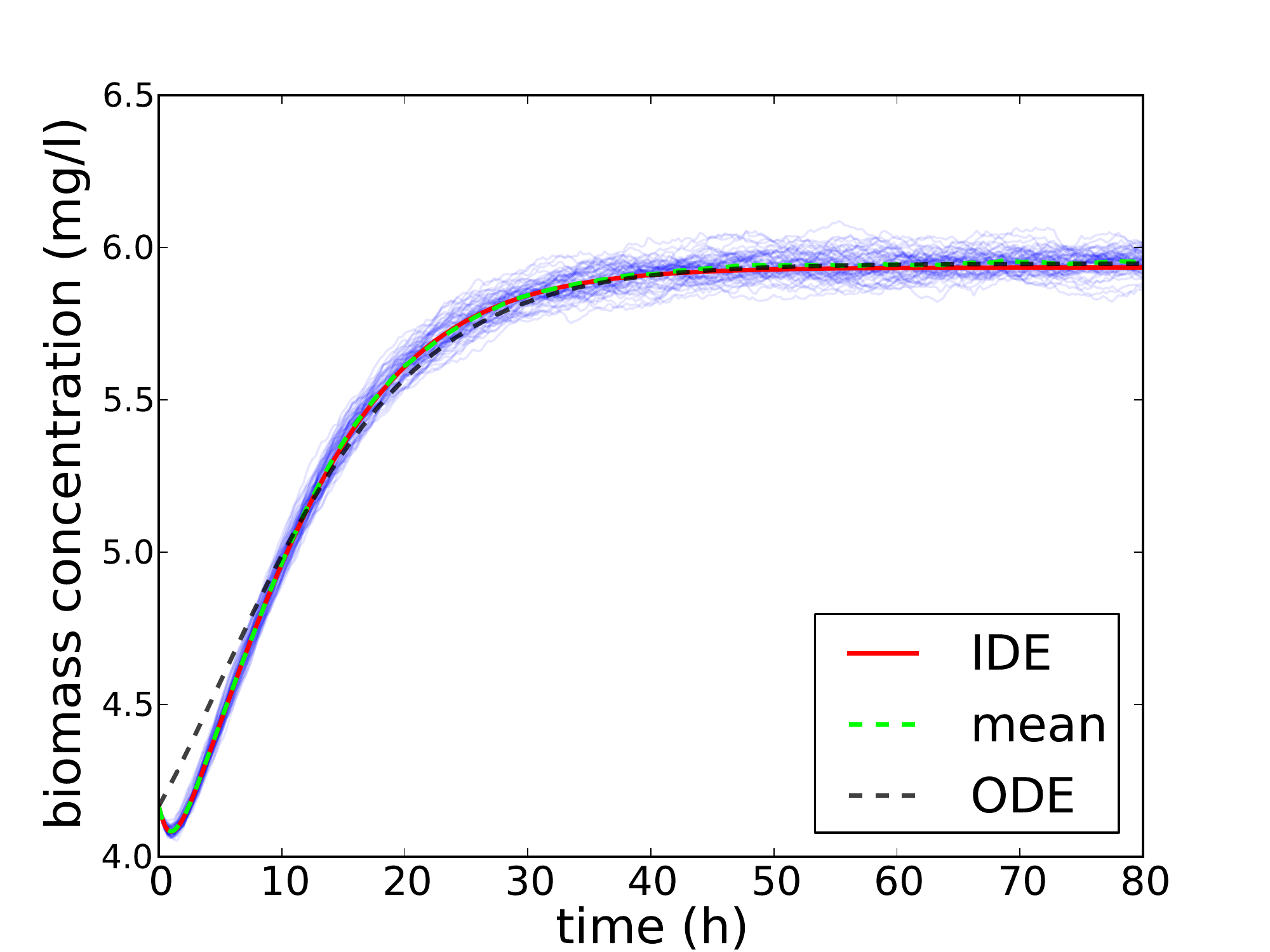}
&
\includegraphics[width=5cm]{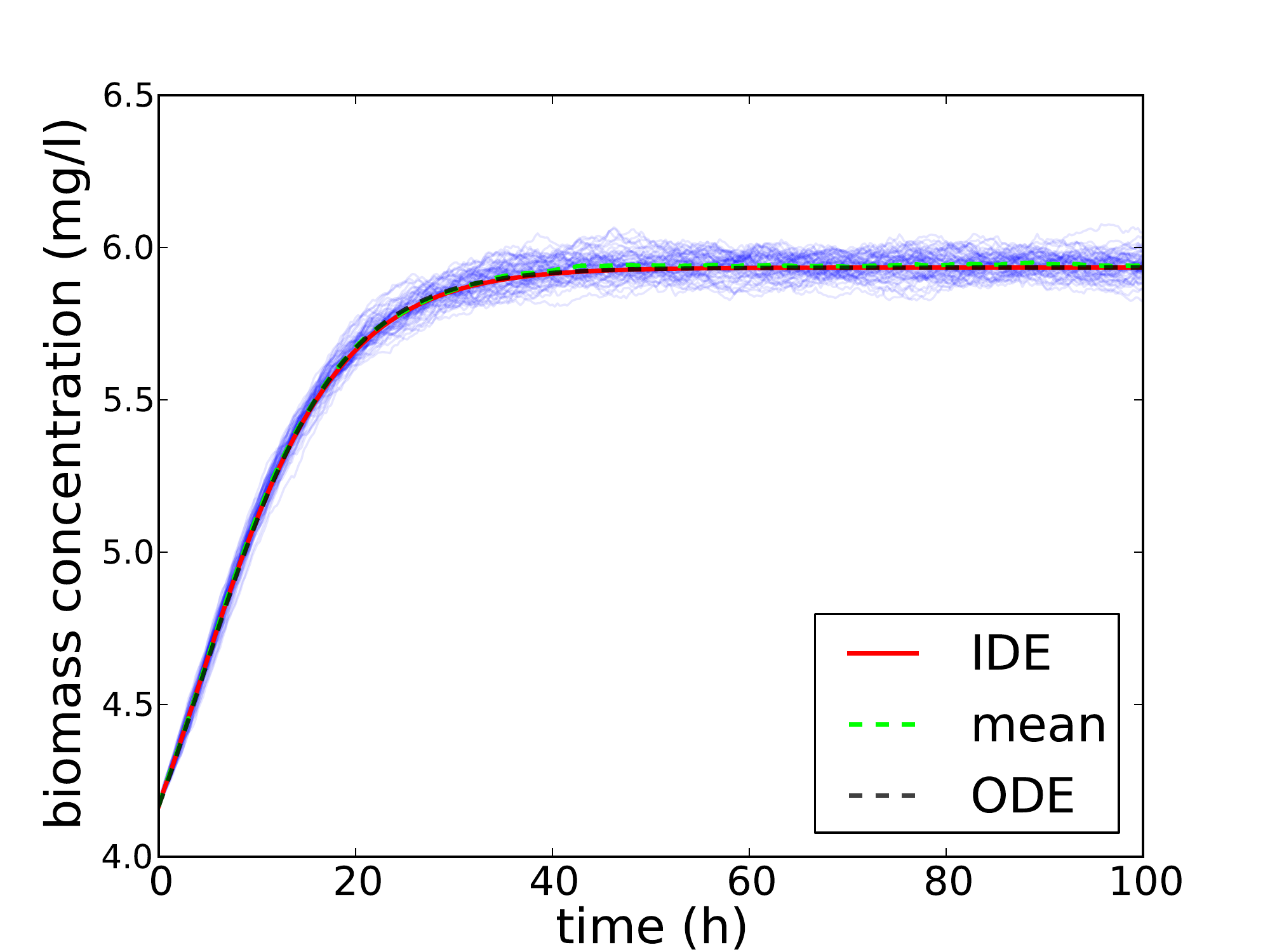}
\\
\includegraphics[width=5cm]{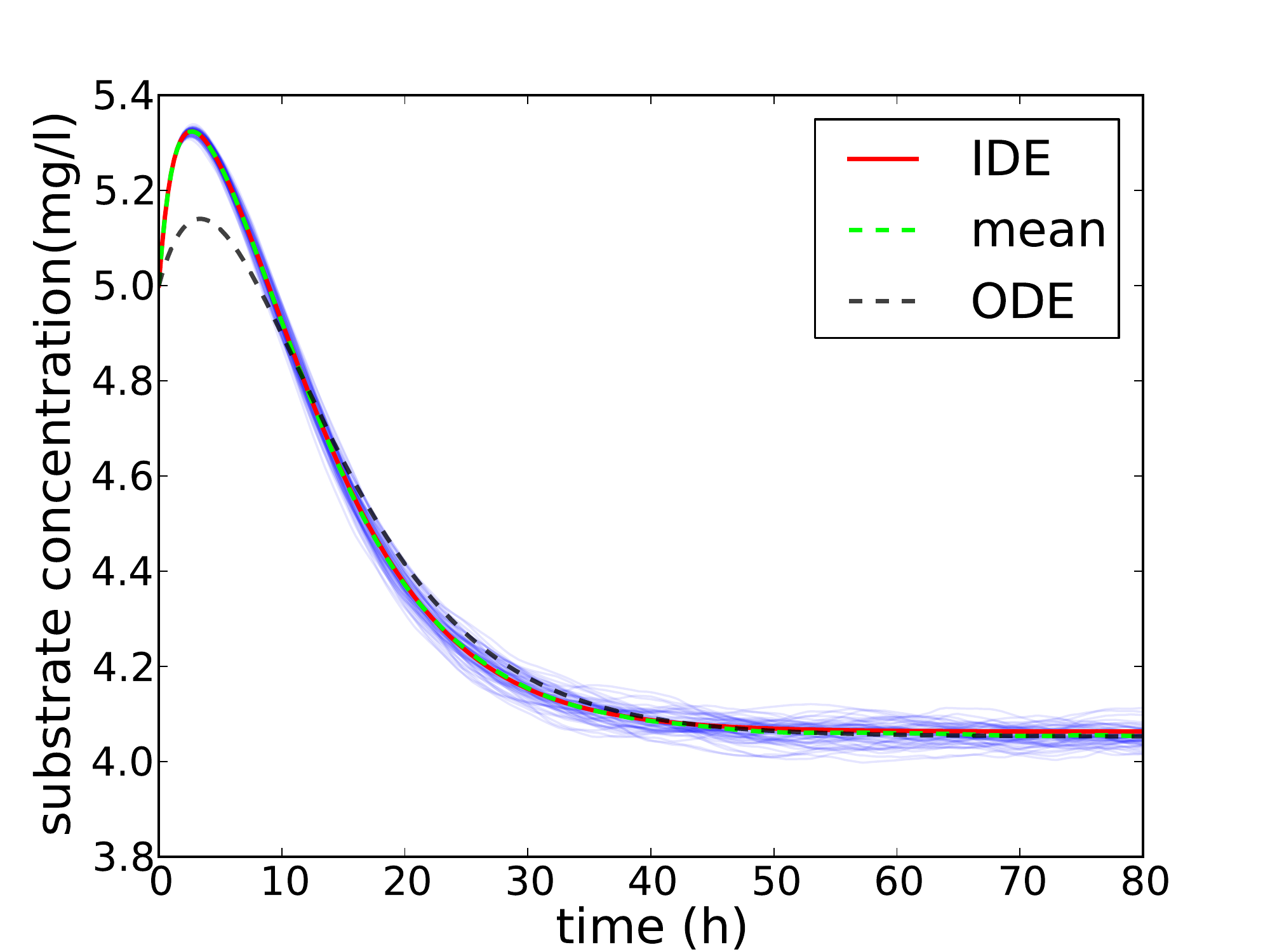}
&
\includegraphics[width=5cm]{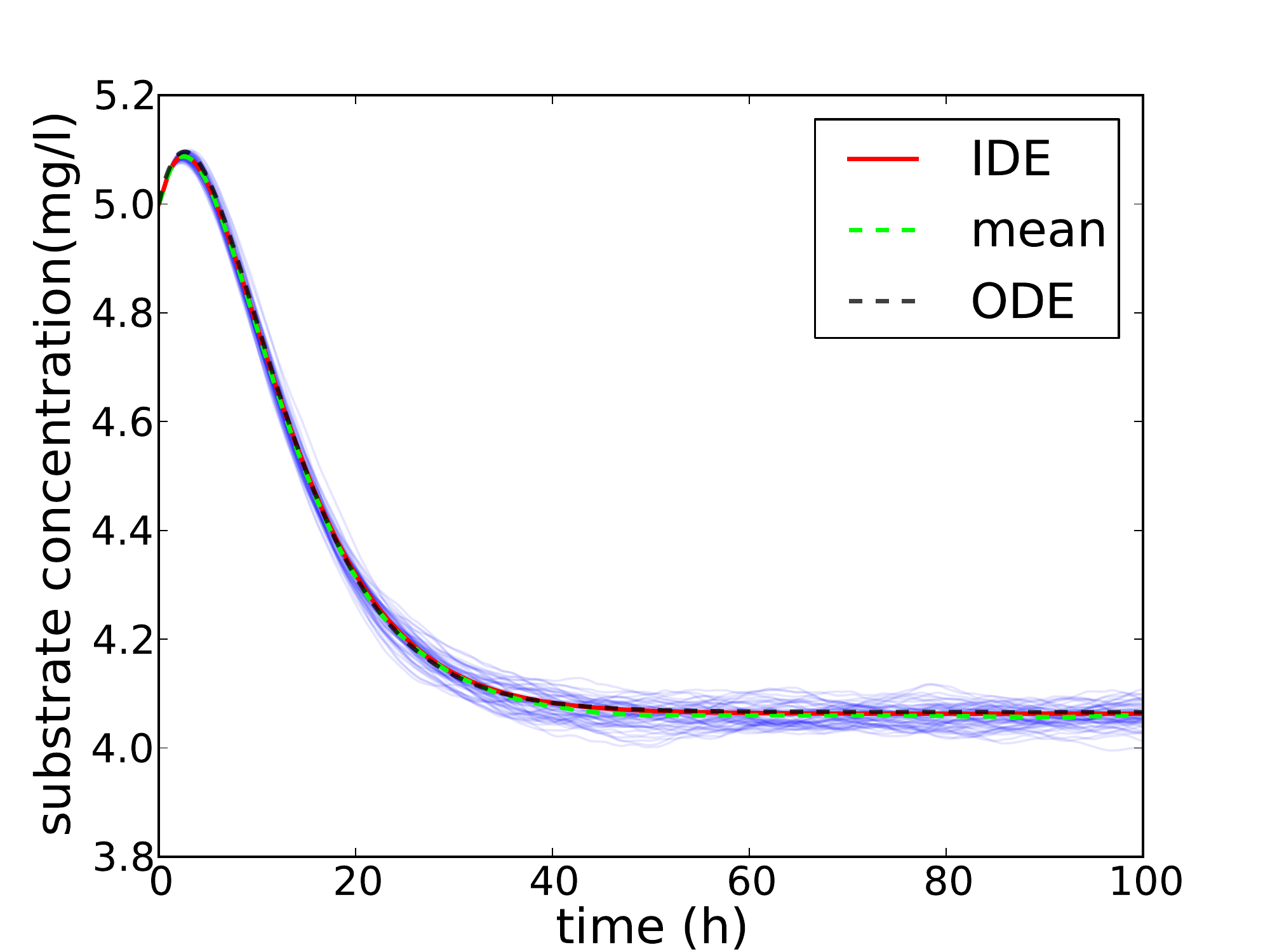}
\\
\includegraphics[width=6cm]{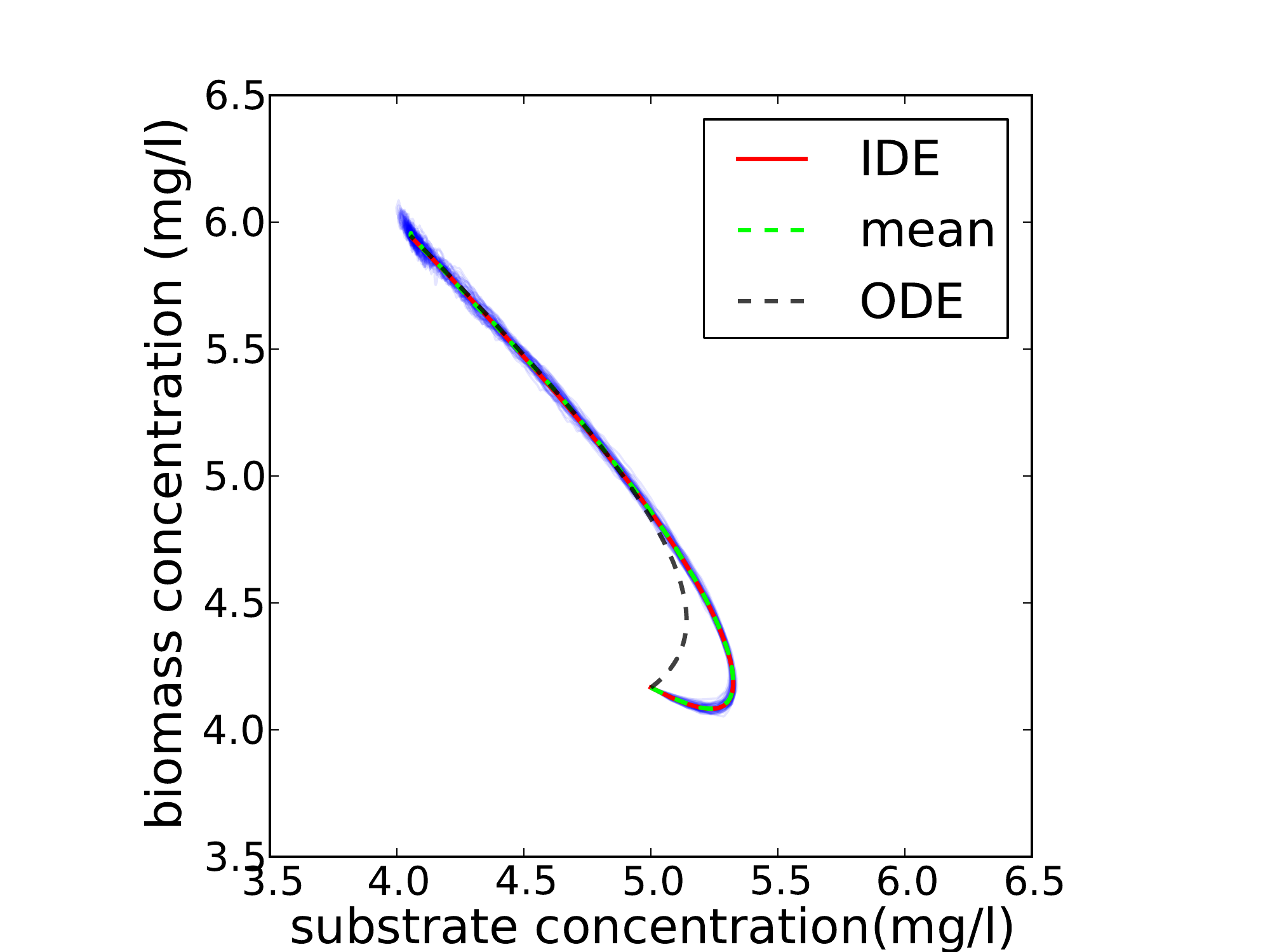}
&
\includegraphics[width=6cm]{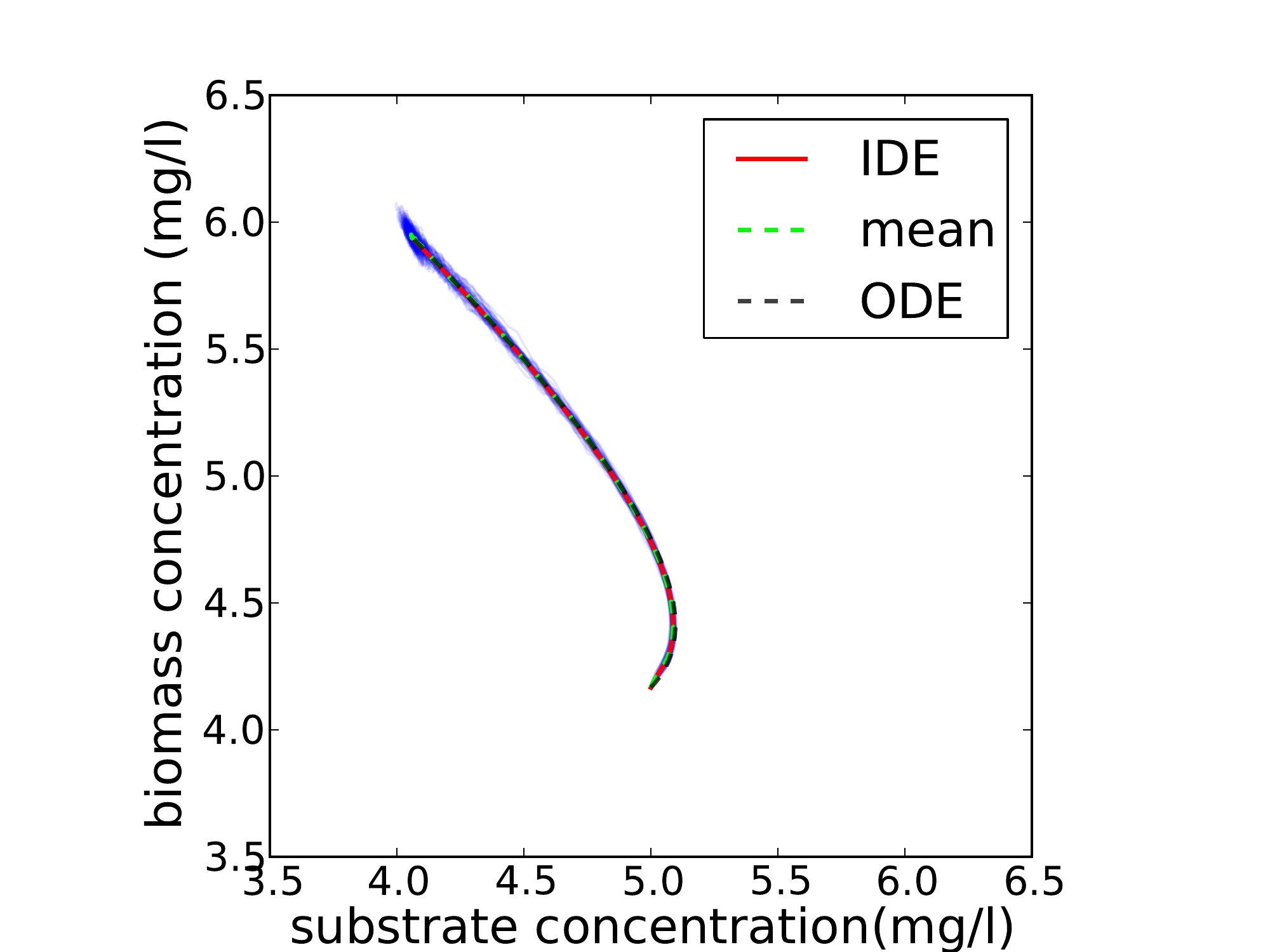}
\\[0.7em]
\small initial density $d(x)$ \eqref{eq.d}
&
\small initial density $d'(x)$ \eqref{eq.d'}
\end{tabular}
\end{center}
\caption{Time evolution of the biomass concentration (top), the substrate concentration (middle) and the concentration trajectories in the phase space (bottom) according to the initial mass distributions \eqref{eq.d} (left) and \eqref{eq.d'} (right). In blue, the trajectories of 60 independent runs of the IBM simulated with $V=3$ l and $N_0=20000$ (left), $N_0=25000$ (right); in green, the mean of the IBM runs; in red, the solution of IDE \eqref{eq.limite.substrat.fort}-\eqref{eq.limite.eid.fort}; in black, the solution of the ODE \eqref{eq.chemostat.edo}. The latter is fitted by the least squares method on the IDE, the parameters of the Monod law \eqref{eq.edo.monod} are $\mu_{\max}=0.341$ and $K_s = 2.862$ in the first case and $\mu_{\max}=0.397$ and $K_s = 3.996$ in the second.
As initial densities are different, $N_{0}$ is adapted so that the average initial biomass concentration is the same in both cases.
The dilution rate $D$ is 0.2 h$^{-1}$.  In the first case, the ODE gives no account for the transient phenomenon described in the previous section, see Figures \eqref{fig.evol.eid} and \eqref{repartition.masse}, while the IDE and the IBM give coherent account for it. In the second case the three models are consistent.}
\label{fig.edo.ibm.eid}
\end{figure}

This phenomenon no longer appears when uses the following density:
\begin{align}
\label{eq.d'}
d'(x)
	& =	
		\Biggl(
			\frac{x-0.00035}{0.0003}
			\,\left(1-\frac{x-0.00035}{0.0003}\right)
		\Biggr)^5 \,
		1_{\{0.00035 < x < 0.00065\}}\,.
\end{align}

Indeed, from Figure \ref{fig.edo.ibm.eid} (right), there is no longer any biomass decay  at the beginning of the simulation and the different simulations are comparable,  
the ODE and the IDE match substantially.

\subsection{Study of the washout}

\begin{figure}
\begin{center}
\includegraphics[width=10cm]{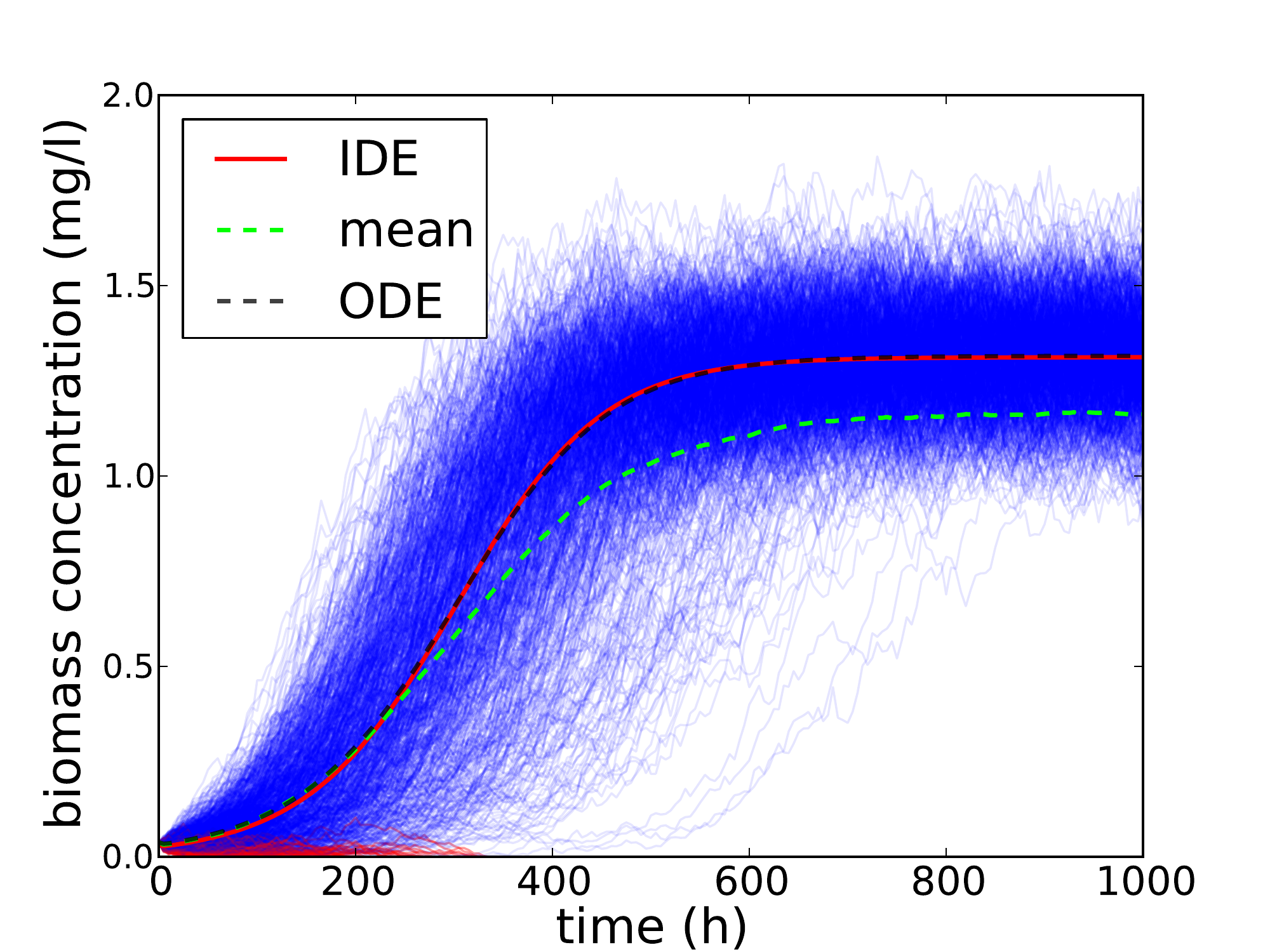}
\end{center}
\caption{Time evolution of the biomass concentration. In blue, 1000 independent realizations of the IBM simulated with $V=0.5$ l and $N_0=30$;  in green, the mean of these runs; in red, the solution of the IDE; in black, the solution of the ODE with parameters values $\mu_{\max}=0.482$ and $K_s = 6.741$. The dilution rate $ D $ is   0.275 h$^{-1}$.  Among the 1000 independent runs of the IBM, 111 lead to washout while the deterministic models converge to an equilibrium with strictly positive biomass. The mean value of the  1000 runs of the IBM gives account for the washout probability while IDE and ODE models do not account for this question.}
\label{fig.lessivage}
\end{figure}

\begin{figure}[p]
\begin{center}
\includegraphics[width=9cm]{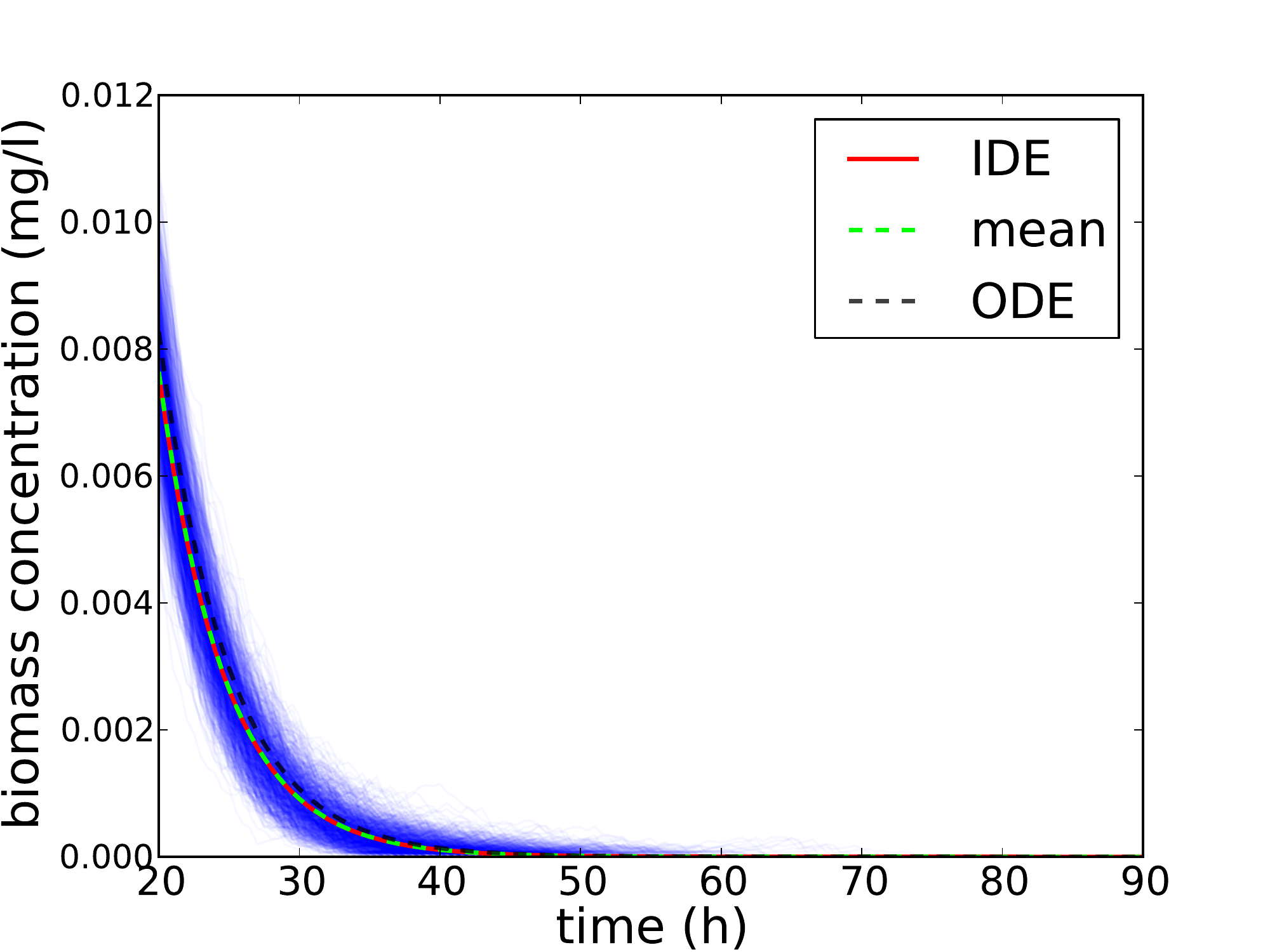}\\
\includegraphics[width=9cm]{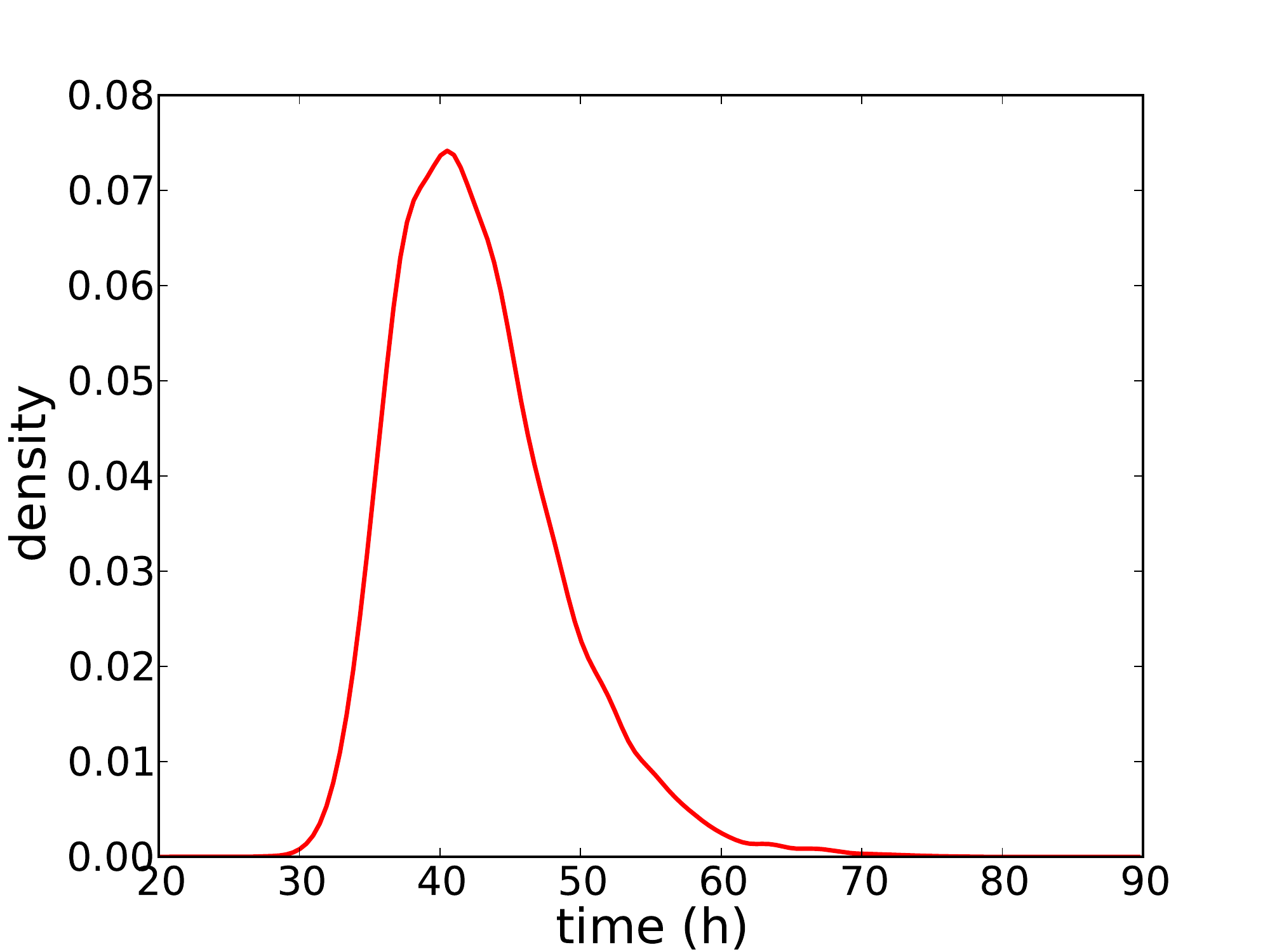}
\caption{$\blacktriangleright$ (Top) Evolution of biomass concentration between $t=20$ and $t=90$ h: blue, 1000 independent runs of the IBM; in green, the mean value of these runs; in red the solution of the IDE; in black, the solution to the ODE with parameters $\mu_{\max}=0.814$ and $K_s = 17.547$. The parameters are $V=10$ l and $N_0=10000$, the dilution rate $D$ is  0.5 h$^{-1}$. For both deterministic models, the size of the population decreases exponentially rapidly to 0 but remains strictly positive for any finite time. However, all the runs of the IBM reach washout in finite time $\blacktriangleright$ (Bottom) empirical distribution of the washout time calculated from 7000 independent runs of the IBM and plotted using a time kernel regularization.}
\label{fig.lessivage2}
\end{center}
\end{figure}

One of the main differences between deterministic  and stochastic models lies in their way of accounting for the washout phenomenon  (or extinction phenomenon in the case of an ecosystem). With a sufficiently small dilution rate $D$, the solutions of the ODE \eqref{eq.chemostat.edo} and the IDE  \eqref{eq.limite.substrat.fort}-\eqref{eq.limite.eid.fort}  converge to an equilibrium point with strictly positive biomass. In fact, the  washout is an unstable equilibrium point and apart from the line corresponding to the null biomass, the complete phase space corresponds to a basin of attraction leading to a solution with a strictly positive biomass asymptotic point. However, from Figure \ref{fig.lessivage}, among the 1000 independent runs of the IBM, 111 of them converge to washout before time $t=1000$ h; so the probability of washout at this instant is approximately 11\%. It may be noted that the IDE and the ODE do not correspond to the average value of the IBM since only the latter may reflect the washout in a finite time horizon.

Now we consider a sufficiently large dilution rate, $D=0.5$ h$^{-1}$, corresponding to the washout conditions. Figure \ref{fig.lessivage2} (left) presents the evolution of the biomass concentration in the different models. The runs of the IBM converge to the washout in finite time whereas both deterministic ODE and IDE models converge exponentially to washout without ever reaching it in finite time. Figure \ref{fig.lessivage2} (right) shows the empirical distribution of the washout time calculated from 7000 independent runs of the IBM. This washout time features a relatively large variance.

\section*{Conclusion}

We proposed a hybrid chemostat model. On the one hand the bacterial population is modeled as an individual-based model: each bacterium is explicitly represented through its mass. On the other hand, the substrate concentration dynamics is represented as a conventional ordinary differential equation. Mechanisms acting on the bacteria are explicitly described: growth, division and up-take. Consumption of the substrate is also described and it is through this mechanism that bacteria interact.

We described an exact Monte Carlo simulation technique of the model. Our main result is   the convergence of the IBM to an integro-differential equation model when the population size tends to infinity. There is a convergence in distribution for coupled stochastic processes: the first one with càdlàg trajectories takes values ​​in the set of finite measures on the space of  masses, the second one with continuous trajectories takes values ​​in the set of substrate concentration values. The integro-differential equation model limit has been known for many years as population-balance equation and is used in particular for growth-fragmentation models. The numerical tests have allowed us to illustrate this convergence: in large population, the integro-differential equation model accurately reflects the behavior of the IBM and thus, the randomness can be neglected. In small or medium population this randomness is not negligible. We also proposed a numerical test where the classical chemostat model, in terms of two coupled ordinary differential equations, cannot account for transient behavior observed with the integro-differential model as with the IBM. Finally, in the case of washout, thus in small population size, IBM gives account for the random washout time, whereas the conventional model as the integro-differential model merely offers a more limited vision of this phenomenon as an asymptotic convergence to washout,  never reached in finite time.

\bigskip

It would be interesting to propose extensive simulations of this model, we can consider several axes. First, our simulations are based on a division rate $\lambda(s,x)$ which  does not depend on the substrate concentration~$s$. One can easily consider a rate depending on~$s$, there are indeed such examples \citep[see e.g.][]{henson2003b}. Similarly, in our model the individual growth dynamics is deterministic, it would be appropriate to consider a stochastic individual growth dynamics. To progress in this direction it would  be necessary to consider specific cases studies where more explicit growth dynamics would be specified.

It would also be pertinent to develop  approximation methods based on moments,
but  these methods  are usually ad hoc and it would be worthwhile to rely on a rigorous and systematic approach.

Finally, there are several relatively immediate extensions of the proposed model.
At first, one can imagine extending the model to the case of several species and several types of substrates.
One can also model the effects of aggregation of bacteria in the chemostat, for example in the form of flocculation with given rates for aggregation and fragmentation of flocs. We can also consider two classes of bacteria, attached bacteria and planktonic bacteria, to account for a biofilm dynamic.

\section*{Acknowledgements}

The authors are grateful to Jérôme Harmand and Claude Lobry for discussions on the model, to Pierre Pudlo and Pascal Neveu for their help concerning the programming of the IBM. This work is partially supported by the project ``Modèles Numériques pour les écosystèmes Microbiens'' of the French National Network of Complex Systems (RNSC call 2012). The work of Coralie Fritsch is partially supported by the Meta-omics of Microbial Ecosystems (MEM) metaprogram of INRA.

\appendix
\section*{Appendices}
\addcontentsline{toc}{section}{Appendices}

\section{Skorohod topology}
\label{appendix.skorohod}

The space of finite measures $\MM_{F}(\X)$ is equipped with the topology of the weak convergence, that is the smallest topology for which the applications  $\zeta\to\crochet{\zeta,f}=\int_{\XX}f(x)\,\zeta(\rmd x)$  are continuous for any  $f\in\C(\X)$. This topology is metrized by the Prokhorov metric:
\begin{align*}
  &\dpr(\zeta,\zeta')
  \eqdef 
  \inf\Bigl\{
     \epsilon >0\,;\,
     \zeta(F)\leq \zeta'(F^\epsilon)+\epsilon\,,\ 
\\[-0.3em]
  &\qquad\qquad\qquad\qquad\qquad\,\,\,
     \zeta'(F)\leq \zeta(F^\epsilon)+\epsilon\,,\ 
     \textrm{ for all closed }F\subset \XX
  \Bigr\}
\end{align*}
where $F^\epsilon\eqdef\{x\in\XX;\inf_{y\in F}|x-y|<\epsilon\}$
\citep[see][Appendix A2.5]{daley2003a}. The Prokhorov distance
is bounded by the distance of the total variation $\dtv(\zeta,\zeta')=\normtv{\zeta-\zeta'}$ associated with the norm defined by:
\begin{align*}
  \normtv{\zeta}
  \eqdef \sup_{A\in\BB(\XX)}|\zeta(A)+\zeta(A^c)| 
  = \zeta_{+}(\XX)+ \zeta_{-}(\XX)
  = \sup_{\substack{f\textrm{ continuous}\\\norm{f}_{\infty}\leq 1}}
  		|\crochet{\zeta,f}|
\end{align*}
for any finite and signed measure $\zeta$ where $\zeta=\zeta_{+}-\zeta_{-}$ is the Hahn-Jordan decomposition of  $\zeta$.

\medskip

The space $\D([0,T],\MM_{F}(\X))$ is equipped with the Skorohod metric $\ds$. Instead of giving the definition of this metric \citep[see][Eq. (5.2) p. 117]{ethier1986a}
we recall a characterization of the convergence for this metric.

\medskip

According to the first characterization \cite[Prop. 5.3, p. 119]{ethier1986a}, a sequence  $(\zeta^{n})_{n\in\N}$ converges to $\zeta$ in $\DD([0,T],\MM_{F}(\X))$, i.e. $\ds(\zeta^{n},\zeta)\to 0$, if and only if there exists a sequence  $\lambda_{n}(t)$ of time change functions such that, for all $n$, 
$t\to \lambda_{n}(t)$ is strictly increasing and continuous with $\lambda_{n}(0)=0$, $\lambda_{n}(t)\to_{t\to\infty}\infty$, satisfying:
\begin{align}
\label{appendix.skorohod.cara.2}
   \sup_{0\leq t\leq T} 
   \dpr({\zeta_t^{n},\zeta_{\lambda_{n}(t)}}) 
     \xrightarrow[n\to\infty]{} 0
\end{align}
and
\begin{align}
\label{appendix.skorohod.cara.3}
	\sup_{0\leq t\leq T}|\lambda_{n}(t)-t|\to 0\,.
\end{align}

\medskip

If $(\zeta^{n})_{n\in\N}$ converges to $\zeta$ in $\DD([0,T],\MM_{F}(\X))$ and if 
$\zeta\in\CC([0,T],\MM_{F}(\XX))$ then in:
\[
  \sup_{0\leq t\leq T}\dpr(\zeta^{n}_{t},\zeta_{t})
  \leq
  \sup_{0\leq t\leq T}\dpr(\zeta^{n}_{t},\zeta_{\lambda_n(t)})
  +
  \sup_{0\leq t\leq T}\dpr(\zeta_{\lambda_n(t)},\zeta_{t})
\]
the first term of the  right-hand side tends to 0 because of \eqref{appendix.skorohod.cara.2};
 the second one tends to 0 because of \eqref{appendix.skorohod.cara.3} and the uniform continuity of $\zeta$ in $[0,T]$. This proves that $\zeta^{n}$ converges to $\zeta$ in $\DD([0,T],\MM_{F}(\X))$ also for the uniform metrics.

%

\section{Numerical integration scheme for the IDE}
\label{appendix.schema.num}

To  numerically solve the system of integer-differential equations \eqref{eq.limite.substrat.fort}-\eqref{eq.limite.eid.fort}, that is the strong version 
of the limit system \eqref{eq.limite.substrat.faible}-\eqref{eq.limite.eid.faible}, 
we make use of finite difference schemes.

Given a time step $\Delta t$ and a mass step  $\Delta x = L/I$, with $I \in \NN^*$,
we discretize the time and mass space with:
\begin{align*}
	t_n & = n \, \Delta t \,
	&
	x_i & = i \, \Delta x\,.
\end{align*}
We introduce the following approximations:
\begin{align*}
	p_{n,i} & \simeq p_{t_n}(x_i) \,,
	&
	s_n & \simeq S_{t_n}\,.
\end{align*}
We also suppose first that at the initial time step there is no individual with null mass in the vessel, i.e. $p_{0,0}=0$; and second that 
individual with null mass cannot be generated during the cell division step,
i.e.  $q$ is regular with $q(0)=0$. This assumption was not necessary in the mathematical development presented in the previous sections but is naturally required to obtain reasonable mass of individuals in the simulation. 

\medskip

For time integration we use an explicit Euler scheme, for space integration,
an uncentered upwind difference scheme, which leads to the coupled integration scheme:
\begin{align*}
	\frac{p_{n+1,i}-p_{n,i}}{\Delta t}
	&  =  
		- \rhog(s_n,x_i)\, \frac{p_{n,i}-p_{n,i-1}}{\Delta x}	
		- \frac{\partial}{\partial x}\rhog(s_n,x_i)\,p_{n,i} 
\\	
	&\qquad\qquad
		- \bigl(\lambda(s_n, x_i)+D \bigr)\,p_{n, i} 
	 	+ 2\, \Delta x \, \sum_{j=1}^{I}
	 		 				\frac{\lambda(s_n, x_j)}{x_j}\,
	 		 				q\left(\frac{x_i}{x_j} \right)\,
	 		 				p_{n,j} \,,
\\
	\frac{s_{n+1}-s_n}{\Delta t}
	& = 
	D\,(\Sin - s_n)
	- \frac{k}{V} \, \Delta x \, \sum_{j=1}^{I} \rhog(s_n, x_j) \, p_{n,j}
\end{align*}
for $n \in \NN$ and $i = 1, \cdots I$, with the boundary condition:
\begin{align*}
	p_{n+1,0} = 0
\end{align*}
and given initial conditions $p_{0,i}$ and $s_{0}$.

We finally get:
\begin{align*}
	p_{n+1,i}
	& =   
 	p_{n,i} + \Delta t \;
	\Biggl\{
		- \rhog(s_n,x_i)\, \frac{p_{n,i}-p_{n,i-1}}{\Delta x}	
		- \frac{\partial}{\partial x}\rhog(s_n,x_i)\,p_{n,i} 
\\	
	& \qquad\qquad
		- \bigl(\lambda(s_n, x_i)+D \bigr) \, p_{n, i} 
	 	+ 2\, \Delta x \, 
			\sum_{j=1}^{I} \frac{\lambda(s_n, x_j)}{x_j}\,
	 		 			q\left(\frac{x_i}{x_j} \right)\, p_{n,j} 
		\Biggr\}
\\
	s_{n+1}
	& = 
	s_n + \Delta t \,  
	\Biggl\{
		D\,(\Sin - s_n) 
		- \frac{k}{V} \, \Delta x \, \sum_{j=1}^{I} \rhog(s_n, x_j) \, p_{n,j}
	\Biggl\}
\end{align*}
for  $n \in \NN$ and $i = 1, \cdots I$ with boundary condition  $p_{n+1,0} = 0$ and given initial conditions $p_{0,i}$ and $s_{0}$.


\begin{thebibliography}{}

\bibitem[Allen, 2003]{allen2003a}
Allen, L.~J. (2003).
\newblock {\em An Introduction to Stochastic Processes with Biology
  Applications}.
\newblock Prentice Hall.

\bibitem[Billingsley, 1968]{billingsley1968a}
Billingsley, P. (1968).
\newblock {\em Convergence of Probability Measures}.
\newblock John Wiley \& Sons, New York.


\bibitem[Campillo et~al., 2011]{campillo2011chemostat}
Campillo, F., Joannides, M., and Larramendy-Valverde, I. (2011).
\newblock Stochastic modeling of the chemostat.
\newblock {\em Ecological Modelling}, 222(15):2676--2689.

\bibitem[Champagnat, 2006]{champagnat2006b}
Champagnat, N. (2006).
\newblock A microscopic interpretation for adaptive dynamics trait substitution
  sequence models.
\newblock {\em Stochastic Processes and their Applications}, 116(8):1127--1160.

\bibitem[Champagnat et~al., 2013]{champagnat2013a}
Champagnat, N., Jabin, P.-E., and M{\'e}l{\'e}ard, S. (2013).
\newblock Adaptation in a stochastic multi-resources chemostat model.
\newblock {\em ArXiv Mathematics e-prints}, arXiv:1302.0552v1 [math.PR].

\bibitem[Crump and O'Young, 1979]{crump1979a}
Crump, K.~S. and O'Young, W.-S.~C. (1979).
\newblock Some stochastic features of bacterial constant growth apparatus.
\newblock {\em Bulletin of Mathematical Biology}, 41(1):53 -- 66.

\bibitem[Daley and Vere-Jones, 2003]{daley2003a}
D.~Daley and D.~Vere-Jones.
\newblock {\em An Introduction to the Theory of Point Processes -- Volume I:
  Elementary Theory and Methods}.
\newblock Springer, second edition, 2003.

\bibitem[Daoutidis and Henson, 2002]{daoutidis2002a}
Daoutidis, P. and Henson, M. (2002).
\newblock Dynamics and control of cell populations in continuous bioreactors.
\newblock {\em {AIChE} Symposium Series}, 326:274--289.

\bibitem[Diekmann et~al., 2005]{diekmann-odo2005a}
Diekmann, O., Jabin, P.-E., Mischler, S., and Perthame, B. (2005).
\newblock The dynamics of adaptation: an illuminating example and a
  {Hamilton-Jacobi} approach.
\newblock {\em Theoretical Population Biology}, 67(4):257--71.

\bibitem[Ethier and Kurtz, 1986]{ethier1986a}
Ethier, S.~N. and Kurtz, T.~G. (1986).
\newblock {\em Markov Processes -- Characterization and Convergence}.
\newblock John Wiley \& Sons.

\bibitem[Fournier and M{\'e}l{\'e}ard, 2004]{fournier2004a}
Fournier, N. and M{\'e}l{\'e}ard, S. (2004).
\newblock A microscopic probabilistic description of a locally regulated
  population and macroscopic approximations.
\newblock {\em Annals of Applied Probability}, 14(4):1880--1919.

\bibitem[Fredrickson et~al., 1967]{fredrickson1967a}
Fredrickson, A., Ramkrishna, D., and Tsuchiya, H. (1967).
\newblock Statistics and dynamics of procaryotic cell populations.
\newblock {\em Mathematical Biosciences}, 1(3):327--374.




\bibitem[Grasman et~al., 2005]{grasman2005a}
Grasman, J., {De Gee}, M., and Herwaarden, O. A.~V. (2005).
\newblock Breakdown of a chemostat exposed to stochastic noise volume.
\newblock {\em Journal of Engineering Mathematics}, 53(3):291--300.

\bibitem[Hatzis et~al., 1995]{hatzis1995a}
Hatzis, C., Srienc, F., and Fredrickson, A.~G. (1995).
\newblock Multistaged corpuscular models of microbial growth: {Monte Carlo}
  simulations.
\newblock {\em Biosystems}, 36(1):19--35.

\bibitem[Henson, 2003]{henson2003b}
Henson, M.~A. (2003).
\newblock Dynamic modeling and control of yeast cell populations in continuous
  biochemical reactors.
\newblock {\em Computers \&\ Chemical Engineering}, 27(8-9):1185--1199.

\bibitem[Hoskisson and Hobbs, 2005]{hoskisson2005a}
Hoskisson, P.~A. and Hobbs, G. (2005).
\newblock {Continuous culture - making a comeback?}
\newblock {\em Microbiology}, 151(10):3153--3159.

\bibitem[Ikeda and Watanabe, 1981]{ikeda1981a}
Ikeda, N. and Watanabe, S. (1981).
\newblock {\em Stochastic Differential Equations and Diffusion Processes}.
\newblock North--Holland/Kodansha.

\bibitem[Imhof and Walcher, 2005]{imhof2005a}
Imhof, L. and Walcher, S. (2005).
\newblock Exclusion and persistence in deterministic and stochastic chemostat
  models.
\newblock {\em Journal of Differential Equations}, 217(1):26--53.


\bibitem[Joffe and M\'etivier, 1986]{joffe1986a}
Joffe, A. and M\'etivier, M. (1986).
\newblock Weak convergence of sequences of semimartingales with applications to
  multitype branching processes.
\newblock {\em Advances in Applied Probability}, 18:20--65.

\bibitem[Lee et~al., 2009]{lee-min-woo2009a}
Lee, M.~W., Vassiliadis, V.~S., and Park, J.~M. (2009).
\newblock Individual-based and stochastic modeling of cell population dynamics
  considering substrate dependency.
\newblock {\em Biotechnology and Bioengineering}, 103(5):891--899.

\bibitem[M{\'e}l{\'e}ard and Roelly, 1993]{meleard1993a}
M{\'e}l{\'e}ard, S. and Roelly, S. (1993).
\newblock Sur les convergences \'etroite ou vague de processus \`a valeurs
  mesures.
\newblock {\em Comptes rendus de l'Acad{\'e}mie des Sciences de Paris S{\'e}r.
  1}, 317:785--788.

\bibitem[Mirrahimi et~al., 2012a]{mirrahimi2012a}
Mirrahimi, S., Perthame, B., and Wakano, J. (2012a).
\newblock Evolution of species trait through resource competition.
\newblock {\em Journal of Mathematical Biology}, 64(7):1189--1223.

\bibitem[Mirrahimi et~al., 2012b]{mirrahimi2012b}
Mirrahimi, S., Perthame, B., and Wakano, J.~Y. (2012b).
\newblock Direct competition results from strong competiton for limited
  resource.
\newblock {\em arXiv:1201.5826v1}.

\bibitem[Monod, 1950]{monod1950a}
Monod, J. (1950).
\newblock La technique de culture continue, th{\'e}orie et applications.
\newblock {\em Annales de l'Institut Pasteur}, 79(4):390--410.

\bibitem[Novick and Szilard, 1950]{novick1950a}
Novick, A. and Szilard, L. (1950).
\newblock Description of the chemostat.
\newblock {\em Science}, 112(2920):715--716.

\bibitem[Ramkrishna, 1979]{ramkrishna1979a}
Ramkrishna, D. (1979).
\newblock Statistical models of cell populations.
\newblock In {\em Advances in Biochemical Engineering}, volume~11, pages 1--47.
  Springer Berlin Heidelberg.

\bibitem[Ramkrishna, 2000]{ramkrishna2000a}
Ramkrishna, D. (2000).
\newblock {\em Population Balances: Theory and Applications to Particulate
  Systems in Engineering}.
\newblock Elsevier Science.

\bibitem[Roelly-Coppoletta, 1986]{roelly1986a}
Roelly-Coppoletta, S. (1986).
\newblock A criterion of convergence of measure-valued processes: application
  to measure branching processes.
\newblock {\em Stochastics}, 17(1-2):43--65.

\bibitem[R{\"u}diger and Ziglio, 2006]{rudiger2006a}
R{\"u}diger, B. and Ziglio, G. (2006).
\newblock It{\^o} formula for stochastic integrals w.r.t. compensated {Poisson}
  random measures on separable {Banach} spaces.
\newblock {\em Stochastics An International Journal of Probability and
  Stochastic Processes}, 78(6):377--410.

\bibitem[Smith and Waltman, 1995]{smith1995a}
Smith, H.~L. and Waltman, P.~E. (1995).
\newblock {\em The Theory of the Chemostat: Dynamics of Microbial Competition}.
\newblock {Cambridge University Press}.

\bibitem[Stephanopoulos et~al., 1979]{stephanopoulos1979a}
Stephanopoulos, G., Aris, R., and Fredrickson, A. (1979).
\newblock A stochastic analysis of the growth of competing microbial
  populations in a continuous biochemical reactor.
\newblock {\em Mathematical Biosciences}, 45:99--135.

\bibitem[Tran, 2006]{vietchitran2006a}
Tran, V.~C. (2006).
\newblock {\em Mod{\`e}les particulaires stochastiques pour des probl{\`e}mes
  d'{\'e}volution adaptative et pour l'approximation de solutions
  statistiques}.
\newblock PhD thesis, Universit{\'e} de Nanterre - Paris X.

\bibitem[Tran, 2008]{vietchitran2008a}
Tran, V.~C. (2008).
\newblock Large population limit and time behaviour of a stochastic particle
  model describing an age-structured population.
\newblock {\em ESAIM: PS}, 12:345--386.

\end{thebibliography}

\end{document}